\documentclass{amsart}
\usepackage[colorlinks=true,linkcolor=blue,citecolor=green!65!black]{hyperref}
\usepackage{cite}
\usepackage{mathtools,bm,comment,amssymb,enumitem,xcolor,array,float,subcaption,float,tikz,pgfplots,siunitx,booktabs,accents}

\usepackage[capitalise,nameinlink]{cleveref}
\usepackage[english]{babel}
\usepackage[margin=1in]{geometry}
\usepackage[foot]{amsaddr}
\usepackage[linesnumbered,ruled]{algorithm2e}
\usepackage[noend]{algpseudocode}

\usepackage[title]{appendix}

\newcommand{\bbone}{\text{\usefont{U}{bbold}{m}{n}1}}
\MakeRobust{\bbone}

\numberwithin{equation}{section}
\crefname{subsection}{Subsection}{Subsections}
\crefformat{equation}{(#2#1#3)}
\crefformat{align}{(#2#1#3)}
\crefformat{enumi}{(#2#1#3)}
\crefname{figure}{Figure}{Figures}

\renewcommand{\tilde}[1]{\accentset{\sim}{#1}}

\newcommand{\weak}{\rightharpoonup}
\newcommand{\longweak}{\relbar\joinrel\rightharpoonup}
\newcommand{\eps}{\varepsilon}
\newcommand{\esssup}{\textup{ess\,sup}}
\newcommand{\ov}{\overline}
\newcommand{\dd}{\,\textup{d}}
\newcommand{\loc}{{\textup{loc}}}

\newcommand{\m}{{\textup{m}}}

\newcommand{\OT}{{\Omega_T}}

\newcommand{\dt}{\,\textup{d}t}
\newcommand{\dtx}{\,\textup{d}(t,\!x)}
\newcommand{\dx}{\,\textup{d}x}
\newcommand{\ds}{\,\textup{d}s}

\newcommand{\ddt}{\frac{\dd}{\dd t}}

\newcommand{\p}{\partial}

\newcommand{\pta}{\p_t^\alpha}

\renewcommand{\div}{\textup{div}}
\newcommand{\R}{\mathbb{R}}
\newcommand{\N}{\mathbb{N}}

\renewcommand{\H}{\mathcal{H}}
\newcommand{\E}{\mathcal{E}}

\newcommand{\rel}{{\textup{rel}}}

\newcommand{\con}{\hookrightarrow}
\newcommand{\com}{\mathrel{\mathrlap{{\mspace{4mu}\lhook}}{\hookrightarrow}}}

\newtheorem{remark}{Remark}
\newtheorem{assumption}{Assumption}
\newtheorem{lemma}{Lemma}
\newtheorem{theorem}{Theorem}

\newtheorem{proposition}{Proposition}
\newtheorem{corollary}{Corollary}

\pgfplotsset{compat=1.15}

\allowdisplaybreaks

\makeatletter
\@namedef{subjclassname@2020}{\textup{2020} Mathematics Subject Classification}
\makeatother

\let\originalleft\left
\let\originalright\right
\renewcommand{\left}{\mathopen{}\mathclose\bgroup\originalleft}
\renewcommand{\right}{\aftergroup\egroup\originalright}

\newcolumntype{C}{>{\centering\arraybackslash}m{10em}}

\newcommand{\ifs}{{\text{if }}}

\makeatletter
\renewcommand{\email}[2][]{%
	\ifx\emails\@empty\relax\else{\g@addto@macro\emails{,\space}}\fi%
	\@ifnotempty{#1}{\g@addto@macro\emails{\textrm{(#1)}\space}}%
	\g@addto@macro\emails{#2}%
}
\makeatother
\title[Time-Fractional Cahn--Hilliard Equation]{Time-Fractional Cahn--Hilliard Equation: \\ Well-Posedness, Degeneracy, and Numerical Solutions}
\author{Marvin Fritz${}^*$, Mabel L. Rajendran, Barbara Wohlmuth}
\address{Department of Mathematics, Technical University of Munich.}
\email{\{fritzm, rajendrm, wohlmuth\}@ma.tum.de}
\thanks{${}^*$Corresponding author}

\begin{document} 
\maketitle
	\vspace{-1cm}
\begin{abstract} 
 
 In this paper, we derive the time-fractional Cahn--Hilliard equation from continuum mixture theory with a modification of Fick's law of diffusion. This model describes the process of phase separation with nonlocal memory effects. We analyze the existence, uniqueness, and regularity of weak solutions of the time-fractional Cahn--Hilliard equation. In this regard, we consider degenerating mobility functions and free energies of Landau, Flory--Huggins and double-obstacle type. We apply the Faedo--Galerkin method to the system, derive energy estimates, and use compactness theorems to pass to the limit in the discrete form. In order to compensate for the missing chain rule of fractional derivatives, we prove a fractional chain inequality for semiconvex functions. The work concludes with numerical simulations and a sensitivity analysis showing the influence of the fractional power. Here, we consider a convolution quadrature scheme for the time-fractional component, and use a mixed finite element method for the space discretization.
\end{abstract}

\section{Introduction}

Phase-field models, such as the Cahn--Hilliard \cite{cahn1958free} and Allen--Cahn equations \cite{allen1972ground}, have numerous applications in real world scenarios, e.g., material sciences \cite{choksi2009phase}, cell biology \cite{garcke2018multiphase}, and image processing \cite{bertozzi2006inpainting,burger2009cahn}. More recently, phase-field models with nonlocal effects have been considered, which are applied to scenarios where long-range interactions are of interest, like in the adhesion properties of cells, e.g., see \cite{fritz2019local}. In general, nonlocal interactions are expressed by integral operators, i.e., integrodifferential equations are investigated.

Nonlocal effect occur naturally in time-fractional PDEs which have numerous applications due to their innate memory effect, e.g., in viscoelasticity \cite{bagley1983theoretical,bagley1986fractional,mainardi2010fractional} and -plasticity \cite{diethelm1999solution}, in image \cite{bai2007fractional,cuesta2011some} and signal processing \cite{marks1981differintegral}, and in the mechanical properties of materials \cite{torvik1984appearance}.  In particular, the time-fractional Cahn--Hilliard equation has already been studied by various authors, but the analysis of weak solutions is still open. Exact solutions have been studied in \cite{guner2015variety} 
and numerical simulations were shown and compared to the time-fractional Allen--Cahn equation, see \cite{liu2018time, prakasha2019two,akinyemi2020iterative,zhang2020non,khristenko2021solving}. In \cite{tang2019energy}, the energy dissipation of the time-fractional Cahn--Hilliard equation was studied numerically and the power law scaling was investigated in \cite{zhao2019power}. In particular, it was shown that the Ginzburg--Landau energy follows a power law $\E \sim t^{-\alpha/3}$ in which the power is linearly proportional to the fractional order. In this context, the authors of \cite{bru2008fractal} have shown that the interface between tumor and host follows a similar interface. This suggests that the time-fractional Cahn--Hilliard equation is suitable to describe tumor growth processes similar to the integer order case, e.g., see \cite{hawkins2012numerical}. 

The time-fractional component of the Cahn--Hilliard equation is included in the space-time fractional Cahn--Hilliard equation, where the Laplace operator is replaced by its fractional counterpart. The exact solution of these equations have been theoretically investigated in
\cite{liu2020fast}.

The Allen--Cahn equation is similar to the Cahn--Hilliard equations in the sense that both are phase-field equations and describe the process of phase separation; particularly, their applications overlap. The time-fractional Allen--Cahn equation has been studied in
\cite{li2017space,ji2020adaptive,inc2018time}, 
both analytically and numerically. The authors in \cite{liu2018time} remark that the Allen--Cahn equation is more straightforward to study because of its similarities to the heat equation and the availability of the weak comparison principle. The uniqueness of the Allen--Cahn equation is also well understood because of its similar structure to the heat equation, whereas it is open for the Cahn--Hilliard equation with degenerated mobility. This results from the structure of fourth order, since even simple looking PDEs of fourth order can admit more than one solution, e.g., $u_t=(u^m u_{xxx})_x$, see \cite{elliott1996cahn}.

The time-fractional Cahn--Hilliard and the time-fractional Allen--Cahn equations were compared in \cite{liu2018time}, and it was concluded that the time-fractional Cahn--Hilliard equation ``is more consistent with practice'' than the Allen--Cahn equation. The consistency results in the volume preservation of the Cahn--Hilliard equation, which is not present in the Allen--Cahn equation. One can also add a Lagrange multiplier to the Allen--Cahn equation so that it imitates the Cahn--Hilliard equation in its conservative nature, see \cite{bates2013global,rubinstein1992nonlocal}.

We mention that some variants of the Cahn--Hilliard equation with memory effects have been studied thoroughly, e.g.,
\cite{dafermos1970asymptotic,rotstein2001phase,galenko2005diffuse,galenko2009kinetic} for modeling via hyperbolic relaxation, 
\cite{pruss2010well,conti2010attractors,vergara2007conserved,vergara2007maximal,novick2002phase,gatti2005memory}
with regard to analysis, and \cite{lecoq2009evolution,lecoq2011numerical} with regard to numerical properties.

In \cref{Sec:Modeling}, we derive the time-fractional Cahn--Hilliard equation from continuum mixture theory with a modified Fick law. In this context, we introduce the Caputo fractional derivative operator.
In \cref{Sec:Preliminaries}, we shortly mention some analytical preliminaries which we will need in the upcoming sections, e.g., the definition of fractional Sobolev--Bochner spaces, and corresponding compactness results. We prove a fractional chain inequality for semi-convex functions, which serves as an alternative to the chain rule for integer-order derivatives.
In \cref{Sec:AnalysisCahn}, we study the time-fractional Cahn--Hilliard equation. We first investigate the case of a positive and bounded mobility function $m$, and the Landau free energy; we show the existence of weak solutions and a corresponding energy inequality via a Faedo--Galerkin approach for time-fractional PDEs. Moreover, for $m$ being constant we prove uniqueness and continuous dependence on the data. We derive higher regularity results. Finally, we investigate the case of degenerated mobilities by approximating the mobility $m$ with a positive function $m_\delta$, deriving uniform $\delta$ estimates and passing to the limit $\delta \to 0$. In this regard, we allow potentials of logarithmic and double-obstacle type.
In \cref{Sec:Numerics}, we show some numerical simulation results of the time-fractional Cahn--Hilliard equation  in a two- and three-dimensional domain. We use the Gr\"unwald--Letnikov approximation formula and mixed finite element spaces. We compare different values of $\alpha$ in the process of block copolymers.  Moreover, we study the influence of the parameter $\alpha$ in subdiffusive tumor growth models.
\section{Modeling of the Time-Fractional Cahn--Hilliard Equation} \label{Sec:Modeling}

In this section, we derive the Cahn--Hilliard equation from continuum mixture theory. Further, we apply a modification of Fick's law of diffusion which results in a time-fractional derivative in the system. Here, we use the fractional derivatives in the sense of Riemann--Liouville and Caputo.

\subsection{Classical theory} 

Let $\phi_1, \phi_2$ be the concentrations of two components with the relation $\phi_1+\phi_2=1$, i.e., they describe local portions, e.g., in binary alloys. They have to satisfy the law of conservation of mass (setting the mass density to $\varrho=1$)
	\begin{equation*}
	\p_t \phi_i = - \div J_i, \quad i \in \{1,2\},
	\end{equation*}
	where $J_i$ denotes the mass flux of the $i$-th component. In order to guarantee $\p_t (\phi_1+\phi_2)=0$, we require the fluxes to satisfy $J_1+J_2=0$. We reduce the equations by setting $\phi=\phi_1-\phi_2$ and $J=J_1-J_2$, yielding
	\begin{equation} \label{Eq:MassCons}  \p_t \phi = - \div J.
	\end{equation}
	
    One can assume that the flux $J$ is given by the negative of the gradient of the chemical potential $\mu$, i.e., $J=-\nabla \mu$. Gurtin \cite{gurtin1996generalized} proposed a mechanical version of the second law of thermodynamics
by introducing a new mass flux with the mobility function $m$ for interactions at a microscopic level given by 
	\begin{equation}
	    \label{Eq:Flux_classical}
	   J=-m(\phi)\nabla\mu. 
	\end{equation}
	Following \cite{cahn1958free}, the chemical potential is defined as the first variation (G\^ateaux derivative) of the Ginzburg--Landau free energy functional
	\begin{equation}
	     \label{Eq:Ginzburg}
	\E(\phi) = \int_\Omega \Psi(\phi) + \frac{\eps^2}{2} |\nabla \phi|^2 \, \dd x,\end{equation}
	i.e., $\mu = \delta\E(\phi)$. A simple calculation yields the so-called Cahn--Hilliard equation,
	\begin{equation} \label{Eq:CahnHilliard} \begin{aligned}
	\p_t \phi &= \div(m(\phi) \nabla \mu), \\  \mu &= \Psi'(\phi) - \eps^2 \Delta \phi.\end{aligned}
	\end{equation}
	Here, the parameter $\eps$ denotes the interface width, and $\Psi$ describes a double-well potential with zeros at $-1$ and $1$, e.g., the Landau potential,
	\begin{equation} \label{Eq:TypicalChoice}
	\Psi(\phi)=\frac14 (1-\phi^2)^2, 
	\end{equation}
	but also logarithmic approximations are possible, see \cite{cherfils2011cahn}, such as the Flory--Huggins logarithmic potential for $\phi \in (-1,1)$,
	\begin{equation}\label{Eq:Flory}	\Psi(\phi)=\frac{\theta}{2}( (1+\phi) \ln(1+\phi) + (1-\phi) \ln(1-\phi)) + \frac{\theta_0}{2}(1- \phi^2),\end{equation}
	where $\theta,\theta_0>\theta$ are given constants. Lastly, we mention potentials of double-obstacle type, which are formally obtained by setting $\theta=0$  in the Flory--Huggins potential, giving
	$$\Psi(\phi)=\begin{cases} \frac{\theta_0}{2}(1-\phi^2), &\phi \in [-1,1], \\ \infty, &\text{else}. \end{cases} $$
	Here, the derivative $\Psi'(\phi)$ has to be interpreted in the sense of subdifferentials, and the Cahn--Hilliard equation becomes a system of differential inclusions.
	
	Typically, the mobility function is of the form $m(\phi)=M|1-\phi^2|^\nu$ for a constant $M$ and some given power $\nu \geq 0$. Mostly, the cases $\nu \in \{0,1,2\}$ are treated in the literature, e.g., see \cite{taylor1994linking,hilliard1970spinodal}. The case $\nu=0$ (i.e. $m=M$) is well-explored and well-posedness can be shown under sufficient assumptions, e.g., see \cite{temam2012infinite}. In the case of a degenerate mobility, a proof or counterexample to uniqueness is still an open problem; this is unsolved for the class of fourth-order degenerate parabolic equations. For the proof of existence, we refer to \cite{elliott1996cahn}, and to \cite{dai2016weak} for weaker assumptions on the degenerating mobility.

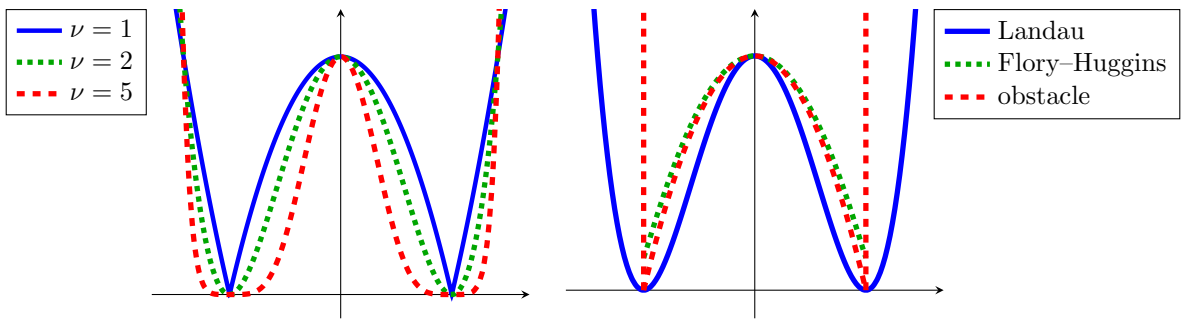
\begin{figure}[H]
\centering
\begin{tikzpicture}
\begin{axis}[width=.4\textwidth,
    legend style={at={(-0.2,1)},anchor=north}, axis lines=middle, samples=70, smooth,
	xtick=\empty, ytick=\empty, ymin=-.1, ymax=1.2,xmin=-1.7,xmax=1.7,legend cell align={left},legend style={fill=white, fill opacity=0.6, draw opacity=1,text opacity=1}]
\addplot[blue,domain=-1.7:1.7,line width=1.7pt,samples=5000] {abs(1-x^2)^(1)};
\addlegendentry{$\nu=1$}
\addplot[green!65!black,dotted,domain=-1.5:1.5,line width=2pt] {(1-x^2)^2};
\addlegendentry{$\nu=2$}
\addplot[red,dashed,domain=-1.5:1.5,line width=2pt] {abs(1-x^2)^5};
\addlegendentry{$\nu=5$}
\end{axis} 
\end{tikzpicture} \quad
\begin{tikzpicture}
\begin{axis}[
  width=.4\textwidth,legend style={at={(1.3,1)},anchor=north},axis lines=middle, samples=70, smooth,
  xtick=\empty, ytick=\empty, ymin=-.03, ymax=0.3,xmin=-1.7,xmax=1.7, legend cell align={left}, legend style={fill=white, fill opacity=0.6, draw opacity=1,text opacity=1}]
\addplot[blue,domain=-1.5:1.5,line width=2pt] {0.25*(1-x^2)^2};
\addlegendentry{Landau}
\addplot[green!65!black,dotted,domain=-1:1,line width=2pt] {((1+x)*ln(abs(1+x))+(1-x)*ln(abs(1-x)))/40+0.25*(1-x^2)};
\addlegendentry{Flory--Huggins}
\addplot[red,dashed,domain=-1:1,line width=2pt] {0.25*(1-x^2)};
\addlegendentry{obstacle}
\addplot[red,line width=2pt, dashed] coordinates {(1,0)(1,2)};
\addplot[red,line width=2pt, dashed] coordinates {(-1,0)(-1,2)};
\end{axis}
\end{tikzpicture}
\caption{\label{Fig:Function} Depiction of the functions $m(x)=|1-x^2|^\nu$, $\nu \in \{1,2,5\}$ (left), and the potentials of Landau, Flory--Huggins, and double obstacle type (right). Here, we have chosen the parameters $\theta=\tfrac{1}{40}$ and $\theta_0=\tfrac14$.}
\end{figure}
	
\subsection{Time-fractional component} 

The phenomenological law given in \eqref{Eq:Flux_classical} represents the simplest relation between the flux and the gradient of the chemical potential. One could replace this law by a more complex phenomenological relationship, which may account for possible nonlocal, nonlinear and memory effects, without violating the conservation law expressed by the continuity equation.

Long term memory effects are incorporated by introducing convolution or Riemann--Liouville fractional derivative \cite{diethelm2010analysis} in the relationship between mass flux and gradient of concentration \cite{gorenflo2002time,povstenko2017two} replacing the classical Fick's law \cite{jost2014mathematical,kruvzik2019mathematical}; in the relationship between heat flux and the gradient of temperature \cite{coleman1967equipresence,povstenko2015fractional} replacing the Fourier's law \cite{eck2017mathematical,kruvzik2019mathematical}; and in the relationship between mass flux and the gradient of pressure \cite{caputo1999diffusion,zacher2012global,jakubowski2001nonlinear} replacing the Darcy's law. We refer to the recent book \cite{povstenko2015fractional} for a description of the time-fractional component in thermoelasticity from a modeling perspective.

The memory effects in phase changes are taken into account by including relaxation in the chemical potential. The presence of a slowly relaxing structure for kinetics of phase separation was observed in \cite{binder1986kinetics} and a phenomenological theory for relaxing the mass flux $J$ is given by
$$\tilde J_\rel = -\int_{-\infty}^t k(t-s) m(\phi(s)) \nabla \delta \E(\phi(s)) \, \dd s, 
$$
where $\E$ is again the Ginzburg--Landau energy \cref{Eq:Ginzburg}, and $k\in L_{1,\loc}(\R_{\geq 0})$ denotes a positive nondecreasing relaxation kernel. Choosing $0$ instead of $-\infty$ as the starting point, we obtain
\begin{equation*}
J_\rel = - \int_0^t k(t-s) m(\phi(s)) \nabla \delta \E(\phi(s)) \, \dd s.
\end{equation*}
We assume that the mass flux $J$ contains only a relaxing flux $J_\rel$. 

Notice that we obtain the classical Cahn--Hilliard equation \eqref{Eq:CahnHilliard} for instantaneous memory, i.e., choosing the relaxation kernel to be of Dirac delta type, i.e., $k(t)=\delta(t)$. In the case of a sufficiently well-behaved kernel of the form $k(t)=e^{-\beta t},\,\beta>0$, one obtains a hyperbolic partial differential equation of the form
$$
\begin{aligned}
    \p_{tt}\phi + \beta \p_t\phi = \div(m(\phi) \nabla \mu),\\
    \mu = \Psi'(\phi) - \eps^2 \Delta \phi,
\end{aligned}$$
which is called hyperbolic model for spinodal decomposition, e.g., see \cite{rotstein2001phase,galenko2007analysis,galenko2009kinetic}. 
In the case of a power law fading kernel of the form $k(t)= \frac{t^{\alpha-1}}{\Gamma(\alpha)}e^{-\beta t}$ with $\beta \geq 0$, $\alpha>0$, one obtains both fast and slow relaxation for $\alpha<1$. Here, $\Gamma :(0,\infty) \to \R$, $t \mapsto \int_0^\infty s^{t-1} e^{-s} \, \dd s$ denotes Euler's Gamma function. The fast relaxation near $t=0^+$ corresponds to an instantaneous contribution of the concentration history. 

In the following, we consider the time nonlocal relation between the mass flux and the gradient of the concentration in the form
\begin{equation} \label{Eq:RelFractional} 
    \begin{aligned}
    J_\rel 
    =- \p_t\int_{0}^t g_{\alpha}(t-s) m(\phi(s))  \nabla \delta \E(\phi (s))\ds = - \p_t \left(g_\alpha * \left[m(\phi)\nabla\delta  \E(\phi)\right]\right)(t), 
    \end{aligned}
\end{equation}
where the kernel $g_\alpha$ is defined by
$$g_{\alpha}(t):=\begin{cases}t^{\alpha-1}/\Gamma(\alpha),&\alpha>0,\\\delta(t),&\alpha=0,\end{cases}$$ for $t>0$, see also \cite{diethelm2010analysis,kilbas2006theory}. Here, the operator $*$ denotes the convolution on the positive half-line with respect to the time variable, i.e., $(g_\alpha*\varphi)(t)=\int_0^t g_\alpha(t-s)\varphi(s)\ds$ for some function $\varphi\in L^1([0,T))$. This convolution is known as the Riemann--Liouville fractional integral and is denoted by $I^{\alpha}$, see \cite{diethelm2010analysis,kilbas2006theory}. 

We introduce the Riemann--Liouville fractional derivative \cite{diethelm2010analysis} of order $\alpha$ as
\begin{equation*}\label{Def:frac_Riemann}
\begin{aligned}
D^{\alpha}_t \varphi(t)& = \p_t (g_{1-\alpha}*\varphi)(t),
\end{aligned}
\end{equation*}
a.e. on $(0,T)$, for some function $\varphi \in L^1([0,T))$ with $g_{1-\alpha}*\varphi\in W^{1,1}(0,T)$. In the global case $T=\infty$, we refer to \cite[Section 2.2]{kilbas2006theory}.
We use this definition in \cref{Eq:RelFractional} to express the relaxed mass flux in terms of the fractional derivative
$$J_\rel = D_t^{1-\alpha} \left(m(\phi)\nabla \delta \E(\phi) \right).
$$

Inserting the relaxed mass flux $J_\rel$ \cref{Eq:RelFractional} into the law of conservation of mass \cref{Eq:MassCons} yields
\begin{equation*}
    \begin{aligned}
    \p_t \phi &= -\div J_\rel =  \div(D_t^{1-\alpha} (m(\phi)\nabla\mu)) \\ 
        \mu &= \Psi'(\phi)-\eps^2 \Delta \phi.
    \end{aligned}
\end{equation*}
We equivalently rewrite this system by taking the convolution with $g_{1-\alpha}$ on both sides of the first equation, which gives the time-fractional Cahn--Hilliard equation in terms of the classical Caputo fractional derivative \cite{caputo1967linear}
 \begin{equation*}  \label{Eq:CHCaputo}
    \begin{aligned}
    \p_t^\alpha \phi &= \div(m(\phi)\nabla\mu),\\
    \mu &= \Psi'(\phi) - \eps^2 \Delta \phi.
    \end{aligned}
\end{equation*}
Here we used the semigroup property of the kernel in the following way
\begin{equation}
    (g_{1-\alpha} * \p_t(g_\alpha * \varphi))(t)=\p_t ((g_{1-\alpha} * g_\alpha) * \varphi)(t) - g_{1-\alpha}(t) (g_{\alpha}*\varphi)(0)  =\p_t (1*\varphi) = \varphi,
\label{Eq:InverseConvolution}
\end{equation}
provided $\varphi$ is sufficiently smooth such that $(g_{\alpha}*\varphi)(0)=0$; for instant, it holds for $\varphi\in L^{\infty}(0,T)$. The classical/well-known form of the Caputo fractional derivative is defined almost everywhere for an absolutely continuous function $\varphi \in AC([0,T))$, see \cite[Theorem 2.1, Equation 2.4.47]{kilbas2006theory}, as
\begin{equation*}\label{Def:frac_Caputo}
\begin{aligned}
\p_t^\alpha \varphi= g_{1-\alpha}*\p_t\varphi(t).
\end{aligned}
\end{equation*}

\begin{remark}
Note that the Caputo derivative requires a function which is absolutely continuous, i.e., its $\alpha=1$ derivative exists almost everywhere. 
This definition can be generalized to a larger class of functions by applying the Riemann--Liouville fractional derivative on $\varphi(t)-\varphi_0$. In fact,
\begin{equation}\label{Def:frac_Modified_Riemann}
\begin{aligned}
\p_t^\alpha \varphi = D^{\alpha}_t(\varphi(t)-\varphi_0),
\end{aligned}
\end{equation}
whenever $\varphi\in L^1([0,T))$ with $g_{1-\alpha}*(\varphi-\varphi_0)\in W^{1,1}_{0}(0,T)$. Here, $\varphi_0$ plays the role of the initial value in the sense that $g_{1-\alpha}*(\varphi-\varphi_0)(0)=0$.
This definition coincides with the classical definition of the Caputo derivative if $\varphi\in AC([0,T))$, see \cite[Theorem 2.1]{kilbas2006theory}. 

Further, the Riemann--Liouville fractional calculus can be modified by including some singularity at $t=0$ so that the group structure holds, see \cite[Definition 2.14]{li2018generalized}. One defines
$J_{\alpha}\varphi= g_{\alpha}*(\H\varphi)$ and $g_{\alpha}(t):= (\H(t)t)^{\alpha-1}/\Gamma(\alpha)$, $\alpha\geq 0$,
where $\H$ is the Heaviside function and the convolution is understood as in \cite[Definition 2.12]{li2018generalized}. 
Then one can relax the classical Caputo definition for functions $\varphi \in L^1_\loc([0,T))$ having $t=0$ as a Lebesgue point from the right and with value $\varphi_0$ at $t=0$ (i.e., $\lim_{t\rightarrow 0+} \frac{1}{t} \int_0^t|\varphi(t)-\varphi_0|\dd t =0$), as
\begin{equation}\label{Def:frac_Liu_Caputo}
    \partial_t^\alpha\varphi = J_{-\alpha}*(\varphi-\varphi_0),
\end{equation}
see \cite[Definition 3.4]{li2018generalized}. Note that $ \partial_t^\alpha\varphi$ is generally a distribution, however 
if it holds that $\varphi\in AC([0,T))$, then $\partial_t^\alpha\varphi$ is the traditional Caputo derivative \cite[Proposition 3.6]{li2018generalized}.
\end{remark}

Consequently, we state the time-fractional Cahn--Hilliard equation in terms of the Caputo derivative
  \begin{equation} \label{Eq:CahnHilliardFractional}
     \begin{aligned}
     \p_t^\alpha \phi &= \div(m(\phi)\nabla\mu) + f,\\
     \mu &= \Psi'(\phi) - \eps^2 \Delta \phi,
     \end{aligned}
 \end{equation}
where we added a source term $f$ to the right hand side. We supplement the system with the initial data $\phi(0)=\phi_0$ in some weak sense and the homogeneous Neumann boundary $\nabla\phi \cdot n_\Omega=\nabla \mu \cdot n_\Omega=0$, where $n_\Omega$ is an outer normal to $\p \Omega$. We analyze the initial-boundary value problem in a given bounded domain $\Omega \subset \R^d$ on the time interval $(0,T)$ in \cref{Sec:AnalysisCahn}. Shortly, we write for the time-space cylinder $\Omega_T=(0,T) \times \Omega$.

\begin{remark}
Instead of relaxing the mass flux $J$, we could have also relaxed the chemical potential which was proposed in \cite{rotstein2001phase}, in the following way
\begin{equation*} 
    \begin{aligned}
    \mu_\rel = \int_0^t k(t-s) \delta\E(\phi(s)) \ds =  \p_t\int_{0}^t g_{\alpha}(t-s)\delta\E(\phi(s))\,\dd s = \p_t \left(g_\alpha *\delta\E(\phi)\right)(t) = D_t^{1-\alpha}\delta\E(\phi).
    \end{aligned}
\end{equation*}
Following this relaxation, we obtain the system
\begin{equation} \label{Eq:DerivationPDE1}
    \begin{aligned}
    \p_t \phi &= \div(m(\phi)\nabla\mu_\rel) =  \div(m(\phi)D_t^{1-\alpha}\nabla\mu), \\
        \mu &= \Psi'(\phi)-\eps^2 \Delta \phi.
    \end{aligned}
\end{equation}
This system is not equivalent to the time-fractional Cahn--Hilliard equation \cref{Eq:CahnHilliardFractional} since we cannot interchange the kernel $g_{1-\alpha}$ with the time-dependent mobility function $m$. In fact, we can calculate the difference of the systems via the product rule of fractional derivatives \cite{diethelm2010analysis}, i.e., for two sufficiently smooth functions $f_1$ and $f_2$,
$$D_t^\alpha(f_1 f_2)=f_1 D_t^{\alpha} f_2 + \sum_{k=1}^\infty \binom{\alpha}{k} \p_t^k f_1 \cdot (g_{1-k+\alpha} * f_2),$$
which yields in the application on the right hand side of \cref{Eq:DerivationPDE1},
$$\p_t \phi = \div(m(\phi) D_t^{1-\alpha} \nabla \mu) = D_t^{1-\alpha} \div(m(\phi) \nabla \mu) - \div \sum_{k=1}^\infty \binom{1-\alpha}{k} \frac{\dd^k}{(\dd t)^k} m(\phi) \cdot (g_{2-k-\alpha} * \nabla \mu).$$
Taking the convolution on both sides with the kernel $g_{1-\alpha}$ yields
$$\p_t^\alpha (\phi-\phi_0) = \div(m(\phi) \nabla \mu) - \div g_{1-\alpha} * \sum_{k=1}^\infty \binom{1-\alpha}{k} \frac{\dd^k}{(\dd t)^k} m(\phi) \cdot  (g_{2-k-\alpha} * \nabla \mu).$$
Note that the left hand side and the first term on the right hand side correspond to the time-fractional Cahn--Hilliard equation. Additionally, we have the infinite sum on the right hand side, which is zero in the special case of a time-independent mobility function. If we assume a mobility of the form $m(\phi)=M\phi$, $M>0$, we achieve
$$\p_t^\alpha \phi = \div(m(\phi) \nabla \mu) - (1-\alpha) M \div (g_{1-\alpha} * (\p_t \phi \cdot (g_{1-\alpha} * \nabla \mu))).$$
One observes that a time derivative of $\phi$ appears on the right hand side, which makes it the leading term and changes the structure of the PDE. For nonlinear mobility functions $m$ this change amplifies even more, resulting in convolutions with $\p_t^k \phi$ on the right hand side.
\end{remark}

\section{Analytical Preliminaries} \label{Sec:Preliminaries}
In this section, we shortly introduce the function spaces and analytical techniques, which will be used frequently in the following sections. We equip the Sobolev and Lebesgue spaces $W^{k,p}(\Omega)$ and $L^p(\Omega)$ on $\Omega$ with the norms $\|\cdot\|_{W^{k,p}(\Omega)}$ and $\|\cdot\|_{L^p(\Omega)}$; their vector-valued variants are denoted by $W^{k,p}(\Omega)^d$ and $L^p(\Omega)^d$. We write the dual product as $\langle f,u \rangle_X$ for $f \in X'$ and $u \in X$. 
We recall the Young convolution, the Poincar\'e--Wirtinger and the Sobolev inequalities \cite{brezis2010functional,evans2010partial,roubicek}
\begin{equation}\begin{aligned}
		\|f*g\|_{L^r(\Omega)} & \leq \|f\|_{L^p(\Omega)} \|g\|_{L^q(\Omega)} && \text{for all } f \in L^p(\Omega), ~g \in L^q(\Omega), \quad \smash{\frac{1}{p}+\frac{1}{q} = 1 + \frac{1}{r}}, \\
		\|f-\langle f\rangle_\Omega\|_{L^p(\Omega)} & \leq C \|\nabla f\|_{L^p(\Omega)} && \text{for all } f \in W^{1,p}(\Omega),   \\
		\|f\|_{L^p(\Omega)}              & \leq C  \|\nabla f\|_{L^p(\Omega)} && \text{for all } f \in W^{1,p}_0(\Omega), \\[-.18cm]
		\|f\|_{W^{m,q}(\Omega)} & \leq C \|f\|_{W^{k,p}(\Omega)} && \text{for all } f \in W^{k,p}(\Omega), \quad k-\frac{d}{p}  \geq m-\frac{d}{q} , \quad k\geq m,
	\end{aligned} \label{Eq:SobolevInequality} \end{equation}
where $\langle f\rangle_\Omega=\frac{1}{|\Omega|} (f,1)_{L^2(\Omega)}$ denotes the mean of $f$. For a given Banach space $X$, we define the Bochner space of order $p \in [1,\infty)$, see, e.g., \cite{evans2010partial},
$$L^p(0,T;X)=\big\{u:(0,T) \to X : u \text{ is strongly measurable}, ~ \|u\|_{L^p(0,T;X)}^p=\textstyle\int_0^T \|u(t)\|_X^p \, \dt <\infty \big\}.$$
In the case of $p=\infty$, we equip the space $L^\infty(0,T;X)$ with the norm $\|u\|_{L^\infty(0,T;X)} = \esssup_{t \in (0,T)} \|u(t)\|_X$. We define the fractional Sobolev--Bochner space as
$$W^{\alpha,p}(0,T;X)=\{u \in L^p(0,T;X) : \p_t^\alpha u \in L^p(0,T;X)\}.$$
In the special case of $p=2$, we write $H^\alpha(0,T;X)$. We note that this space is equal to the Bessel potential space of order $\alpha$, denoted by $H^\alpha_2(0,T;X)$.

In the existence proof later, we need compactness results to pass to the limit in the nonlinear functions $\Psi$ and $m$. For a given Gelfand triple of Banach spaces $X \com Y \con Z$, we recall the classical Aubin--Lions lemma \cite{simon1986compact},
$$L^p(0,T;X) \cap W^{1,1}(0,T;Z) \com L^p(0,T;Y), \quad p \in [1,\infty).$$

Analogously, we have in the fractional setting, see \cite[Theorem 3.1]{wittbold2020bounded} and \cite{ouedjedi2019galerkin},
\begin{equation}
    \label{Eq:Compact}
    \begin{aligned}
L^{p}(0,T;X) \cap W^{\alpha,p}(0,T;Z) &\com L^p(0,T;Y), \quad p \in [1,\infty), \\
L^{p}(0,T;X) \cap W^{\alpha,1}(0,T;Z) &\com L^r(0,T;Y), \quad 1\leq r<p.
\end{aligned}
\end{equation}

We note that the product and chain rules for integer order derivatives, which facilitate to obtain key estimates, are not available for fractional derivatives \cite{tarasov2016chain}. The following proposition serves as an alternative to chain rule in the fractional setting for semiconvex functionals. Here, we call a function $f:\R\to\R$ semiconvex if it is $\lambda$-convex for some $\lambda \in \R$, i.e., the function $x \mapsto f(x)-\frac{\lambda}{2}|x|^2$ is convex. If $\lambda<0$ holds, then semiconvexity is a weaker notion than convexity; for $\lambda>0$ it implies convexity. If $f \in C^1(\R)$ holds, then $\lambda$-convexity is equivalent to the condition 
\begin{equation} \label{Eq:LambdaC1}
f(y)-f(x) \geq f'(x)(y-x) + \frac{\lambda}{2} (y-x)^2 \quad \text{ for all } x,y \in \R.
\end{equation}
If $f \in C^2(\R)$ holds, then $f$ is $\lambda$-convex if and only if \begin{equation}  \label{Eq:LambdaC2}
f''(x)\geq \lambda \quad \text{ for all } x \in \R.
\end{equation}

\begin{proposition}[Fractional chain inequality] \label{Lem:Chain} Let $V$ be a Banach space such that $V \hookrightarrow L^2(\Omega) \hookrightarrow V'$  forms a Gelfand triple. Let $u \in H^\alpha(0,T;V')\cap L^\infty(0,T;V)$ with $u_0 \in L^2(\Omega)$, and $E \in C^1(\R)$ a $\lambda$-convex function with $\lambda \in \R$.
If $E'(u) \in L^2(0,T;V)$, then we have for all $t\in(0,T)$

\begin{subequations} \label{Eq:ChainE}
\begin{align}
\int_0^t \langle \pta u,E'(u)\rangle_V - \lambda \langle \p_t^\alpha u,u \rangle_V \dd s \geq g_{1-\alpha}*\int_\Omega E(u)- E(u_0) \dx - \frac{\lambda}{2} g_{1-\alpha}*\big(\|u\|^2_{L^2(\Omega)}-\|u_0\|^2_{L^2(\Omega)}\big),\label{Eq:ChainE-1}\\
g_{\alpha}*\langle \pta u,E'(u)\rangle_V - \lambda g_{\alpha}*\langle \p_t^\alpha u,u \rangle_V \geq \int_\Omega E(u)- E(u_0) \dx - \frac{\lambda}{2} \big(\|u\|^2_{L^2(\Omega)}-\|u_0\|^2_{L^2(\Omega)}\big).\label{Eq:ChainE-2}
\end{align}
\end{subequations}
\end{proposition}
\begin{proof} 
We define the Yosida approximation $g_{1-\alpha}^k \in W^{1,1}(0,T)$ of the kernel $g_{1-\alpha}$ as in \cite{vergara2008lyapunov}. One can show that the operators $B_i u = \p_t^\alpha u$, $i \in \{1,2\}$, with domains $D(B_1)=W^{\alpha,1}(0,T)$ and $D(B_2)=H^\alpha(0,T;V')$ are $m$-accretive in $L^1(0,T)$ and $L^2(0,T;V')$, respectively, see \cite{gripenberg1985volterra}. Their Yosida approximations can be defined by $B_i^k=\p_t(g_{1-\alpha}^k * u)(t)$, where $g_{1-\alpha}^k = k s^k$. Here, $s^k$ solves the Volterra equation $$s^k(t)+k(g_\alpha*s^k)(t)=1.$$ Then $s^k \in W^{1,1}(0,T)$ is nonnegative and nonincreasing, see \cite{pruss2013evolutionary}. Hence, $g_{1-\alpha}^k$ is nonincreasing, i.e., $(g_{1-\alpha}^k)'(t) \leq 0$, a.e., in $(0,T)$. For $k>0$, let $h^k$ denote the resolvent kernel associated with $kg_{\alpha}$, that is, $$h^k(t)+k(h^k*g_{\alpha})=kg_{\alpha}(t).$$ Then it can be shown, see \cite{zacher2008boundedness}, that the following property holds \begin{equation} \label{Eq:Kernel_hk} g_{1-\alpha}^k= g_{1-\alpha}*h^k.
\end{equation}

By definition of the approximation, we have $B_i^k u \to B_i u$ for all $u \in D(B_i)$. Further, for any $u\in L^1(0,T)$, $g_{\alpha}*u\in D(B_1)$ and we have 
$h^k*u = \p_t(g_{1-\alpha}^k*g_{\alpha}*u)\rightarrow \p_t(g_{1-\alpha}*g_{\alpha}*u)=u$ in $L^1(0,T)$. In particular, we have $1\in D(B_1)$, $u \in D(B_2)$, $\int_\Omega E(u(\cdot,x)) \dx \in L^1(0,T)$ and $\|u\|_{L^2(\Omega)}\in L^1(0,T)$. Thus as $k \to \infty$, we have
$$\begin{aligned}
g_{1-\alpha}^k &\longrightarrow g_{1-\alpha} &&\text{ in } L^1(0,T), \\
\p_t (g_{1-\alpha}^k * (u-u_0)) &\longrightarrow \p_t^\alpha u &&\text{ in } L^2(0,T;V'), \\
h^k * \int_\Omega \left(E(u(\cdot,x))-E(u_0)\right)\dx &\longrightarrow \int_\Omega (E(u(\cdot,x))-E(u_0))\dx &&\text{ in } L^1(0,T), \\
h^k * \left(\|u\|_{L^2(\Omega)}^2-\|u_0\|_{L^2(\Omega)}^2\right) &\longrightarrow \left(\|u\|_{L^2(\Omega)}^2-\|u_0\|_{L^2(\Omega)}^2\right) &&\text{ in } L^1(0,T), \\
g_{1-\alpha}^k*\int_\Omega \left(E(u(\cdot,x))-E(u_0)\right)\dx &\longrightarrow g_{1-\alpha}*\int_\Omega (E(u(\cdot,x))-E(u_0))\dx &&\text{ in } L^1(0,T), \\
g_{1-\alpha}^k * \left(\|u\|_{L^2(\Omega)}^2-\|u_0\|_{L^2(\Omega)}^2\right) &\longrightarrow g_{1-\alpha}*\left(\|u\|_{L^2(\Omega)}^2-\|u_0\|_{L^2(\Omega)}^2\right) &&\text{ in } L^1(0,T). \\
\end{aligned} 
$$
By a standard argument there is a subsequence of $g_{1-\alpha}^k$ which gives pointwise convergence a.e. in $(0,T)$. In the following, we drop the subsequence index. With the more regular kernel, we have from \cite[Lemma 4.1]{gripenberg1990volterra}, particular case of $E(x)=\frac{x^2}{2}$ is given in \cite[Lemma 2.1]{zacher2009weak}
\begin{equation*}
\begin{aligned}
&\int_\Omega \partial_t(g_{1-\alpha}^k*(u-u_0))(t)\,E'(u(t)) \dx -\lambda \int_\Omega  \partial_t(g_{1-\alpha}^k*(u-u_0))(t)\,u(t) \dx \\&= \partial_t\bigg(g_{1-\alpha}^k*\int_\Omega E(u)-E(u_0)\dx\bigg) - \frac{\lambda}{2} \partial_t\bigg(g_{1-\alpha}^k*\left(\|u\|^2_{L^2(\Omega)}-\|u_0\|^2_{L^2(\Omega)}\right)\bigg)\\
&\quad+  g_{1-\alpha}^k(t)  \int_\Omega E(u_0)- E(u(t))- E'(u(t)) (u_0-u(t))  - \frac{\lambda}{2} (u_0-u(t))^2\dx \\
&\quad - \int_0^t (g_{1-\alpha}^k)'(s) \int_\Omega  E(u(t-s))-E(u(t)) -  E'(u(t)) (u(t-s)-u(t)) - \frac{\lambda}{2}(u(t-s)-u(t))^2\dx    \ds,
\end{aligned}
\end{equation*}
for every $k \in \mathbb{N}$ and almost every $t \in (0,T)$. Noting that $g_{1-\alpha}^k$ is nonnegative and its derivative is nonincreasing, we apply the $\lambda$-convexity \cref{Eq:LambdaC1} of $E$ on the right hand side yielding
\begin{equation*} 
\begin{aligned}
&\int_\Omega \partial_t(g_{1-\alpha}^k*(u-u_0))(t)\,E'(u(t)) \dx -\lambda \int_\Omega  \partial_t(g_{1-\alpha}^k*(u-u_0))(t)\,u(t) \dx \\
&\geq \partial_t\bigg(g_{1-\alpha}^k*\int_\Omega E(u)-E(u_0)\dx\bigg) - \frac{\lambda}{2} \partial_t\bigg(g_{1-\alpha}^k*\left(\|u\|^2_{L^2(\Omega)}-\|u_0\|^2_{L^2(\Omega)}\right)\bigg).
\end{aligned}
\end{equation*}
Taking the integral from $0$ to $t$ and the convolution with $g_{\alpha}$ on both sides, using the property \cref{Eq:Kernel_hk} of the resolvent kernel $h_k$, it gives the following two inequalities, respectively,
\begin{subequations} \label{Eq:Limit}
\begin{align}
\int_0^t \int_\Omega \partial_s(g_{1-\alpha}^k*(u-u_0))(s)\,E'(u(s))\dx  \dd s -  \lambda \int_0^t \int_\Omega \p_t\left(g_{1-\alpha}^k*(u-u_0)\right)(s) u(s) \dx \dd s\nonumber\\
\geq \left(g_{1-\alpha}^k*\int_\Omega \left(E(u)-E(u_0)\right)\dx\right)(t) -\frac{\lambda}{2}g_{1-\alpha}^k* \left(\|u\|_{L^2(\Omega)}^2-\|u_0\|_{L^2(\Omega)}^2\right)(t),\\
g_{\alpha}*\left(\int_\Omega \partial_t(g_{1-\alpha}^k*(u-u_0))\,E'(u)\dx\right)(t)- \lambda g_{\alpha}*\left(\int_\Omega \p_t \left( g_{1-\alpha}^k*(u-u_0)\right) u \dx\right)(t)\nonumber \\
\geq  h^k*\left(\int_\Omega \left(E(u)-E(u_0)\right)\dx\right)(t)-\frac{\lambda}{2} h^k*\left(\|u\|_{L^2(\Omega)}^2-\|u_0\|_{L^2(\Omega)}^2\right)(t). 
\end{align}
\end{subequations}
Finally, taking the limit $k \to \infty$ in \cref{Eq:Limit} yields by dominated convergence theorem the required inequalities \cref{Eq:ChainE}.
\end{proof}

\begin{remark} \label{Remark:chain}
A similar result to \cref{Lem:Chain} was proved for convex  functionals in $\R^d$ in \cite[Corollary 6.1]{kemppainen2016decay} and in the distributional sense in \cite[Proposition 3.11]{li2018generalized}, by considering higher regularity of $u$ in \cite{li2018some}. The special case $E(\cdot)=\frac12 |\cdot|^2$ was proved in a Hilbert space setting in \cite[Theorem 2.1]{vergara2008lyapunov}. The key point is that in this special case it holds $\int_\Omega E(u)\dd x\in W^{\alpha,1}(0,T)$, see \cite[Proposition 2.1]{vergara2008lyapunov}. 
\end{remark} 

We will apply the fractional chain inequality in the proof for the existence of solutions in the discrete Faedo--Galerkin setting and in the continuous setting for the uniqueness. Further, we need a Gr\"onwall--Bellman type inequality in the proof of the existence of weak solutions to achieve an energy inequality. 
\begin{lemma}[cf. {\cite[Corollary 1]{ye2007generalized}}] \label{Lem:FractionalGronwall}
Let $u \in L^1(0,T;\R_{\geq 0})$ and $a,b>0$. If $u$ satisfies
$$u(t) \leq a + b (g_\alpha * u)(t) \qquad \text{a.e. }  t \in (0,T),$$
then we have
$$u(t)  \leq a \cdot C(\alpha, b,T)  \qquad \text{a.e. }  t \in (0,T).$$
\end{lemma}

We prove the following corollary of the fractional Gr\"onwall--Bellman inequality, which is more convenient in the application for the energy estimates in the existence proof.
\begin{corollary}[Fractional Gr\"onwall--Bellman inequality]
\label{Lem:FractionalGronwall1}
Let $u,v \in L^1(0,T;\R_{\geq 0})$, and $a,b>0$. If $u$ and $v$ satisfy
$$u(t)+g_{\alpha}*v(t) \leq a + b (g_\alpha * u)(t) \qquad \text{a.e. }  t \in (0,T),$$
then we have
$$u(t)+ v(t) \leq a \cdot C(\alpha,b,T)  \qquad \text{a.e. }  t \in (0,T).$$
\end{corollary}

\begin{proof}
We define the function $w$ by
$w(t)=u(t)+(g_{\alpha}* v)(t)$ for almost all $t \in (0,T)$.
Hence, we have by the non-negativity of the function $v$ and the kernel $g_\alpha$ the inequality
$$w(t) \leq a+b(g_\alpha * u)(t) \leq a+b(g_\alpha * w)(t).$$
Applying \cref{Lem:FractionalGronwall} yields the desired result.
\end{proof}

\section{Analysis: Time-Fractional Cahn--Hilliard Equation} \label{Sec:AnalysisCahn}

For notational simplicity, we define the spaces $H$ and $V$ in the Gelfand triple as
$$V=H^1(\Omega) \com H=L^2(\Omega) \com V'.$$
We state and prove the existence of weak solutions via the Faedo--Galerkin and compactness methods in \cref{Thm:PosMob}. Moreover, in the case of a constant mobility we prove uniqueness and continuous dependence on the data. In both cases, we assume a potential with certain properties which are fulfilled by the semiconvex Landau potential \cref{Eq:TypicalChoice} for example. Higher regularity is shown in \cref{Thm:Reg} in \cref{Sub:Reg}. For more general cases, e.g., the degenerating mobility $m(x)=|1-x^2|^\nu$, $\nu>0$, and potentials of double-obstacle and Flory--Huggins type, we refer to \cref{Thm:DegMob} in \cref{Sub:Deg}.

\subsection{Positive mobility} First, we consider the case of a positive and bounded mobility, e.g., the continuous function $$m(x)=\delta+\beta \bbone_{[-1,1]} |1-x^2|^\nu, \quad \nu \geq 0,$$ for $\delta,\beta>0$.  Here, $\bbone$ denotes the characteristic function. For $\beta=0$ it gives the constant mobility.
We make the following assumption regarding the well-posedness theorem below.

\begin{assumption} \label{As:Pos} ~\\[-0.4cm]
\begin{enumerate}[label=(A\arabic*), ref=A\arabic*] \itemsep.1em
    \item $\Omega \subset \R^d$ bounded $C^{1,1}$-domain with $d \geq 2$, $T>0$ finite time horizon.
    \item $f \in L^\infty(0,T;H)$, and $\phi_0 \in V$ with $\Psi(\phi_0) \in L^1(\Omega)$.
    \item $m \in C^0(\R)$ such that $0< M_0 \leq m(x) \leq M_\infty$ for all $x \in \R$ for some constants $M_0,M_\infty<\infty$. \label{Ass:Pos:Mob}
    \item $\Psi \in C^{1,1}(\R;\R_{\geq 0})$, $\Psi(0)=\Psi'(0)=0$, is $(-C_\Psi)$-convex and $|\Psi'(x)|\leq C_\Psi (1+\Psi(x))$ for all $x \in \R$ for some $C_\Psi < \infty$.
    \label{Ass:Pos:Pot}
\end{enumerate}
\end{assumption}

We state the existence and uniqueness theorem as follows.

\begin{theorem} \label{Thm:PosMob} Let \cref{As:Pos} hold.
	Then there exists a weak solution $(\phi,\mu)$ with 
	$$\begin{aligned}
	\phi &\in H^\alpha(0,T;V') \cap L^\infty(0,T;V) \text{ with }
	g_{1-\alpha} * \phi \in C^0([0,T];H), \\
	\mu &\in L^2(0,T;V),
	\end{aligned}$$
	to \cref{Eq:CahnHilliardFractional} in the sense that $g_{1-\alpha} * (\phi-\phi_0)(0)=0$ in $H$ and
	\begin{equation} \label{Eq:Weak}
	\begin{aligned}
	\langle \pta \phi,\xi\rangle_V+ (m(\phi) \nabla \mu,\nabla \xi)_{H} &= (f,\xi)_H, \\
	 (\Psi'(\phi),\zeta)_H + \eps^2 (\nabla \phi,\nabla \zeta)_H & =(\mu,\zeta)_H ,
	\end{aligned}
	\end{equation}
	for all $\xi, \zeta \in V$. The weak solution satisfies the energy inequality
	\begin{equation} \label{Eq:EnergySolution} \begin{aligned}
	&\|\phi\|_{L^\infty(0,T;V)}^2 +\|\mu\|_{L^2(0,T;V)}^2+\|\sqrt{m(\phi)}\nabla \mu\|_{L^2(0,T;H)}^2 +  \|\Psi(\phi)\|_{L^\infty(0,T;L^1(\Omega))}  \\  &\leq  C(T)\cdot \big(\|\phi_0\|_{V}^2+\|\Psi(\phi_0)\|_{L^1(\Omega)} + \|f\|_{L^2(\OT)}^2 \big).
	\end{aligned} \end{equation}
	Moreover, if $\alpha >\tfrac12$, then $\phi(0)=\phi_0$ in $H$ and $\phi \in C([0,T];H)$. If the mobility $m$ is constant, then the solution is unique and depends continuously on the data $\phi_0$ and $f$.
\end{theorem}

\noindent \textit{Proof.} We employ the Faedo--Galerkin method \cite{lions1969some} to reduce the system to fractional ordinary differential equations, which admit a solution $(\phi^k,\mu^k)$ due to a well-studied theory \cite{diethelm2010analysis}. We derive energy estimates, which imply the existence of weakly convergent subsequences by the Eberlein--\v{S}mulian theorem. We pass to the limit $k\to \infty$ and apply compactness methods to return to the time-fractional Cahn--Hilliard equation. Recently, the Faedo--Galerkin method has been applied to various time-fractional PDEs, see, e.g., \cite{li2018some,fritz2020subdiffusive,djilali2018galerkin}.
	
\subsubsection*{Discrete approximation}

	 Let $\{h_k\}_{k \in \N}$ be the eigenfunctions of the Neumann--Laplace operator with corresponding eigenvalues $\{\lambda_k\}_{k\in \N}$. The eigenfunctions form an orthonormal basis  in $H$ and an orthogonal basis in $V$ with $(h_i,h_j)_V=\lambda_j \delta_{ij}$, see \cite{evans2010partial}. 
	 We pursue a function $(\phi^k,\mu^k)$ that takes its values in $H_k=\{h_1,\dots,h_k\}$, i.e., is of the form
	\begin{equation} \label{Eq:Ansatz}
	    \begin{aligned}
\phi^k(t) = \sum_{j=1}^k \phi^k_j(t) h_j, \qquad \mu^k(t) = \sum_{j=1}^k \mu^k_j(t) h_j,
	    \end{aligned}
	\end{equation}
	with coefficient functions $\phi_j^k,\mu_j^k : (0,T) \to \R$, $j \in \{1,\dots,k\}$, that solve the Faedo--Galerkin system 
	\begin{equation} \begin{aligned}
	(\p_t^\alpha \phi^k,u)_H + (m(\phi^k) \nabla \mu^k, \nabla u)_H &= (f,u)_H,  \\
	(\Psi'(\phi^k),v)_H + \eps^2 (\nabla \phi^k, \nabla v)_H &= (\mu^k,v)_H, 
	\end{aligned} \label{Eq:FaedoUV}
	\end{equation}
	for all $u,v \in H_k$. We equip the system with the initial data $\phi^k(0)=\Pi_k \phi_0$, where $\Pi_k: H \to H_k$, $h \mapsto \sum_{i=1}^k (h,h_j)_H h_j$, denotes the orthogonal projection onto $H_k$. In particular, we exploit its key properties $\|\Pi_k\|_{\mathcal{L}(H)}\leq 1$, $\|\Pi_k\|_{\mathcal{L}(V)}\leq 1$ and
	\begin{equation} \label{Eq:Projection}
	\phi^k(0) = \Pi_k \phi_0 \longrightarrow \phi_0 \quad \text{in $V$ as }  k \to \infty,\end{equation}
	e.g., see \cite[Lemma 7.5]{robinson2001infinite}. Additionally, for an element $g \in V'$, we have $\langle \Pi_k g,v \rangle_V = \langle g,\Pi_k v \rangle_V$ for all $v \in V$.
	
	Since the test functions $u,v \in H_k$ are spanned by the eigenfunctions $h_j$, $j \in \{1,\dots,k\}$, we can equivalently write the Faedo--Galerkin system as
		\begin{subequations} \begin{align}
	(\p_t^\alpha \phi^k,h_j)_H + (m(\phi^k) \nabla \mu^k, \nabla h_j)_H &= (f,h_j)_H, \label[equation]{Eq:FaedoPhi} \\
	(\Psi'(\phi^k),h_j)_H + \eps^2 (\nabla \phi^k, \nabla h_j)_H &= (\mu^k,h_j)_H, \label[equation]{Eq:FaedoMu}
	\end{align} \label[equation]{Eq:Faedo}
	\end{subequations}
	for all $j \in \{1,\dots, k\}$.
     Inserting the ansatz functions \cref{Eq:Ansatz} into this system  and exploiting the orthonormality of the eigenfunctions in $H$ and their orthogonality in $V$, it yields
		\begin{equation} \begin{aligned}
	\p_t^\alpha \phi^k_j &=- \lambda_j  \mu^k_j \sum_{i=1}^k \big(m\big( \textstyle \sum_{j=1}^k \phi^k_j(t) h_j\big) \nabla h_i,\nabla h_j \big)_H + (f,h_j)_H,  \\
	\mu^k_j &= \big(\Psi'\big(\textstyle \sum_{j=1}^k \phi^k_j(t) h_j\big),h_j\big)_H + \eps^2 \lambda_j \phi^k_j, 
	\end{aligned} 
	\end{equation}for all $j \in \{1,\dots,k\}$, and the initial data $\phi_j^k(0)=(\phi_0,h_j)_H$. Note that the right hand side depends continuously on $\phi_1^k,\dots,\phi_k^k$.
By the theory of fractional ODEs, e.g., see \cite[Theorem A.1]{fritz2020subdiffusive}, and using the fact that $f\in L^\infty(0,T;H)$, there exists a solution $(\phi_j^k,\mu_j^k)$ to the fractional ODE on a time interval $[0,T_k)$ with either $T_k=\infty$ or $T_k<\infty$ and $\limsup_{t \to T_k} |(\phi_1^k,\dots,\phi_k^k)|_{\ell^2}=\infty$. 
Therefore, we have shown the existence of a solution tuple $$(\phi^k,\mu^k) \in H^\alpha(0,T_k;H_k)\cap L^\infty(0,T;H_k) \times L^2(0,T_k;H_k),$$
	to the Faedo--Galerkin system \cref{Eq:Faedo}.

	\subsubsection*{Energy estimates} After we have proven the existence of a solution to the ODE, we can begin to test the Faedo--Galerkin system \cref{Eq:FaedoUV} with suitable test functions.
	First, we take the test functions $u=\mu^k+C_{\Psi}\phi^k$ and $v=-\p_t^\alpha \phi^k$ in \cref{Eq:FaedoUV}, which yields
	\begin{equation*}
	    \begin{aligned}
	    (\partial_t^\alpha \phi^k,\mu^k+\phi^k)_H  &= - (m(\phi^k) \nabla \mu^k, \nabla \mu^k + \nabla\phi^k)_H+ ( f,\mu^k+\phi^k)_H,  \\
	    -(\mu^k,\p_t^\alpha \phi^k)_H &= -(\Psi'(\phi^k) ,\p_t^\alpha \phi^k)_H - \eps^2 (\nabla \phi^k, \nabla \p_t^\alpha \phi^k)_H.
	    \end{aligned}
	\end{equation*}
Adding the two equations above cancels the term $(\p_t^\alpha \phi^k,\mu^k)_H$, and we have
\begin{equation} \label{Eq:EEstimate2} 
	\begin{aligned}
&C_{\Psi}(\partial_t^\alpha \phi^k,\phi^k)_H + \eps^2 (\nabla \phi^k, \p_t^\alpha \nabla \phi^k)_H + (\Psi'(\phi^k),\p_t^\alpha \phi^k)_H  + (m(\phi^k),|\nabla \mu^k|^2)_H\\
&=  ( f,\mu^k + C_{\Psi}\phi^k)_H + C_{\Psi}(m(\phi^k) \nabla \mu^k, \nabla\phi^k).
	\end{aligned}
	\end{equation}
Testing with  $v=1$ in \cref{Eq:Faedo} yields $(\mu^k,1)_H =  (\Psi'(\phi^k),1)_H$ and consequently, we have by the Poincar\'e--Wirtinger inequality \cref{Eq:SobolevInequality}
	\begin{equation*}
	\|\mu^k\|_H \leq \big\|\mu^k - |\Omega|^{-1} (\mu^k,1)_H \big\|_H + |\Omega|^{-1} \|(\mu^k,1)_H \|_H \leq C \|\nabla \mu^k\|_H + |\Omega|^{-1/2} \|\Psi'(\phi^k)\|_{L^1(\Omega)},
	\end{equation*}
	and due to assumption \cref{Ass:Pos:Pot} we can bound the derivative of the potential giving
	\begin{equation}\label{Eq:EEstimate_mu}
	    \|\mu^k\|_H \leq C \big(1+\|\nabla \mu^k\|_H+\|\Psi(\phi^k)\|_{L^1(\Omega)}\big).
	\end{equation}
We use the lower bound of $m$, see \cref{Ass:Pos:Mob}, insert \cref{Eq:EEstimate_mu} on the right hand side in \cref{Eq:EEstimate2} and use the $\eps$-Young inequality to make the prefactors of $\|\nabla \mu^k\|_H$ sufficiently small in order to absorb them by the terms on the left hand side of the inequality. This procedure gives the estimate
\begin{equation}\label{Eq:EEstimate1} 
\begin{aligned}
&C_{\Psi}(\partial_t^\alpha \phi^k,\phi^k)_H + \eps^2 (\nabla \phi^k, \p_t^\alpha \nabla \phi^k)_H + (\Psi'(\phi^k),\p_t^\alpha \phi^k)_H  + \frac{M_0}{2}\|\nabla \mu^k\|_H^2\\
&\leq C\left(\|f\|_H + \|\phi^k\|_H + \|\nabla\phi^k\|_H + \|\Psi(\phi^k)\|_{L^1(\Omega)}\right).
\end{aligned}
\end{equation}
Convolving on both sides by $g_{\alpha}$, and applying the fractional chain inequality \cref{Eq:ChainE-2} on the convex functional $\nabla \phi^k \mapsto \frac12 |\nabla \phi^k|^2$ and the $\lambda$-convex functional $\phi^k \mapsto \Psi(\phi^k)$ with $\lambda=-C_\Psi$, gives
\begin{equation*} 
	\begin{aligned}
	&\frac{C_{\Psi}}{2} \|\phi^k(t)\|_H^2  + \frac{\eps^2}{2} \|\nabla \phi^k(t)\|_H^2 +  \|\Psi(\phi^k(t))\|_{L^1(\Omega)}+  \frac{M_0
		}{2} \big(g_\alpha * \|\nabla \mu^k\|_{H}^2\big)(t)\\
	&\leq C\Big( 1+ \big(g_{\alpha}*\|f\|_{H}^2\big)(t) +  \big( g_{\alpha}*\|\Psi(\phi^k)\|_{L^1(\Omega)}\big)(t) + \big(g_{\alpha}*\|\phi\|_{V}^2\big)(t) + \|\Psi(\phi^k(0))\|_{L^1(\Omega)}  +  \|\phi^k(0)\|_{V}^2 \Big). 
	\end{aligned}
	\end{equation*}
According to the fractional Gr\"onwall--Bellman inequality, see \cref{Lem:FractionalGronwall1}, we infer
\begin{equation} \label{Eq:FFinal2} \|\phi^k(t)\|_{V}^2 +  \|\Psi(\phi^k(t))\|_{L^1(\Omega)}    \leq  C(T) \cdot \big(\|f\|_{L^\infty(0,T;H)}^2+\|\Psi(\phi_0)\|_{L^1(\Omega)}+ \|\phi_0\|_{V}^2 \big),
\end{equation}
for almost all $t\in (0,T_k)$, where we used $\phi^k(0)=\Pi_k \phi_0$ and the boundedness of the operator norm of the orthogonal projection. Next we take the integral from $0$ to $t<T_k$ in \eqref{Eq:EEstimate1}, apply the fractional chain inequality \cref{Eq:ChainE-1} on the similar functionals as before to get
\begin{equation} \label{Eq:FFinal3} \|\nabla \mu^k\|_{L^2(0,t;H)}^2  \leq  C(T) \cdot \big(\|f\|_{L^2(\OT)}^2+\|\Psi(\phi_0)\|_{L^1(\Omega)}+ \|\phi_0\|_{V}^2\big).
\end{equation}
Combining $\cref{Eq:EEstimate_mu},\cref{Eq:FFinal2}$ and $\cref{Eq:FFinal3}$ yields the energy estimate
\begin{equation*}   \begin{aligned} &\|\phi^k\|_{L^\infty(0,T_k;V)}^2  +  \|\Psi(\phi^k)\|_{L^\infty(0,T_k;L^1(\Omega))} +\|\mu^k\|_{L^2(0,T_k;V)}^2 \\ &  \leq  C(T) \cdot \big(\|f\|_{L^\infty(0,T;H)}^2+\|\Psi(\phi_0)\|_{L^1(\Omega)}+ \|\phi_0\|_V^2\big).
\end{aligned} \end{equation*}

We complete this energy estimate by testing with $u=\phi^k$ in \cref{Eq:FaedoUV} 
and argue by the boundedness of $m$, see \cref{Ass:Pos:Mob}, to achieve a bound of the term $\big\|\sqrt{m(\phi^k)} \nabla \mu^k\big\|_{L^2(0,T_k;H)}$. This gives the energy bound
\begin{equation}   \label{Eq:FFinalEnergy} \begin{aligned} &\|\phi^k\|_{L^\infty(0,T_k;V)}^2+  \|\Psi(\phi^k)\|_{L^\infty(0,T_k;L^1(\Omega))}+ \|\mu^k\|_{L^2(0,T_k;V)}^2 +\smash{\big\|\textstyle\sqrt{m(\phi^k)} \nabla \mu^k\big\|_{L^2(0,T_k;H)}^2}  \\ &  \leq  C(T) \cdot \big(\|f\|_{L^\infty(0,T;H)}^2+\|\Psi(\phi_0)\|_{L^1(\Omega)}+ \|\phi_0\|_V^2\big).
\end{aligned} \end{equation}
Since the right hand side is independent of $k$, we can argue with a blow-up criterion and extend the time interval by setting $T_k=T$ for all $k$.

	\subsubsection*{Estimate on the fractional time-derivative} The energy estimate \cref{Eq:FFinalEnergy} already gives the existence of converging subsequences. Since the Faedo--Galerkin system \cref{Eq:Faedo} involves the nonlinear functions $\Psi$ and $m$, we need to derive an estimate on the time-fractional derivative of $\phi$ in order to apply the compactness result \cref{Eq:Compact} and achieve strong convergence. 
	
	Let $u \in L^2(0,T;V)$. Then we have $\Pi_k u = \sum_{j=1}^k u_j^k h_j$ for time-dependent coefficient functions $u_j^k : (0,T) \to \R$, $j \in \{1,\dots,k\}$. We multiply equation \cref{Eq:FaedoPhi} by $u_j^k$, take the sum from $j=1$ to $k$, and integrate over the interval $(0,T)$, which yields 
	\begin{equation*}
	\begin{aligned}\left|\int_0^T(\p_t^\alpha \phi^k, u)_H \dt\right| &= \left|\int_0^T(\p_t^\alpha \phi^k,\Pi_k u)_H\dt\right| \\ &\leq M_\infty \|\nabla \mu^k\|_{L^2(0,T;H)} \|\nabla \Pi_k u\|_{L^2(0,T;H)} + \|f\|_{L^2(\OT)} \|\Pi_k u\|_{L^2(0,T;V)} \\ &\leq C(T,f,\phi_0) \|u\|_{L^2(0,T;V)},
	\end{aligned}
	\end{equation*}
	where we used the energy estimate \cref{Eq:FFinalEnergy} to bound the terms on the right hand side. Since $u$ was chosen arbitrarily, we have
	\begin{equation}	\label{Eq:Energy3}
	    \| \p_t^\alpha \phi^k \|_{L^2(0,T;V')} = \sup_{\|u\|_{L^2(0,T;V)} \leq 1} \Big|\int_0^T(\p_t^\alpha \phi^k, u)_H \dt\Big| \leq C(T,f,\phi_0).
	\end{equation}
	
	\subsubsection*{Limit process} We note that the Eberlein--\v{S}mulian theorem infers that a bounded sequence in a reflexive Banach space \cite{diestel1977vector} has a weakly/weakly-$*$ convergent subsequence. By a standard abuse of notation, we drop the subsequence index. Hence from the energy estimate \cref{Eq:FFinalEnergy} and \cref{Eq:Energy3}, we obtain the existence of limit functions $\phi$ and $\mu$ such that
	$$\begin{aligned}	
	\phi^k &\longweak \phi &&\text{weakly-$*$ in } L^\infty(0,T;V), \\
	\p_t^\alpha \phi^k &\longweak \p_t^\alpha \phi &&\text{weakly\phantom{-*} in } L^2(0,T;V'), \\
		\phi^k &\longrightarrow \phi &&\text{strongly\hspace{1mm} in }
 L^p(0,T;H), \\
	\mu^k &\longweak \mu &&\text{weakly\phantom{-*} in } L^2(0,T;V),
	\end{aligned}$$
for all $p \in [1,\infty)$ as $k \to \infty$, where we applied the compact embedding \cref{Eq:Compact} to achieve the strong convergence of $\phi^k$. Moreover, the weak limit of $\p_t^\alpha \phi^k$  is equal to  $\p_t^\alpha \phi$, see \cite[Proposition 3.5]{li2018some}.

In a next step, we prove that the limit functions $\phi$ and $\mu$ satisfy the weak form of the time-fractional Cahn--Hilliard equation \cref{Eq:Weak}. By multiplying the Faedo--Galerkin system \cref{Eq:Faedo} by a test function $\eta \in C^\infty_c(0,T)$ and integrating over the time interval $(0,T)$, which find
	\begin{equation} \begin{aligned}
\int_0^T (\p_t^\alpha \phi^k,h_j)_H \eta(t) \dt   + \int_0^T (m(\phi^k) \nabla \mu^k, \nabla h_j)_H \eta(t) \dt &= \int_0^T (f,h_j)_H \eta(t) \dt,  \\
\int_0^T (\Psi'(\phi^k),h_j)_H \eta(t) \dt +\eps^2 \int_0^T  (\nabla \phi^k, \nabla h_j)_H \eta(t) \dt &= \int_0^T (\mu^k,h_j)_H \eta(t) \dt,
\end{aligned} \label{Eq:FaedoTime}
\end{equation}
for all $j \in \{1,\dots,k\}$. We take the limit $k \to \infty$ in these two equations. The linear terms follow directly from the weak/weak-$*$ convergences, e.g., the functional
$$\phi^k \mapsto \eps^2 \int_0^T (\nabla \phi^k,\nabla h_j)_H \eta(t) \dt$$
is linear and continuous on $L^\infty(0,T;V)$, since we have
$$\left| \eps^2 \int_0^T (\nabla \phi^k,\nabla h_j)_H \eta(t) \dt \right| \leq \eps^2 \|\nabla \phi^k\|_{L^\infty(0,T;V)} \|h_j\|_V \|\eta\|_{L^1(0,T)}.$$
The weak-$*$ convergence gives by definition as $k \to \infty$
$$\eps^2 \int_0^T (\nabla \phi^k,\nabla h_j)_H \eta(t) \dt \longrightarrow \eps^2 \int_0^T (\nabla \phi,\nabla h_j)_H \eta(t) \dt.$$

It remains to treat the integrals involving the nonlinear functions $m$ and $\Psi'$. Since $m$ and $\Psi'$ are continuous functions, see \cref{Ass:Pos:Mob}, we have by the strong convergence $\phi^k \to \phi$ in $L^2(\OT)$ also $m(\phi^k) \to m(\phi)$ a.e. in $\OT$. By the boundedness of $m$, we infer from the Lebesgue dominated convergence theorem $m(\phi^k) \nabla h_j \eta \to m(\phi) \nabla h_j \eta$ in $L^2(\OT)^d$. By the weak convergence of $\nabla \mu^k$, we conclude as $k \to \infty$
$$m(\phi^k) \eta \nabla h_j \cdot \nabla \mu^k \longrightarrow m(\phi) \eta \nabla h_j \cdot \nabla \mu \quad \text{ in } L^1(\OT).$$
The continuity of $\Psi'$ gives then $\Psi'(\phi^k) \to \Psi'(\phi)$ a.e. in $\OT$. Further, from the assumption \cref{Ass:Pos:Pot} on the potential function $\Psi$, we infer the bound $$\|\Psi'(\phi^k(t)) \eta(t) h_j\|_{L^1(\Omega)} \leq C \|\eta\|_{L^\infty(0,T)} \cdot \|h_j\|_{H^2(\Omega)} \big( 1+\|\Psi(\phi^k(t))\|_{L^1(\Omega)}\big),$$
for almost every $t \in (0,T)$, and the right hand side is bounded by the energy estimate \cref{Eq:FFinalEnergy}. Consequently, the Lebesgue dominated convergence theorem gives for $k \to \infty$
$$\int_0^T (\Psi'(\phi^k),h_j)_H \eta(t) \dt \longrightarrow \int_0^T (\Psi'(\phi),h_j)_H \eta(t) \dt. $$

After having taking care of the nonlinear functions, we are ready to take the limit $k \to \infty$ in the equations \cref{Eq:FaedoTime} and use the density of $H_k$ in $V$, which yields 
	\begin{equation} \label{Eq:WeakInitial2} \begin{aligned}
\int_0^T \langle \pta \phi,h \rangle_V \eta(t) \dt   + \int_0^T (m(\phi) \nabla \mu, \nabla h)_H \eta(t) \dt &= \int_0^T (f,h)_H \eta(t) \dt,  \\
\int_0^T (\Psi(\phi),h)_H \eta(t) \dt + \int_0^T \eps^2 (\nabla \phi, \nabla h)_H \eta(t) \dt &= \int_0^T (\mu,h)_H \eta(t) \dt,
\end{aligned} 
\end{equation}
for all $h \in V$ and $\eta \in C_c^\infty(0,T)$. Applying the fundamental lemma of calculus of variations, we infer that $(\phi,\mu)$ is a weak solution of the time-fractional Cahn--Hilliard equation, i.e., satisfies the weak form \cref{Eq:Weak}.

\subsubsection*{Initial condition} From the estimate \cref{Eq:Energy3} we have $\p_t^\alpha \phi^k\in L^2(0,T;V')$, and by the definition of the fractional derivative this gives $g_{1-\alpha}*(\phi^k-\phi^k(0))\in H^{1}(0,T;V')$. By the continuous embedding $L^2(0,T;V)\cap H^1(0,T;V')\hookrightarrow C^0([0,T];H)$ we have $g_{1-\alpha}*(\phi^k-\phi^k(0))\in C^0(0,T;H)$. Now we repeat the steps from the limit process, but we test with $\eta \in C_c^\infty([0,T))$, i.e., $\eta$ has compact support on the set $[0,T)$ and does not necessarily vanish at $t=0$. This gives after integration by parts
\begin{equation} \label{Eq:WeakInitial} \begin{aligned}
-\int_0^T (\phi^k-\phi^k(0),h)_H \tilde\p_t^\alpha \eta(t) \dt  + \int_0^T (m(\phi^k) \nabla \mu^k, \nabla h)_H \eta(t) \dt \\
= (g_{1-\alpha} * (\phi^k-\phi^k(0))(0),h_j)_H \eta(0)+ \int_0^T (f,h)_H \eta(t) \dt,
\end{aligned} 
\end{equation}
for all $j \in \{1,\dots,k\}$. Here, $\tilde\p_t^\alpha$ denotes the right Caputo derivative, see \cite[Definition 2.6]{li2018some}. The only difference to before is the term with $g_\alpha *(\phi^k-\phi^k(0))(0)$ on the right hand side. We have
$$\begin{aligned}
\left| (g_{1-\alpha} *(\phi^k-\phi^k(0))(0),h_j)_H \eta(0) \right| & \leq \|\phi^k-\phi^k(0)\|_{L^\infty(0,T;H)}\|h_j\|_H\|\eta\|_{L^\infty(0,T)}g_{2-\alpha}(0)=0
\end{aligned}
$$
for all $k$. Now, taking $k \to \infty$ in \cref{Eq:WeakInitial} and repeating the steps from above, yields in a comparison with the weak form of the solution $(\phi,\mu)$ after integration with $\eta \in C_c^\infty([0,T))$
$$(g_\alpha * (\phi-\phi_0)(0),h)_H \eta(0) = 0, $$
for all $h \in H$ and $\eta \in C_c^\infty([0,T))$.
Choosing a test function with $\eta(0)=1$ yields the required result.

\subsubsection*{Energy inequality} We prove that the solution tuple $(\phi,\mu)$ satisfies the energy inequality \cref{Eq:EnergySolution}. First, we note that norms are weakly/weakly-$*$ lower semicontinuous, e.g., we have $\mu^k \weak \mu$ in $L^2(0,T;V)$ and therefore, we infer
$$\| \mu \|_{L^2(0,T;V)} \leq \liminf_{k \to \infty} \|\mu^k\|_{L^2(0;T;V)}.$$  
We apply the Fatou lemma on the continuous and non-negative function $\Psi$ to obtain
$$\int_\Omega \Psi(\phi) \, \dd x  \leq \liminf_{k \to \infty} \int_\Omega \Psi(\phi^k) \, \dd x.$$ 
Consequently, passing to the limit $k \to \infty$ in the discrete energy inequality \cref{Eq:FFinalEnergy} leads to  \cref{Eq:EnergySolution}. 
	
\subsubsection*{Uniqueness} We assume the case of a constant mobility $m=M>0$. Consider two weak solution pairs $(\phi_1,\mu_1)$ and $(\phi_2,\mu_2)$, and we denote their differences by $\phi=\phi_1-\phi_2$ and $\mu=\mu_1-\mu_2$. Each pair fulfills the weak form, and we find for $(\phi,\mu)$
	\begin{equation*} 
	\begin{aligned}
	\langle \p_t^\alpha \phi,u \rangle_V + M (\nabla \mu,\nabla u)_H &= 0, \\
	 (\Psi'(\phi_1)-\Psi'(\phi_2),v)_H + \eps^2 (\nabla \phi,\nabla v)_H & =(\mu,v)_H.
	\end{aligned}
	\end{equation*}
for test functions $u,v \in V$.
Taking $u=(-\Delta)^{-1} \phi$ and $v=M\phi$, yields
	\begin{equation*} 
	\begin{aligned}
	\langle \p_t^\alpha \phi,(-\Delta)^{-1} \phi\rangle_V + M (\nabla \mu,\nabla (-\Delta)^{-1} \phi)_H &= 0, \\
	 M(\Psi'(\phi_1)-\Psi'(\phi_2),\phi)_H + M\eps^2 (\nabla \phi,\nabla \phi)_H & =M(\mu,\phi)_H.
	\end{aligned}
	\end{equation*} 
Exploiting the property $(\nabla \mu, \nabla (-\Delta)^{-1} \phi)_H = (\mu,\phi)_H$ of the Neumann--Laplace operator, gives after adding the equations and canceling,
\begin{equation} \label{Eq:UniqueDifference}
\begin{aligned}
\langle \p_t^\alpha \phi, (-\Delta)^{-1} \phi\rangle_V + M \eps^2 \|\nabla \phi\|_H^2 = M(\Psi'(\phi_2)-\Psi'(\phi_1),\phi)_H.
\end{aligned}
\end{equation}
It can be seen that $[((-\Delta)^{-1}\cdot,\cdot)_H]^{1/2}$ is a norm on $H_0=\{u \in H: \int_\Omega u \dx = 0 \}$. Moreover, we note that the domain of the operator $\nabla (-\Delta)^{-1}$ is equal to the dual space of the domain of $\nabla$, which is in fact equal to $V_0'$ for $V_0=\{u \in V: \int_\Omega u \dx = 0 \}$. Therefore, the graph norm $\|\nabla (-\Delta)^{-1} \cdot\|_H$ is equivalent to the usual norm of $V_0$. We set $\|\cdot\|_{V_0'}=\|\nabla (-\Delta)^{-1} \cdot\|_H$ and note that $\int_\Omega \phi \dx =0$ by taking the test function $u=1$ and using the fractional chain inequality \cref{Eq:ChainE-1}. 
\begin{equation} 
\begin{aligned} 
\langle \p_t^\alpha \phi, (-\Delta)^{-1} \phi\rangle_V = \langle -\Delta (-\Delta)^{-1} \p_t^\alpha \phi, (-\Delta)^{-1} \phi\rangle_V  &= \langle\p_t^\alpha \nabla (-\Delta)^{-1} \phi, \nabla (-\Delta)^{-1} \phi\rangle_V. 
\end{aligned} \label{Eq:UniqueChain} 
\end{equation}
Using the $(-C_\Psi)$-convexity of $\Psi$, see \cref{Eq:LambdaC2}, we have by the mean value theorem
$$(\Psi'(\phi_1)-\Psi'(\phi_2),\phi)_H \geq - C_\Psi \|\phi\|_H^2,
$$
and consequently, we obtain by the $\eps$-Young inequality 
$$(\Psi'(\phi_2)-\Psi'(\phi_1),\phi)_H \leq C_\Psi \|\phi\|_H^2 = C_\Psi (\nabla (-\Delta)^{-1} \phi,\nabla \phi)_H \leq \frac{\eps^2}{2} \|\nabla \phi\|^2_H + \frac{C_\Psi^2}{2\eps^2} \|\nabla (-\Delta)^{-1} \phi\|_H^2.$$
Therefore, applying this estimate and \cref{Eq:UniqueChain} to \cref{Eq:UniqueDifference} yields
		\begin{equation*} 
	\begin{aligned}
\langle\p_t^\alpha \nabla (-\Delta)^{-1} \phi, \nabla (-\Delta)^{-1} \phi\rangle_V  + \frac{M \eps^2}{2} \|\nabla \phi\|_H^2 \leq \frac{MC_\Psi^2}{2\eps^2} \|\phi\|_{V_0'}^2,
	\end{aligned}
	\end{equation*}
and convolving with $g_\alpha$, using fractional chain inequality \cref{Eq:ChainE-2} and applying the fractional Gr\"onwall--Bellmann inequality, see \cref{Lem:FractionalGronwall1}, gives
		\begin{equation} \label{Eq:WeakDifference}
	\begin{aligned}
	\frac12 \| \phi(t)\|_{V_0'}^2 + \frac{M \eps^2}{2} g_{\alpha}*\|\nabla \phi\|_{H}^2 \leq C(T) \cdot \|\phi_0 \|_{V_0'}^2 = 0,
	\end{aligned}
	\end{equation}
hence $\phi_1=\phi_2$ in the sense $\|(\phi_1-\phi_2)(t)\|_{V_0'}=0$ for a.e. $t \in [0,T]$, and consequently $\mu_1=\mu_2$.  

\subsubsection*{Continuous dependence} The proof follows analogously to the procedure of uniqueness, deriving \cref{Eq:WeakDifference} with two initial conditions $\phi_{1,0}$ and $\phi_{2,0}$ and data $f_1,f_2$, resulting in the inequality
\begin{equation*}\pushQED{\qed} 
	\frac12 \|\phi_1(t)-\phi_2(t)\|_{V_0'}^2 + \frac{M \eps^2}{2} g_{\alpha}*\|\nabla (\phi_1- \phi_2)\|_{H}^2 \leq C(T) \cdot \left(  \|\phi_{1,0}- \phi_{2,0}\|_{V_0'}^2 + \|f_1-f_2\|_{L^2(\OT)} \right).
 \qedhere
 \popQED
	\end{equation*}

\begin{remark}
We cannot guarantee continuity-in-time of the solution due to the low regularity in the case of small $\alpha$. The embedding $H^\alpha(0,T;X) \hookrightarrow C([0,T];X)$ holds for $\alpha > \tfrac12$, \cite[Theorem 2.2.4/1]{runst2011sobolev}.

\end{remark}

\subsection{Higher spatial regularity} \label{Sub:Reg}

We adapt the proofs on the higher regularity of the Cahn--Hilliard equation, see, e.g.,
\cite{garcke2018optimal},
to the time-fractional case. 

\begin{theorem} \label{Thm:Reg}
	Let the assumption of \cref{Thm:PosMob} hold. Then there exists a weak solution $(\phi,\mu)$ to the time-fractional Cahn--Hilliard equation in the sense
	$$\begin{aligned} \pta \phi &= \div(m(\phi) \nabla \mu) +f &&\text{ in } L^2(0,T;V'), \\ \mu&=\Psi'(\phi)-\eps^2 \Delta \phi &&\text{ a.e. in } \OT, \end{aligned}$$ with the additional regularity $\phi \in L^2(0,T;H^2(\Omega))$. Moreover, the energy inequality can be extended to
		\begin{equation} \label{Eq:EnergyExtended} 
	\|\sqrt{m(\phi)} \nabla \mu\|_{L^2(\OT)}^2  +  \|\phi\|_{L^\infty(0,T;V)}^2+\|\phi\|_{L^2(0,T;H^2(\Omega))}^2  \leq  C(T,f,\phi_0). \end{equation}
	Additionally, if $\Psi \in C^2(\R)$ satisfies the growth estimate 
	\begin{equation} \label{Eq:Growth} |\Psi''(x)|\leq C(1+|x|^r) \quad \text{for all } x \in \R \text{ where } \begin{cases} r=\frac{2}{d-2}, &d >2, \\
	r \geq 2, &d=2,
	\end{cases}\end{equation} for some constant $C<\infty$, then it holds $\Psi'(\phi) \in L^2(0,T;V)$ and $\phi \in L^2(0,T;H^3(\Omega))$.
\end{theorem}
\begin{proof}
In the proof of \cref{Thm:PosMob}, we have in the Faedo--Galerkin setting $\phi^k(t) \in H_k \subset H^2(\Omega)$. 
We take the test function $\Delta \phi^k(t) \in H_k$ in the equation for $\mu^k$, which gives
$$\eps^2 \|\Delta \phi^k\|_H^2 = (\nabla \mu^k,\nabla \phi^k)_H - (\Psi''(\phi^k),|\nabla \phi^k|^2)_H.$$
Using the additional assumption $\Psi \in C^2(\R)$ it holds by the semiconvexity $\Psi''(x) \geq -C_\Psi$ for all $x \in \R$, and thus, we arrive after integrating from $0$ to $T$ at
$$\eps^2 \|\Delta \phi^k\|_{L^2(\OT)}^2 \leq \|\nabla\mu^k\|_{L^2(\OT)} \|\nabla\phi^k\|_{L^2(\OT)} + C_\Psi \|\nabla \phi^k\|_{L^2(\OT)}^2 \leq C(T,\phi_0).$$
Since $(\|\cdot\|_H^2+\|\Delta \cdot\|_H^2)^2$ is an equivalent norm on $H^2(\Omega)$, see \cite[III.Lemma 4.2]{temam2012infinite}, it yields the uniform boundedness of $\phi^k$ in $L^2(0,T;H^2(\Omega))$ and consequently, by the reflexivity of the Hilbert space it holds for the limit $\phi \in L^2(0,T;H^2(\Omega))$.

Inserting $\mu^k = \Pi_k \Psi'(\phi^k)-\eps^2 \Delta \phi^k$ into the equation of $\phi^k$ and considering the Galerkin system 
$$(\pta \phi^k,u)_H+(m(\phi^k)  \nabla \Pi_k\Psi'(\phi^k),\nabla u)_H - (m(\phi^k) \nabla \Delta \phi^k,\nabla u)_H = (f,u)_H,$$
for all $u \in H_k$, and taking the test function $u=-\Delta \phi^k$, we get
$$\begin{aligned} (\pta\nabla\phi^k,\nabla\phi^k)_H  + M_0 \|\nabla \Delta \phi^k\|_H^2 &\leq \|f\|_H \|\Delta \phi^k\|_H + M_\infty \|\nabla \Psi'(\phi^k)\|_H  \|\nabla \Delta \phi^k\|_H \\ &\leq C(T,\phi_0,f) + C\|\nabla \Psi'(\phi^k)\|_H^2 + \frac{M_0}{2} \|\nabla \Delta \phi^k\|_H^2.\end{aligned}$$
By assumption \cref{Eq:Growth} it holds the growth estimate $|\Psi''(x)| \leq C(1+|x|^r)$ for $r=\frac{2}{d-2}$ for all $x \in \R$ in the case of $d>2$ (for $d=2$ choose any exponent $r \geq 2$ and use the embedding $V \hookrightarrow L^r(\Omega)$).
Therefore, we apply the H\"older and Sobolev inequalities \cref{Eq:SobolevInequality} to obtain the bound
$$\|\nabla \Psi'(\phi^k)\|_{H} = \|\Psi''(\phi^k) \nabla \phi^k\|_{H} \leq \|\Psi''(\phi^k)\|_{L^d(\Omega)} \|\nabla \phi^k\|_{L^{2d/(d-2)}(\Omega)}
\leq C \|1+\phi^k\|_{V}^{2/(d-2)} \|\nabla \phi^k\|_{V}.$$
Taking the square on both sides and integrating over the interval $[0,T]$, it yields
$$\|\nabla \Psi'(\phi^k)\|_{L^2(0,T;H)}  \leq C \|1+\phi^k\|_{L^\infty(0,T;V)}^{4/(d-2)} \| \phi^k\|_{L^2(0,T;H^2(\Omega))}\leq C(T,\phi_0),$$
and thus, it follows from typical estimates $\nabla \Delta \phi^k \in L^2(\Omega_T)$ and elliptic regularity theory \cite{agmon1959estimates} gives $\phi \in L^2(0,T;H^3(\Omega))$.
\end{proof}

\begin{remark}
We can derive a formal estimate on the Ginzburg--Landau energy $\E$, see \cref{Eq:Ginzburg}, in the case of the constant mobility $m=M$ and zero force $f=0$ by taking the test functions $u=\frac{1}{M} (-\Delta)^{-1} \p_t \phi$ and $v=\p_t \phi$, which gives
$$\ddt (\Psi(\phi),1)_H + \frac{\eps^2}{2} \ddt \|\nabla \phi\|_H^2 + \frac{1}{M} (\p_t^\alpha \phi, (-\Delta)^{-1} \p_t \phi)_H + (\nabla \mu, \nabla (-\Delta)^{-1} \p_t \phi)_H = (\mu, \p_t \phi)_H. $$
Note that the Ginzburg--Landau energy is given by $\E(\phi)=\frac{\eps^2}{2} \|\nabla \phi\|_H^2 + (\Psi(\phi),1)_H$ and thus, we have
$$\ddt \E(\phi) =  -\frac{1}{M} (\p_t^\alpha \phi, (-\Delta)^{-1} \p_t \phi)_H =  -\frac{1}{M} (\p_t^\alpha \nabla (-\Delta)^{-1} \phi, \nabla (-\Delta)^{-1} \p_t \phi)_H. $$
After integrating on $(0,t)$, we apply the inequality \cite[Lemma 3.1]{mustapha2014well} on the term on the right hand side  to achieve
$$\E(\phi(t))-\E(\phi_0) \leq -\frac{\cos((1-\alpha) \pi/2)}{M}  \|\p_t^{\alpha/2} \phi\|^2_{L^2(0,T;V_0')} \leq 0, $$
and therefore, one can bound the energy at time $t$ by the initial energy. This property is also called weak energy stability in the topic of numerical schemes. Note that the energy dissipation of gradient flows of fractional order is an open problem, e.g., see the discussion in \cite{tang2019energy}.
\end{remark}
\subsection{Degenerating mobility} \label{Sub:Deg}
	
	We employ the same technique as in \cite{elliott1996cahn,abels2013incompressible,dai2016weak,scarpa2020stochastic}, and approximate and extend the mobility function $m \in W^{1,\infty}(-1,1)$ with $m(x)>0$ for all $x\in (-1,1)$ and  $m(\pm 1)=0$ by a strictly positive function $m_\delta$ in the following way:
	$$m_\delta(x) = \begin{cases} m(\delta-1), &\ifs x \leq \delta-1, \\ m(x), &\ifs |x| \leq 1-\delta, \\ m(1-\delta), &\ifs x \geq 1-\delta, \end{cases}$$
where $\delta \in (0,1)$. We extend $m$ by zero outside of $[-1,1]$ and denote the extension by $\ov m \in W^{1,\infty}(\R)$. Note that $m_\delta'=\ov m'$ on $[-1+\delta,1-\delta]$. The approximation $m_\delta$ is positive and admits regularity in $W^{1,\infty}(\R)$ with the upper and lower bounds (for $\delta$ sufficiently small)
	$$0<\min\{m(-1+\delta),m(1-\delta)\} \leq m_\delta(x) \leq \max_{y \in [-1,1]} m(y)  \quad \forall x \in \R.$$
	
	Further, we consider the potential $\Psi:(-1,1) \to \R_{\geq 0}$ and assume the splitting $\Psi=\Psi_1+\Psi_2$ with $\Psi_1 \in C^2(-1,1)$ convex and $\Psi_2 \in C^2([-1,1])$ being $(-C_\Psi)$-convex. We define its regularization $\Psi_\delta : \R \to \R$ as $\Psi_\delta = \Psi_{1,\delta} + \ov\Psi_2$ where $\Psi_{1,\delta}\in C^2(\R)$ is the unique function with $\Psi_{1,\delta}(0)=\Psi_1(0)$, $\Psi_{1,\delta}'(0)=\Psi_1'(0)$, and 
		$$(\Psi_{1,\delta})''(x) = \begin{cases} (\Psi_1)''(\delta-1), &\ifs x \leq \delta-1, \\ (\Psi_1)''(x), &\ifs |x| \leq 1-\delta, \\ (\Psi_1)''(1-\delta), &\ifs x \geq 1-\delta. \end{cases}$$
		In particular, $\Psi_{1,\delta}$ is convex on $\R$ since $\Psi_1$ itself is assumed to be convex on $(-1,1)$. Moreover, we introduce the extension $\ov \Psi_2\in C^2(\R)$ of $\Psi_2$ to the reals by setting
		$$\ov\Psi_2(x)=\begin{cases} \Psi_2(-1)+\Psi_2'(-1)(x+1) + \frac12 \Psi_2''(-1) (x+1)^2, &\ifs x < -1, \\
		\Psi_2(x), &\ifs |x|\leq 1, \\
		\Psi_2(1)+\Psi_2'(1)(x-1) + \frac12 \Psi_2''(1) (x-1)^2, &\ifs x >1. 
		\end{cases}$$
		It holds $\|\ov{\Psi}_2''\|_{C(\R)} \leq \| \Psi_2''\|_{C([-1,1])} \leq C$ and $\Psi''_\delta(x) \geq - C_\Psi$ for all $x \in \R$. By definition we have $\Psi_\delta=\Psi$ and $m_\delta = m$ on the interval $[-1+\delta,1-\delta]$ for $\delta \in (0,1)$, see also \cref{Fig:Approx} for a depiction of the approximations. 
		
			\begin{figure}[H]
		\centering
			\begin{tikzpicture}
			\begin{axis}[
			width=.47\textwidth,
			legend pos = north east, axis lines=middle, samples=70, smooth,
			xtick={-1,1}, ytick=\empty, ymin=-.2, ymax=1.1,xmin=-1.5,xmax=1.5,legend cell align={left},style={fill=white, fill opacity=0.6, draw opacity=1,text opacity=1}]
			\addplot[blue,domain=-1:1, line width=2pt] {(1-x^2)^2};
			\addlegendentry{$\ov m$}
			\addplot[green!80!black,densely dashed,domain=-0.9:0.9, line width=2pt] {(1-x^2)^2};
			\addlegendentry{$m_{0.1}$}
						\addplot[red,loosely dashed,domain=-0.8:0.8, line width=2pt] {(1-x^2)^2};
			\addplot[green!80!black,densely dashed,domain=-1.5:-0.9, line width=2pt] {(1-0.9^2)^2};
			\addplot[green!80!black,densely dotted, domain=0.9:1.5, line width=2pt] {(1-0.9^2)^2};
			\addplot[red,loosely dashed,domain=-1.5:-0.8, line width=2pt] {(1-0.8^2)^2};
			\addplot[red,loosely dashed,domain=0.8:1.5, line width=2pt] {(1-0.8^2)^2};
						\addlegendentry{$m_{0.2}$}
			\addplot[blue,domain=-1.5:-1, line width=2pt] {0};
			\addplot[blue,domain=1:1.5, line width=2pt] {0};
			\end{axis}
			\end{tikzpicture} 
						\begin{tikzpicture}
			\begin{axis}[
			width=.47\textwidth,
			legend pos = north east, axis lines=middle, samples=70, smooth,
			xtick={-1,1}, ytick=\empty, ymin=-1.5, ymax=8,xmin=-1.5,xmax=1.5,legend cell align={left},style={fill=white, fill opacity=0.6, draw opacity=1,text opacity=1}]
			\addplot[blue,domain=-0.99:0.99, line width=2pt] {-2+1/(1-x^2)};
			\addlegendentry{$\Psi''$}
			\addplot[green!80!black,densely dashed,domain=-0.9:0.9, line width=2pt] {-2+1/(1-x^2)};
			\addlegendentry{$\Psi''_{0.1}$}
						\addplot[red,loosely dashed,domain=-0.8:0.8, line width=2pt] {-2+1/(1-x^2)};
			\addplot[green!80!black,densely dashed,domain=-2:-0.9, line width=2pt] {-2+1/(1-0.9^2)};
			\addplot[green!80!black,densely dotted, domain=0.9:2, line width=2pt] {-2+1/(1-0.9^2)};
			\addplot[red,loosely dashed,domain=-2:-0.8, line width=2pt] {-2+1/(1-0.8^2)};
			\addplot[red,loosely dashed,domain=0.8:2, line width=2pt] {-2+1/(1-0.8^2)};
						\addlegendentry{$\Psi''_{0.2}$}
			\end{axis}
			\end{tikzpicture}

		\caption{\label{Fig:Approx} Depiction of the functions $\ov m$ for $m(x)=(1-x^2)^2$ and the second derivative of the Flory--Huggins potential, see \cref{Eq:Flory}, and their approximations $m_\delta$ (left) and $\Psi_\delta''$ for $\delta \in \{0.1,0.2\}$ (right).}
	\end{figure}
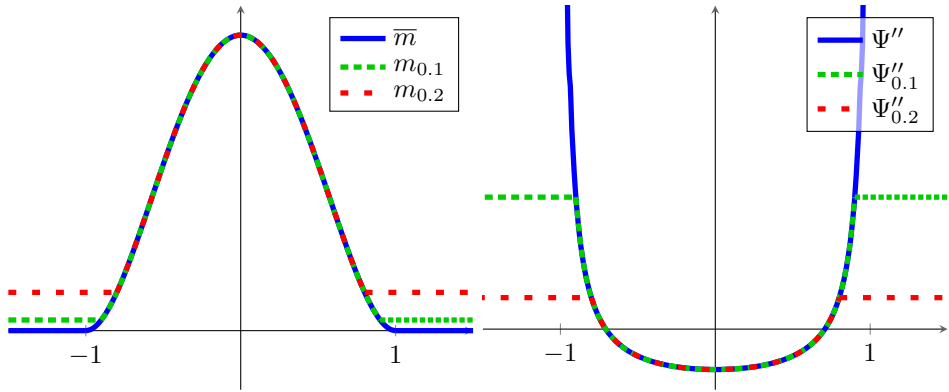

	We consider the auxiliary problem
	\begin{equation}\begin{aligned}
	\partial_t^\alpha \phi_\delta  &=  \div ( m_\delta(\phi_\delta) \nabla \mu_\delta),  \\
	\mu_\delta &= \Psi_\delta'(\phi_\delta) - \eps^2 \Delta \phi_\delta,
	\end{aligned} \label{Eq:CHdelta} \end{equation}
	with initial data $\phi_{\delta,0}=\phi_0 \in (-1,1)$ a.e. in $\Omega$, which has a weak solution $(\phi_\delta,\mu_\delta)$ according to \cref{Thm:PosMob} and \cref{Thm:Reg}, i.e., it satisfies
	\begin{equation}\begin{aligned}
	\langle \partial_t^\alpha \phi_\delta,\xi\rangle_V  &=  -(m_\delta(\phi_\delta) \nabla \mu_\delta,\nabla \xi)_H,  \\
	(\mu_\delta,\zeta)_H &= (\Psi_\delta'(\phi_\delta),\zeta)_H - \eps^2 (\Delta \phi_\delta,\zeta)_H,
	\end{aligned} \label{Eq:CHdeltaWeak} \end{equation}
	for all $\xi \in V$ and $\zeta \in H$. Indeed, the mobility $m_\delta$ is positive, continuous and bounded, and the potential $\Psi_\delta=\Psi_{1,\delta}+\ov \Psi_2$ is $(-C_\Psi)$-convex as discussed before and fulfills the required growth estimates due to the definitions of $\Psi_{1,\delta}$ and $\ov \Psi_2$.
	Redoing the estimates from \cref{Thm:PosMob}, see \cref{Eq:EnergySolution}, we have the $\delta$-uniform energy estimate
    \begin{equation} \label{Eq:UniformEnergy} \begin{aligned} &\| \p_t^\alpha \phi_\delta \|_{L^2(0,T;V')}^2+ \|\phi_\delta\|_{L^\infty(0,T;V)}^2 +\|\sqrt{m_\delta(\phi_\delta)}\nabla\mu_\delta\|_{L^2(\OT)}^2 +  \|\Psi_{\delta}(\phi_\delta)\|_{L^\infty(0,T;L^1(\Omega))}   \\ &\leq  C(T) \big(\|\phi_0\|_V+\|\Psi(\phi_0)\|_{L^1(\Omega)}\big), \end{aligned}
    \end{equation}
    where we used that $\Psi_\delta(\phi_{0}) =\Psi(\phi_0)$  a.e. on $\Omega$ for $\delta$ sufficiently small, see the discussion in \cref{Rk:PsiInitial} below. 
    
	We multiply the variational form by a smooth test function $\eta \in C_c^\infty(0,T)$ and exploit the density of the tensor space $C_c^\infty(0,T) \otimes V$ in $L^2(0,T;V)$ (and analogously for $H$) to formulate the weak form in terms of time-dependent test functions, i.e., 
	\begin{subequations}\begin{align}
	\label[equation]{Eq:CHdeltaWeakTimePhi} \int_0^T \langle \partial_t^\alpha \phi_\delta,\xi\rangle_V \dt  &=  -\int_0^T (m_\delta(\phi_\delta) \nabla \mu_\delta,\nabla \xi)_H \dt,  \\
	\int_0^T (\mu_\delta,\zeta)_H \dt &= \int_0^T (\Psi_\delta'(\phi_\delta),\zeta)_H - \eps^2 (\Delta \phi_\delta,\zeta)_H \dt,
	\label[equation]{Eq:CHdeltaWeakTimeMu}
	\end{align} \label[equation]{Eq:CHdeltaWeakTime} \end{subequations}
	for all $\xi \in L^2(0,T;V)$ and $\zeta \in L^2(\OT)$.
	We derive $\delta$-uniform estimates and pass to the limit $\delta \to 0$. 
	
	We make the following assumptions for the following proofs.
	
	\begin{assumption} \label{As:Deg} ~\\[-0.4cm]
	\begin{enumerate}[label=(B\arabic*), ref=B\arabic*] \itemsep.1em
    \item $\Omega \subset \R^d$ bounded $C^{1,1}$-domain with $d \geq 2$, $T>0$ finite time horizon.
    \item $f \in L^2(\OT)$, and $\phi_0 \in V$ with $\Psi(\phi_0) \in L^1(\Omega)$, $\Phi(\phi_0) \in L^1(\Omega)$ (see \cref{Lem:DegMobEst}), and $|\phi_0(x)|\leq 1$ for a.e. $x \in \Omega$. \label{Ass:Deg:Ini}
    \item $\Psi=\Psi_1+\Psi_2$ with $\Psi_1 \in C^2(-1,1)$ convex and $\Psi_2 \in C^2([-1,1])$ being $(-C_\Psi)$-convex for some $C_\Psi<\infty$. \label{Ass:Deg:Pot}
        \item $m \in W^{1,\infty}(-1,1)$ such that $m(x) >0$ for all $x\in (-1,1)$, $m(\pm 1)=0$, and $m\Psi'' \in C^0([-1,1])$.\label{Ass:Deg:Mob}
\end{enumerate}
	\end{assumption}
	
	\begin{remark} \label{Rk:PsiInitial} We make the following remarks regarding \cref{As:Deg}.
	\begin{itemize} 
	\item We assume in \cref{Ass:Deg:Mob} a mobility, which degenerates at $\pm 1$. For the general case of degeneracy at points $a,b\in \R$, one has to shift the interval $[-1,1]$ by an operator $A:[-1,1] \to [a,b]$, see \cite{abels2013incompressible}. We assume that the mobility compensates an eventual blow-up of $\Psi''$ at $\pm 1$ by assuming $m\Psi'' \in C^0([-1,1])$ in \cref{Ass:Deg:Mob}. E.g., the Flory--Huggins potential \cref{Eq:Flory} has the second derivative $\Psi''(x)=\theta/(1-x^2)-\theta_0$ for $x \in (-1,1)$ and therefore, degenerates at $x=\pm 1$. Then with the typical mobility $m(x)=(1-x^2)^\nu$, $\nu\geq 1$, $x \in [-1,1]$, we have indeed $m \Psi'' \in C^0([-1,1])$. For the double-obstacle potential one chooses $\Psi_1=0$ and $\Psi_2(x)=1-x^2$, since $\Psi_1$ does not have to be defined on the boundary $\pm 1$ in the assumption \cref{Ass:Deg:Pot} of \cref{Thm:DegMob}.
	
	\item We remark that we assume $|\phi_0(x)|\leq 1$ a.e. in $\Omega$ in \cref{Ass:Deg:Ini} instead of excluding the values $\pm 1$ to guarantee $\Psi_\delta(\phi_0)=\Psi(\phi_0)$ for $\delta$ sufficiently small. We use the same argument as in \cite{abels2013incompressible}. The assumption $|\phi_0(x)| \leq 1$ a.e. implies $\frac{1}{|\Omega|}\int_\Omega \phi_0 \dx \in [-1,1]$. In the case of $\frac{1}{|\Omega|}\int_\Omega \phi_0 \dx=\pm 1$ it holds $\phi_0 = \pm 1$ a.e. in $\Omega$, which readily gives the existence of a weak solution $(\phi,J)=(\pm 1,0)$. Therefore, in the proof we solely consider the case $|\phi_0|<1$ almost everywhere.
	\end{itemize}
	\end{remark}
	
	We formulate and prove two lemmata, which will be needed in the proof of existence theorem in case of a degenerating mobility. First, we will derive an estimate on the so-called entropy function $\Phi$. Second, we will prove a key inequality which will allow us to achieve the result $\phi(t,x) \in [-1,1]$ for a.e. $(t,x) \in \OT$.

		\begin{lemma} \label{Lem:DegMobEst} Let \cref{As:Deg} hold. Further, let $\Phi:(-1,1) \to \R_{> 0}$ be the unique function, which is given by $\Phi''(x)=1/m(x)$, $\Phi'(0)=\Phi(0)=0$. Further, its approximation $\Phi_\delta:\R \to \R_{>0}$ is defined by $\Phi''_\delta(x)=1/m_\delta(x)$ and $\Phi_\delta'(0)=\Phi_\delta(0)=0$. Then the following $\delta$-uniform bound holds
	\begin{equation} \label{Eq:EnergyH2} \|\Phi_\delta(\phi_\delta)\|_{L^\infty(0,T;L^1(\Omega))} + \|\Delta \phi_\delta\|_{L^2(\OT)}^2 + \big( \Psi_{\delta}''(\phi_\delta),|\nabla\phi_\delta|^2\big)_{L^2(\OT)} \leq C(T,\phi_0). 
	\end{equation}
	\end{lemma}
	
	\begin{proof}
	After integrating by parts in the weak formulation \cref{Eq:CHdeltaWeak}, we have
	\begin{equation}  \label{Eq:WeakDeg1}   \langle \p_t^\alpha \phi_\delta, \xi \rangle_V = ( \mu_\delta , \div(m_\delta(\phi_\delta) \nabla \xi))_H=  (-\eps^2 \Delta \phi_\delta + \Psi_\delta'(\phi_\delta), \div (m_\delta(\phi_\delta) \nabla \xi))_H,
	\end{equation}
	for all $\xi \in V$. 
	Since it holds $\Phi_\delta''\in L^\infty(\R)$ by the boundedness of $m_\delta$, we have $\Phi_\delta' \in C^{0,1}(\R)$ and $\Phi_\delta'(\phi_\delta) \in L^2(0,T;V)$. Moreover, $\Phi_\delta$ is a convex and non-negative functional due to $\Phi_\delta''(x)>0$ for all $x \in \R$, and we can write $$\Phi_\delta(x)=\int_0^x \int_0^y \frac{1}{m_\delta(z)} \,\text{d}z \text{d}y.$$ Using the chain rule, we have
	$$\nabla \Phi'_\delta(\phi_\delta)=\Phi_\delta''(\phi_\delta) \nabla \phi_\delta = \frac{\nabla\phi_\delta}{m_\delta(\phi_\delta)} \in H,$$
	and thus, $\xi=\Phi_\delta'(\phi_\delta)\in V$ is a valid test function in \cref{Eq:CHdeltaWeak}; we obtain after integration by parts
	\begin{equation*} 
	\begin{aligned}
	\langle \p_t^\alpha \phi_\delta, \Phi_\delta'(\phi_\delta) \rangle_V &=  \big(-\eps^2 \Delta \phi_\delta + \Psi_\delta'(\phi_\delta),\div (m_\delta(\phi_\delta) \nabla \Phi_\delta'(\phi_\delta))\big)_H\\
	&= -\eps^2  \|\Delta \phi_\delta\|_H^2  - (\nabla \Psi_\delta'(\phi_\delta),  \nabla\phi_\delta)_H\\
	&= -\eps^2  \|\Delta \phi_\delta\|_H^2  - (\Psi_\delta''(\phi_\delta) ,|\nabla\phi_\delta|^2)_H.
	\end{aligned}
	\end{equation*} 
We take the convolution with the kernel $g_\alpha$ on both sides which yields
	\begin{equation*} 
	\begin{aligned}
	\big(g_\alpha * \langle \p_t^\alpha \phi_\delta, \Phi_\delta'(\phi_\delta) \rangle_V \big)(t) 
	= -\eps^2 \big(g_\alpha * \|\Delta \phi_\delta\|_H^2\big)(t)  - \big(g_\alpha * (\Psi_\delta''(\phi_\delta) ,|\nabla\phi_\delta|^2)_H\big)(t).
	\end{aligned}
	\end{equation*} 
	
Applying the convolved fractional chain inequality \cref{Eq:ChainE-2}, we have 
$$\big( g_\alpha * \langle \p_t^\alpha \phi_\delta, \Phi_\delta'(\phi_\delta) \rangle_V\big)(t) \geq \int_\Omega \Phi(\phi_\delta(t)) \dx - \int_\Omega \Phi_\delta(\phi_{0}) \dx,$$
and thus,
 $$\|\Phi_\delta(\phi_\delta)\|_{L^1(\Omega)} + \eps^2 g_{\alpha} * \|\Delta \phi_\delta\|_H^2 + g_{\alpha} * \big( \Psi_{\delta}''(\phi_\delta),|\nabla\phi_\delta|^2\big)_H \leq \|\Phi_\delta(\phi_0)\|_{L^1(\Omega)}.
$$
We note the property $\Phi_\delta(\phi_0) \leq \Phi(\phi_0)$ a.e. due to $m_\delta(\phi_0) \geq m(\phi_0)$ a.e., which gives the desired $\delta$-uniform bound.
	\end{proof}
	
	\begin{lemma} \label{Lem:DegMobEst2} Let \cref{As:Deg} hold. Then it yields	$\|(|\phi_\delta|-1)_+\|_{L^\infty(0,T;H)} \leq C \sqrt{\delta}$ where $x_+=\max\{0,x\}$.
	\end{lemma}
		
	\begin{proof}
	Using straightforward computations, we derive for all $x>1$ and $\delta \in (0,1)$ the following lower estimate
	$$\begin{aligned} \Phi_\delta(x) &= \Phi_\delta(1-\delta) + \Phi_\delta'(1-\delta) (x-(1-\delta))+\frac12 \Phi_\delta''(1-\delta) (x-(1-\delta))^2 \\
	&\geq \frac12 \Phi_\delta''(1-\delta) (x-1+\delta)^2 =  \frac{(x-1+\delta)^2}{2m_\delta(1-\delta)}  \geq \frac{(x-1)^2 }{2m_\delta(1-\delta)},
	\end{aligned}$$
    and analogously, it holds $\Phi_\delta(x) \geq  \frac{(x+1)^2}{2m_\delta(\delta-1)}$ for $x<-1$.   Combining these two results gives
    \begin{equation} \label{Eq:LemDeriv} (|x|-1)_+^2 \leq 2 \Phi_\delta(x) \max\{m_\delta(1-\delta),m_\delta(\delta-1)\},
    \end{equation}
    for all $x \in \R$. But we have $m_\delta(1-\delta)=m(1-\delta)$ and $m(1)=0$, which implies by the mean value theorem
    $$|m_\delta(1-\delta)| = |m(1-\delta)-m(1)| \leq \delta \|m'\|_{L^\infty(-1,1)},$$
    and analogously, it holds $|m_\delta(\delta-1)| \leq \delta\|m'\|_{L^\infty(-1,1)}$. 
    Hence, using \cref{Eq:LemDeriv}  we have
	\begin{equation*}\int_\Omega (|\phi_\delta|-1)_+^2 \dx \leq  2 \delta \|m'\|_{L^\infty(-1,1)} \int_\Omega \Phi_\delta(\phi_\delta) \dx,
	\end{equation*}
	a.e. in (0,T) and it yields the desired result after applying the bound of \cref{Lem:DegMobEst}.
	\end{proof}
	
	\noindent 
	Having proved the two lemmata, we are now ready to state and prove the existence theorem in the case of a degenerate mobility.
	\begin{theorem} \label{Thm:DegMob}
	Let \cref{As:Deg} hold. Then there exists a weak solution $(\phi,J)$ with
	$$\begin{aligned}
	\phi &\in H^\alpha(0,T;V') \cap L^\infty(0,T;V) \cap L^2(0,T;H^2(\Omega)) \text{ with } |\phi| \leq 1 \text{ a.e. in } \OT,\\ 
	J &\in L^2(\OT)^d,
	\end{aligned}$$
	to \cref{Eq:CahnHilliardFractional} in the sense that
\begin{subequations}
	\begin{align}\label[equation]{Eq:WeakDegPhi}
	\langle \p_t^\alpha \phi,\xi \rangle_{L^2(0,T;V)}  &=  (J,\nabla \xi)_{L^2(\OT)}, \\
	(J,\varphi)_{L^2(\OT)}  &= -(\Psi'(\phi) - \eps^2 \Delta \phi,\div(m(\phi)  \varphi))_{L^2(\OT)},
    \label[equation]{Eq:WeakDegJ}
	\end{align} \label[equation]{Eq:WeakDeg}
	\end{subequations}
	for all $\xi \in L^2(0,T;V),\varphi \in L^2(0,T;V^d) \cap L^\infty(\OT)^d$ with $\varphi \cdot n_\Omega = 0$ on $\p \Omega \times (0,T)$.
	\end{theorem}

		We note that \cref{Thm:DegMob} is not stating the existence of a tuple $(\phi,\mu)$ but instead $(\phi,J)$. This is due to the low regularity of $\mu$ in the degenerate case. In the weak form with the mass flux $J$ the terms are well-defined.	
	\begin{proof}
	We consider a weak solution $(\phi_\delta,\mu_\delta)$ to \cref{Eq:CHdelta} which exists by \cref{Thm:PosMob} and fulfills the $\delta$-uniform energy inequality \cref{Eq:UniformEnergy}.
	 Hence, there are functions $(\phi,J,\tilde J)$ such that
	$$\begin{aligned}	
	\phi_\delta &\longweak \phi &&\text{\rlap{weakly-$*$ }\phantom{strongly} ~in } L^\infty(0,T;V), \\
	\p_t^\alpha \phi_\delta &\longweak \p_t^\alpha \phi &&\text{\rlap{weakly}\phantom{strongly} ~in } L^2(0,T;V'), \\
	\phi_\delta &\longrightarrow \phi &&\text{strongly ~in }  L^2(\OT), \\
		\tilde J_\delta = -\sqrt{m_\delta(\phi_\delta)} \nabla  \mu_\delta &\longweak \tilde J &&\text{\rlap{weakly}\phantom{strongly} ~in }
 L^2(\OT)^d, \\
J_\delta = -m_\delta(\phi_\delta) \nabla  \mu_\delta &\longweak J &&\text{\rlap{weakly}\phantom{strongly} ~in }
 L^2(\OT)^d,
	\end{aligned}$$
	as $\delta \to 0$. Here, we used the estimate $$\|J_\delta\|_{L^2(\OT)}^2 = \|\sqrt{m_\delta(\phi_\delta)} \tilde J_\delta\|_{L^2(\OT)}^2 \leq \|m_\delta\|_{L^\infty(\R)} \|\tilde J_\delta\|_{L^2(\OT)}^2 \leq C(T,\phi_0). $$ 
	Due to the higher spatial regularity, see \cref{Thm:Reg} and the improved energy inequality \cref{Eq:EnergyH2}, we have
	\begin{equation} \label{Eq:Weak_Deg} \begin{aligned}\phi_\delta &\longweak \phi &&\text{\rlap{weakly}\phantom{strongly} in } L^2(0,T;H^2(\Omega)), \\
	\phi_\delta &\longrightarrow \phi &&\text{strongly in } L^2(0,T;V),
	\end{aligned}
	\end{equation}
	where we employed the compact embedding $$H^\alpha(0,T;V') \cap  L^2(0,T;H^2(\Omega)) \com L^2(0,T;V).$$ for the Gelfand triple $H^2(\Omega) \com V  \con V'.$ Moreover, using lower semicontinuity and passing to the limit $\delta \to 0$ in $\int_\Omega (|\phi_\delta|-1)_+^2 \dx \leq C\delta$, see \cref{Lem:DegMobEst2}, gives $|\phi(t,x)|\leq 1$ for a.e. $(t,x) \in \OT$.
	
	We take the limit $\delta \to 0$ in the weak form \cref{Eq:CHdeltaWeakTime} of the solution $(\phi_\delta,\mu_\delta)$ and use the weak and strong convergences resulting in
		$$\begin{aligned} \int_0^T \langle \p_t^\alpha \phi, \xi \rangle_V \dt &= \int_0^T (J,\nabla \xi)_H \dt, \\
		\int_0^T (\mu,\zeta)_H \dt &= \int_0^T (\Psi'(\phi) - \eps^2 \Delta \phi, \zeta)_H \dt,
		\end{aligned}$$
		for all $\xi\in L^2(0,T;V)$ and $\zeta \in  L^2(\OT)$. It remains to show that 
		$$J=-m(\phi) \div( \Psi'(\phi) -\eps^2 \Delta \phi),$$
		in the sense of the weak form \cref{Eq:WeakDegJ}.

			We take the test function $\zeta = \div(m_\delta(\phi_\delta) \varphi)$ in \cref{Eq:CHdeltaWeakTimeMu} for any $\varphi \in L^2(0,T;V^d) \cap L^\infty(\OT)^d$ with $\varphi \cdot n_\Omega =0$ on $\partial \Omega \times (0,T)$. Indeed, the test function is well defined due to	$$\begin{aligned}\|\div(m_\delta(\phi_\delta) \varphi)\|_{L^2(\OT)} &\leq \| m_\delta'(\phi_\delta) \nabla \phi_\delta \cdot \varphi\|_{L^2(\OT)} + \|m_\delta(\phi_\delta) \div \varphi\|_{L^2(\OT)} \\
			&\leq \|m_\delta'\|_{L^\infty(\R)} \|\nabla \phi_\delta\|_{L^2(\OT)} \|\varphi\|_{L^\infty(\OT)} + \|m_\delta\|_{L^\infty(\R)} \|\varphi\|_{L^2(0,T;V)}. \end{aligned} $$
			Then we have after integration by parts
	\begin{equation} \label{Eq:TestMuDeg}
	     -\int_0^T \! (\nabla \mu_\delta,m_\delta(\phi_\delta) \varphi)_H \dt = \int_0^T\! (\mu_\delta, \div(m_\delta(\phi_\delta) \varphi))_H \dt = \int_0^T \!  (\Psi_\delta'(\phi_\delta)-\eps^2\Delta \phi_\delta, \div(m_\delta(\phi_\delta) \varphi))_H \dt. 
	\end{equation}
	The left hand side of this equation is equal to $\int_0^T (J_\delta,\varphi)_H \dt$ and converges to $\int_0^T (J,\varphi)_H \dt$ for all $\varphi$ as $\delta \to 0$ due to the weak convergence of $J_\delta$. Hence, we also take the limit $\delta \to 0$ in the right hand side in order to match the weak form of $J$, i.e., we have to show
	$$\int_0^T  (\Psi_\delta'(\phi_\delta)-\eps^2\Delta \phi_\delta, \div(m_\delta(\phi_\delta) \varphi))_H \dt \longrightarrow \int_0^T (\Psi'(\phi) - \eps^2 \Delta \phi,\div(m(\phi) \varphi))_H \dt,$$
	as $\delta \to 0$.
	To do so, we rewrite the term on the left hand side as
		\begin{equation}\begin{aligned}   \int_\OT \Psi_{1,\delta}'(\phi_\delta) \div(m_\delta(\phi_\delta) \varphi) + \Psi_2'(\phi_\delta)\div(m_\delta(\phi_\delta) \varphi)  -\eps^2 \Delta \phi_\delta \div(m_\delta(\phi_\delta) \varphi) \dtx,	\end{aligned} \label{Eq:Deg3Terms}\end{equation}
		and take the limit $\delta \to 0$ in each of the three terms. 
	
	We begin with the second and third term. We split them once more by employing the weak product rule on $\div(m_\delta(\phi_\delta)\varphi)$. We have proven $\phi_\delta \to \phi$ a.e. in $\OT$ and $m_\delta \to \ov m$ uniformly as $\delta \to 0$ since
	$$|m_\delta(x)-\ov m(x)| \leq \max\{m(1-\delta),m(\delta-1)\} \longrightarrow 0,$$
	for all $x \in \R$ as $\delta \to 0$; note that $\ov m(\phi)=m(\phi)$ due to $|\phi|\leq 1$.
	Moreover, $\Psi_2'(\phi_\delta) \to \Psi_2'(\phi)$ a.e. by the continuity of $\Psi_2'$, $\Delta \phi_\delta \weak \Delta \phi$ weakly in $L^2(\OT)$ by \cref{Eq:Weak_Deg} and thus, we conclude by the Lebesgue dominated convergence theorem
	$$\begin{aligned}
	\int_\OT \Psi_2'(\phi_\delta) m_\delta(\phi_\delta) \div \varphi \dtx &\longrightarrow \int_\OT \Psi_2'(\phi) m(\phi) \div \varphi \dtx,
	\\
	\int_\OT \Delta \phi_\delta m_\delta(\phi_\delta) \div\varphi \dtx &\longrightarrow \int_\OT  \Delta \phi m(\phi) \div\varphi \dtx.
	\end{aligned}$$
	
	Next, we treat the other parts of the product formula, i.e., we have to pass to the limit in the terms involving $m_\delta'(\phi_\delta) \nabla \phi_\delta$. 
	Since $m_\delta' \to m_\delta$ uniformly as $\delta \to 0$, we have by the dominated convergence theorem $m_\delta'(\phi_\delta) \nabla \phi_\delta \to m'(\phi) \nabla \phi$ in $L^2(\OT)^d$ due to the strong convergence of $\nabla \phi_\delta$. Thus, we have as $\delta \to 0$
		$$\begin{aligned}
	\int_\OT \Psi_2'(\phi_\delta) m_\delta'(\phi_\delta) \nabla \phi_\delta \cdot \varphi \dtx &\longrightarrow \int_\OT \Psi_2'(\phi) m'(\phi) \nabla \phi \cdot \varphi \dtx,
	\\
	\int_\OT  \Delta \phi_\delta m_\delta'(\phi_\delta) \nabla \phi_\delta \cdot \varphi \dtx &\longrightarrow \int_\OT  \Delta \phi m'(\phi) \nabla \phi \cdot \varphi \dtx.
	\end{aligned}$$
	
		At this point, we only miss the first term of \cref{Eq:Deg3Terms}. We have after integration by parts
	$$\int_\OT \Psi_{1,\delta}''(\phi_\delta) m_\delta(\phi_\delta) \nabla \phi_\delta \cdot \varphi \dtx.$$
	The term $m_\delta \Psi''_{1,\delta}$ is uniformly bounded, and it holds $\nabla \phi_\delta \to \nabla \phi$ a.e. in $\OT$ according to \cref{Eq:Weak_Deg}. Therefore, we have to show
	$$m_\delta(\phi_\delta) \Psi_{1,\delta}''(\phi_\delta) \longrightarrow m(\phi) \Psi_1''(\phi) \text{ a.e. in } \OT,$$
	and we proceed as in \cite[p.416f]{elliott1996cahn}. If it holds $|\phi|<1$ a.e. in $\OT$, then the result follows from $m_\delta(\phi)=m(\phi)$ and $\Psi_{1,\delta}(\phi)=\Psi_1(\phi)$. Hence, we consider the case $\phi_\delta \to \phi=1$ a.e. in $\OT$. If it holds $\phi_\delta \geq 1-\delta$, then it gives
	$$m_\delta(\phi_\delta) \Psi_{1,\delta}''(\phi_\delta) = m(1-\delta) \Psi_{1}''(1-\delta) \longrightarrow m(1) \Psi_1''(1)=m(\phi) \Psi_1''(\phi),$$
	and finally, in the other case of $\phi_\delta \leq 1-\delta$, it yields
	$$m_\delta(\phi_\delta) \Psi_{1,\delta}''(\phi_\delta) = m(\phi_\delta) \Psi_{1}''(\phi_\delta) \longrightarrow m(\phi) \Psi_1''(\phi),$$
	which completes the proof.
	\end{proof}
\begin{remark}

The key challenge in obtaining the results for the fractional version is the absence of a chain rule inequality for semiconvex functionals with low regular functions. 

For Hilbert-valued functions, the fractional chain inequality in \cite[Proposition 2.1]{vergara2008lyapunov} for the special case $E(\cdot)=\frac{1}{2}|\cdot|^2$ can be extended to semiconvex functionals having some extra terms. Unfortunately, this result requires a regularity assumption on the composition $E(u)$. Using composition theorems of fractional order, see \cite[Theorem 5.3.4/1]{runst2011sobolev}, this assumption is satisfied for example if $u \in H^\alpha(0,T;L^1(\Omega))\cap L^\infty(0,T;V)$. Since the Faedo--Galerkin solution is of high enough regularity, this result would suffice in the discrete setting to get a lower bound for $\langle \p_t^\alpha \phi^k, \Psi'(\phi
^k) \rangle_V$.

Unfortunately in the continuous limit, we do not have that; we had to estimate the term $\langle \p_t^\alpha \phi_\delta, \Phi_\delta'(\phi_\delta) \rangle_V$ in \cref{Lem:DegMobEst} for the degenerate case. Here, we only have $\phi_\delta \in H^\alpha(0,T;V')$ in contrast to $\phi^k \in H^\alpha(0,T;H_k)$ in the discrete setting. Thus, we need to follow a different path. We are able to apply our new convolved version of the fractional chain inequality for these low regular functions, see \cref{Eq:ChainE-2} in \cref{Lem:Chain}. The $\delta$-uniform estimate on $\|\Phi_\delta(\phi_\delta)\|_{L^\infty(0,T;L^1(\Omega))}$ of \cref{Lem:DegMobEst} is a key result to derive the bound $|\phi|\leq 1$ a.e. in $\Omega_T$ which was then used throughout the proof of \cref{Thm:DegMob}. 

\end{remark}
\section{Applications and numerical simulations} \label{Sec:Numerics}

In our simulations, the time discretization is performed using a first order quadrature scheme. We show simulations of the Cahn--Hilliard equation applied to tumor growth and block copolymers.

\subsection{Time and space discretization schemes} \label{Sec:TimeSpace}

Let $t_n=n T/N$, $n\in\{0,1,\ldots,N\}$, be a subdivision of $[0,T]$ in $N$  intervals of size $\Delta t= T/N$.
We apply a convolution quadrature scheme to approximate the fractional time derivative of Caputo type by
\begin{equation}\label{Eq:cq}
    \p_t^{\alpha} \phi \approx  \frac{1}{(\Delta t)^{\alpha}}\sum_{j=0}^{N} b_j (\phi_{n-j}-\phi_0),
\end{equation}
where $\phi_{n-j}$ is the approximation to $\phi(t_{n-j})$, e.g., see \cite{Lubich86,Lubich88,diethelm2020good}. We observe in \cref{Eq:cq} the memory effect in form of the history from the previous time steps $\phi_{n-j}$.  We apply the Gr\"unwald--Letnikov approximation \cite{diethelm2010analysis,Dumitru12} to compute the quadrature weights $(b_j)_{j \geq 1}$ by the recursive formula
\begin{equation}\label{Eq:cqweights}
    b_0=1,\quad b_j=-\frac{\alpha-j+1}{j}b_{j-1}\quad \text{for }j\geq 1.
\end{equation}

Moreover, we use the classical energy splitting method for the potential $\Psi=\Psi_1+\Psi_2$, which provides unconditional stability in the case of $\alpha=1$, e.g., see \cite{elliott1993global}. That means we treat the expansive part $\Psi_1$ explicitly and the contractive part $\Psi_2$  implicitly.  Applying the scheme \eqref{Eq:cq}--\eqref{Eq:cqweights} to the time-fractional Cahn--Hilliard equation \cref{Eq:CahnHilliardFractional} and denoting by $(\phi_n,\mu_n) \approx (\phi(t_n),\mu(t_n))$ the approximate solution tuple at time $t_n$, $n\in\{1,\ldots,N\}$, we have
\begin{subequations}\label{DiscreteSystem}
\begin{align}
\frac{\sum_{j=0}^n b_j (\phi_{n-j}-\phi_{0})}{(\Delta t)^\alpha}&=\div \left(m(\phi_n) \nabla \mu_n\right) + f \\
    \mu_n&=\Psi_1'(\phi_{n-1})+\Psi_2'(\phi_{n})-\eps^2 \Delta \phi_n.
\end{align}
\end{subequations}

We use mixed Q1-Q1 linear finite elements for the space discretizaton. Namely, at the $n$-th time step, we look at the problem
\begin{equation} \label{Eq:CH_Numeric}
\begin{aligned}
\frac{b_0(\phi_{n}-\phi_0,\xi)_H}{(\Delta t)^\alpha}+(m(\phi_n) \nabla \mu_n,\nabla \xi)_H=& (f,\xi)_H-\frac{\sum_{j=1}^{n-1} b_j (\phi_{n-j}-\phi_{0})}{(\Delta t)^\alpha}  ,\\
(\mu_n,\zeta)_H-\eps^2(\nabla \phi_n,\nabla \zeta)_H-(\Psi_2'(\phi_n),\zeta)_H=&  (\Psi_1'(\phi_{n-1}),\zeta)_H,
\end{aligned}
\end{equation}
for test function $\xi, \zeta$. Hence, we are interested in a nonlinear, coupled algebraic system with the unknown tuple $(\phi_n,\mu_n)$. At each time step we solve this system with the Newton method. The procedure in this section has been implemented in FEniCS \cite{fenics} to obtain the numerical results shown in the next two subsections.

\subsection{Application in the self-assembly of block copolymers}

Lithography is a technology for fabricating nansoscale electronic devices. One uses directed self-assembly of block copolymers for the manufacturing, see \cite{bates2000block}. Block copolymers are composed of chemically-dissimilar polymer chains with covalently linked monomers. The immisicibility of the polymers blends results in a phase separation on a mesoscopic scale, i.e., the length scale is around 5--20 nanometers. This is described by a modification of the Ginzburg--Landau energy functional, also called Ohta--Kawasaki energy \cite{ohta1986equilibrium},
	\begin{equation*}
 \int_\Omega \Psi(\phi) + \frac{\eps^2}{2} |\nabla \phi|^2 +\frac{\kappa}{2} |(-\Delta)^{-1/2} (\phi-\m)|^2 \, \dd x,
	\end{equation*}
where $\Delta^{-1/2}$ is the fractional inverse Laplacian of order $\tfrac12$, $\text{m}=\int_\Omega \phi \, \dd x$ the mass of $\phi$, and $\kappa$ a parameter for the nonlocal long-range interactions. Here, $\phi$ describes the difference of the volume fractions for the two copolymers.
Note that the G\^ateuax derivative of the new part of the energy is given by
$$\begin{aligned}\frac{\dd}{\dd\theta}\bigg|_{\theta=0} \! \int_\Omega \! \frac{\kappa}{2} |(-\Delta)^{-1/2} (\phi+\theta v-\m)|^2 \dx  &=\!\! \int_\Omega \! \kappa ((-\Delta)^{-1/2} (\phi -\m)) (-\Delta)^{-1/2} v \dx = -\int_\Omega \! \kappa (-\Delta)^{-1} (\phi -\m) v \dx, \end{aligned} $$
and consequently, the system reads
\begin{equation*}
	\begin{aligned}
	\p_t^{\alpha}\phi &= \div( m(\phi) \nabla\mu), \\ 
	\mu &= \Psi'(\phi) - \eps^2 \Delta \phi - \kappa \nu, \\
	-\Delta \nu &= \phi - \text{m}.
	\end{aligned}
\end{equation*}
Note that this system is volume-conserving because integrating with the test function $\xi=1$ gives
$\int_\Omega \p_t^\alpha \phi \, \dx =  0$ and thus, after applying the inverse kernel with a convolution and taking the time derivative, it yields $\int_\Omega \phi(t,x) \dx = \int_\Omega \phi(0,x) \dx$.
Thus, the nonlocal mass $\text{m}$ is given by the constant value $\int_\Omega \phi_0 \, \dd x$. If one assumes a constant mobility function $m(\phi)=M$, it gives the simplified system
\begin{equation}\label{Eq:CHOhta}
	\begin{aligned}
	\p_t^{\alpha}\phi &= M\Delta\mu -M\kappa (\phi-\text{m}),\\
	\mu &= \Psi'(\phi) - \eps^2 \Delta \phi.
	\end{aligned}
\end{equation}

We apply the time and space discretizations to \cref{Eq:CHOhta} as described in \cref{Sec:TimeSpace}, and treat the linear source term implicitly. Let $\Omega=(0,1)^3$ be the three-dimensional space domain, which we equip with a uniform hexahedral mesh with mesh size $h=2^{-7}$. Further, we consider the time domain $[0,0.2]$ with $\Delta t=10^{-4}$. As initial data we take 
\begin{equation} \label{Eq:CoInitial}
\phi_0(x)=0.4+\frac{\cos(2\pi x_1) \cos(2\pi x_2) \cos(2\pi x_3)}{100},
\end{equation}
see \cref{Fig:Initial} for a visualization on $\Omega$ and on two intersecting planes inside the box domain. In the following simulations, we set the parameters to $\kappa=100$, $\eps=5 \cdot 10^{-4}$, and $M=1$. Further, we select the double-well potential $\Psi(\phi)=\frac12 (1-\phi^2)^2$ with zeros at $\pm 1$. \vspace{0.3cm}

\begin{figure}[H]
    \centering
        \includegraphics[trim=5cm 6.25cm 4.3cm 7cm, clip,height=.31\textheight]{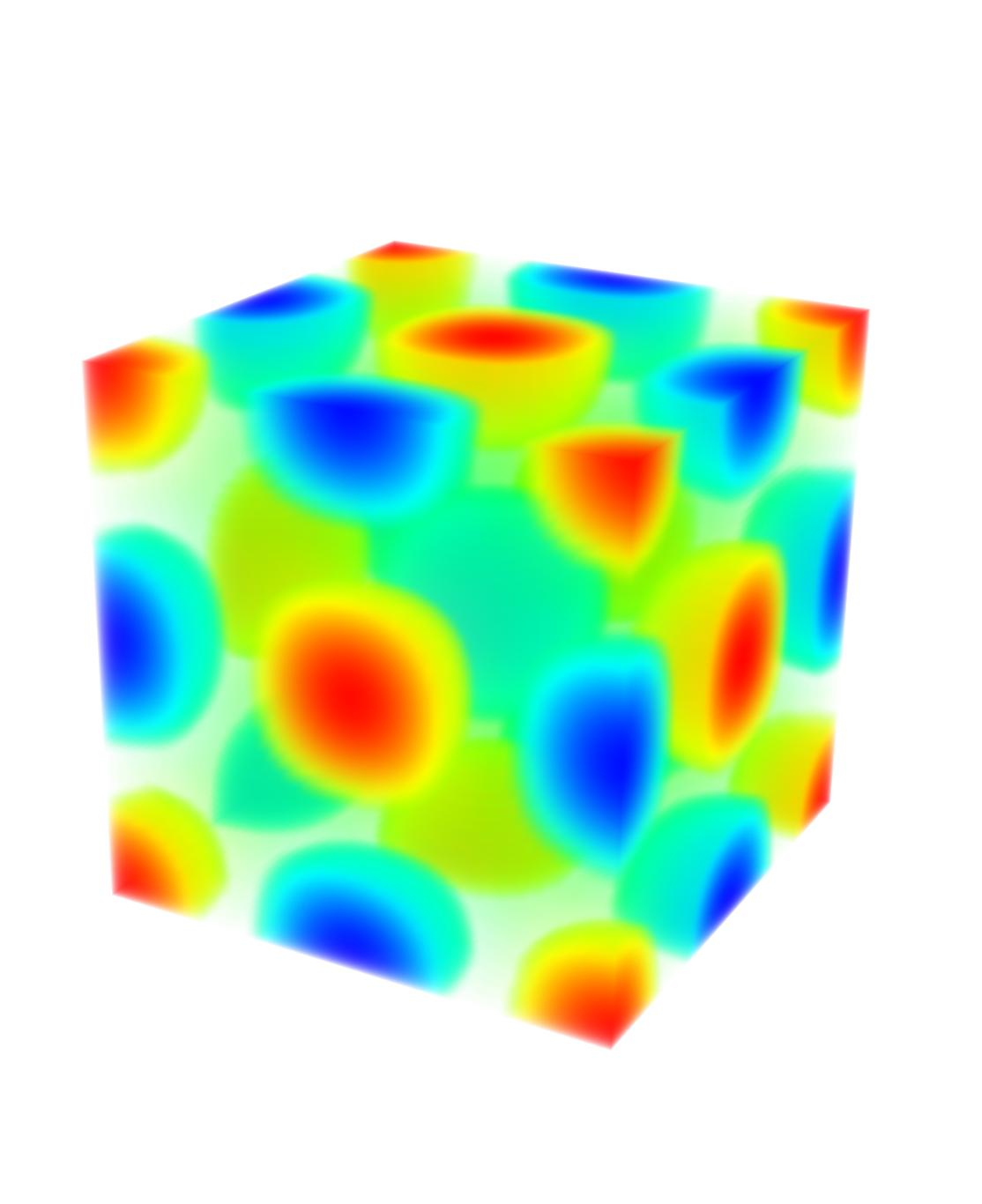} \qquad
    \includegraphics[trim=4cm 4.8cm 2.5cm 3cm, clip,height=.3\textheight]{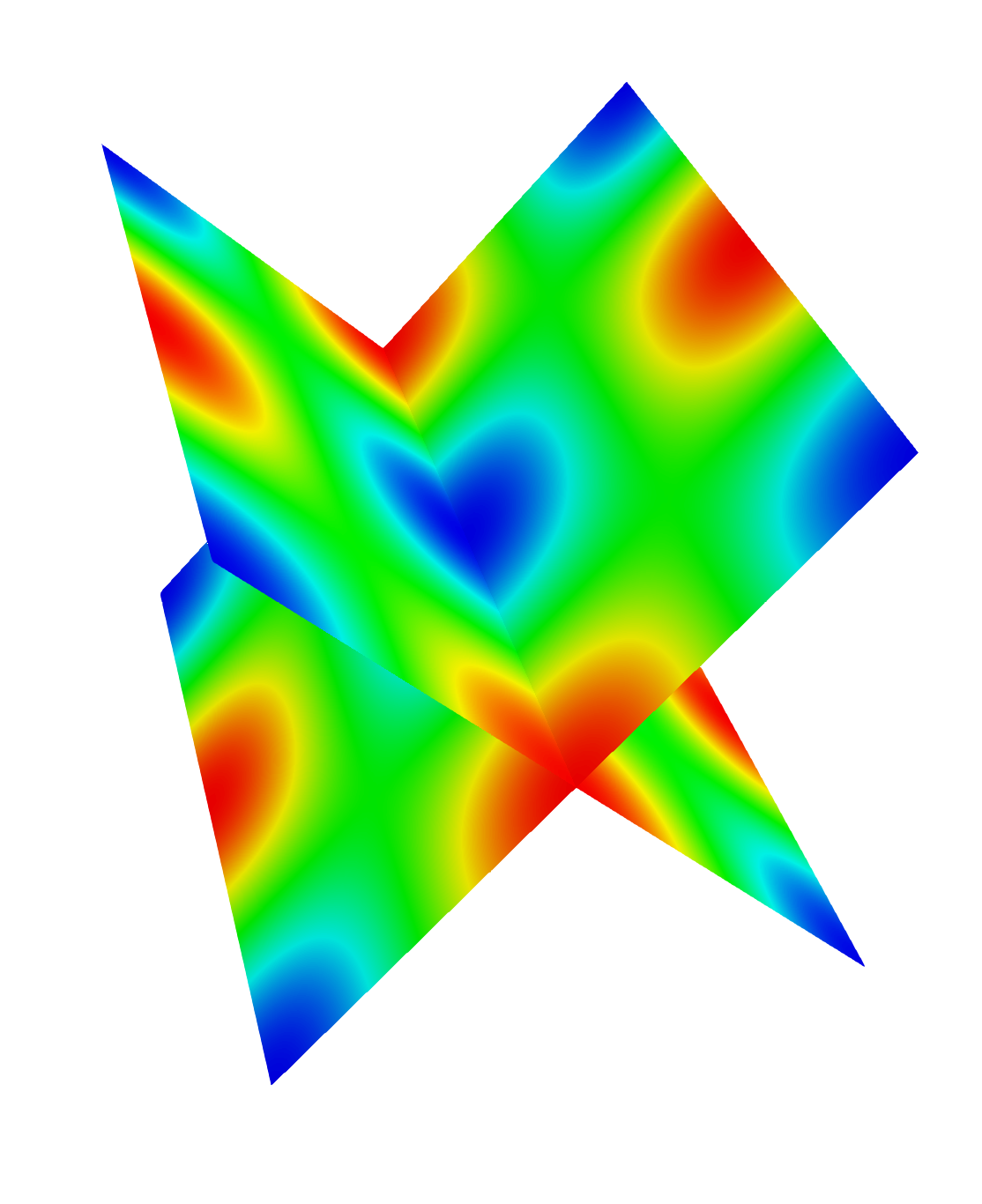} \\
    \includegraphics[width=.5\textwidth]{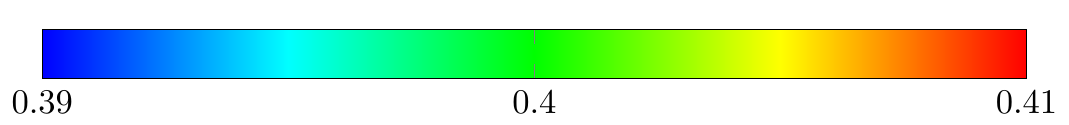}
    \vspace{-.25cm}
    \caption{Visualization of the initial condition $\phi_0$ as given in \cref{Eq:CoInitial} in the box $[0,1]^3$ (left) and on the two planes (right) with the normals $(1,0,0)$ and $(0,1,0)$, respectively.}
    \label{Fig:Initial}
\end{figure}

\pagebreak

In \cref{Fig:Copoly}, we show the evolution of the field $\phi$ for two different values of $\alpha$; we take $\alpha \in \{0.1,0.75\}$. Again, we depict the field on the two intersecting planes.

First, we notice a difference in the speed of the evolution of $\phi$.
For the larger value $\alpha=0.7$, the field at $t=0.05$ is already close to its state at the later time point $t=0.2$, whereas for $\alpha=0.1$ it is still in its evolution at $t=0.05$. This behavior is in accordance to the observations in \cite{ji2020adaptive,tang2019energy,chen2001derivation}. Even though smaller $\alpha$ values have a faster initial evolution, it takes more time to reach the equilibrium state of the system.

At $t=0.2$, we observe for both values of $\alpha$ that $\phi$ mostly attains the values of  $-1$ and $1$, and in between it admits a  smooth transition zone. Further, we notice that the solutions of the two $\alpha$ values are different at $t=0.2$. Consequently, we can conclude that the fractional power $\alpha$ has a large influence on the asymptotic behavior of the solution.

\setlength\tabcolsep{1pt}
\begin{figure}[H] 
\centering 
	{}\hspace{-.2cm}\begin{tabular}{ccccc}
		& {}\hspace{-.5cm}$t=0.015$ & {}\hspace{-.5cm}$t=0.035$ & {}\hspace{-.5cm}$t=0.05$ & {}\hspace{-.5cm}$t=0.2$   \\
	\rotatebox{90}{{}\qquad ~\qquad$\alpha=0.1$}\, &
	\includegraphics[trim=4cm 4.8cm 2.5cm 3cm, clip,width=.22\textwidth]{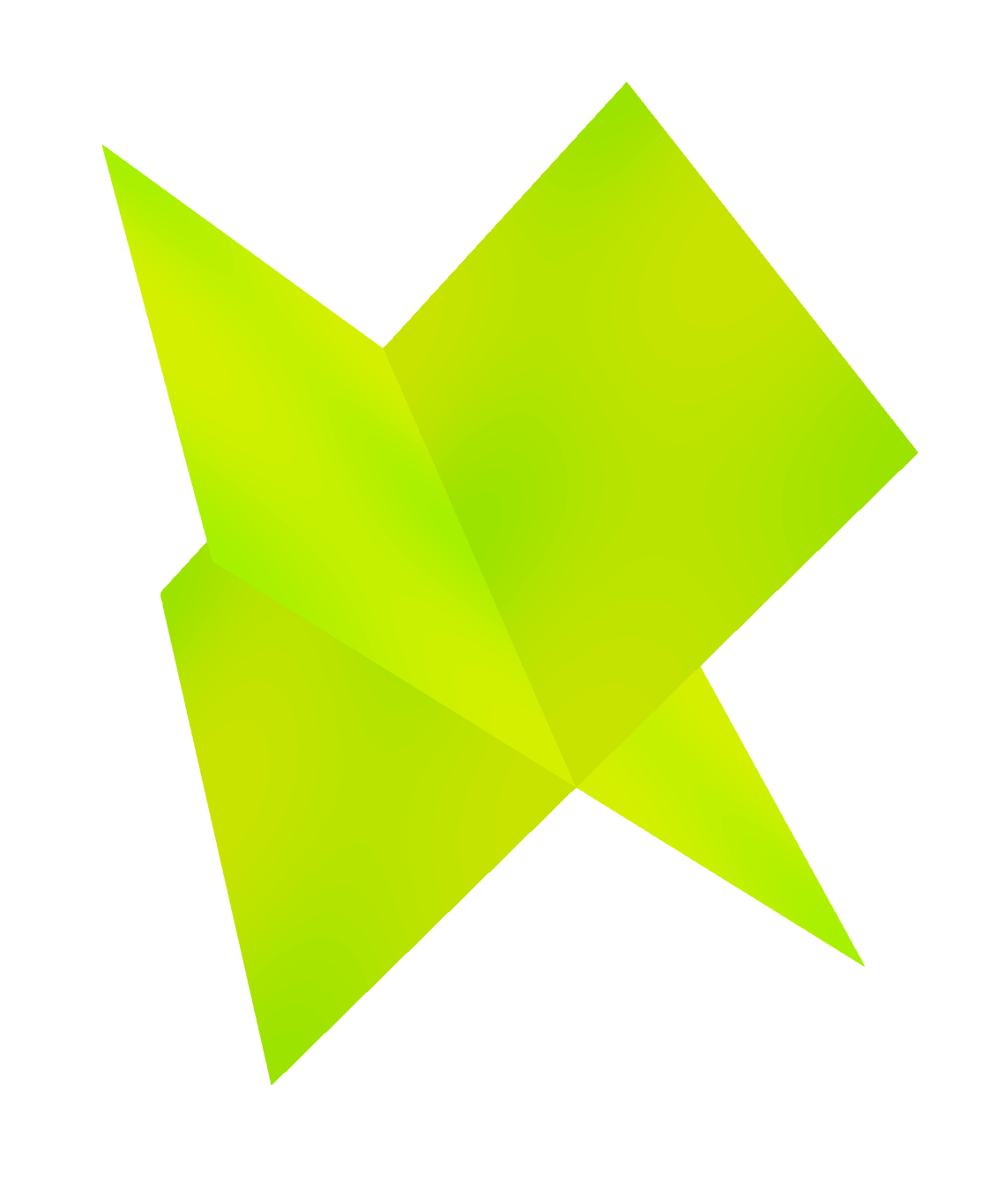} &
	\includegraphics[trim=4cm 4.8cm 2.5cm 3cm, clip,width=.22\textwidth]{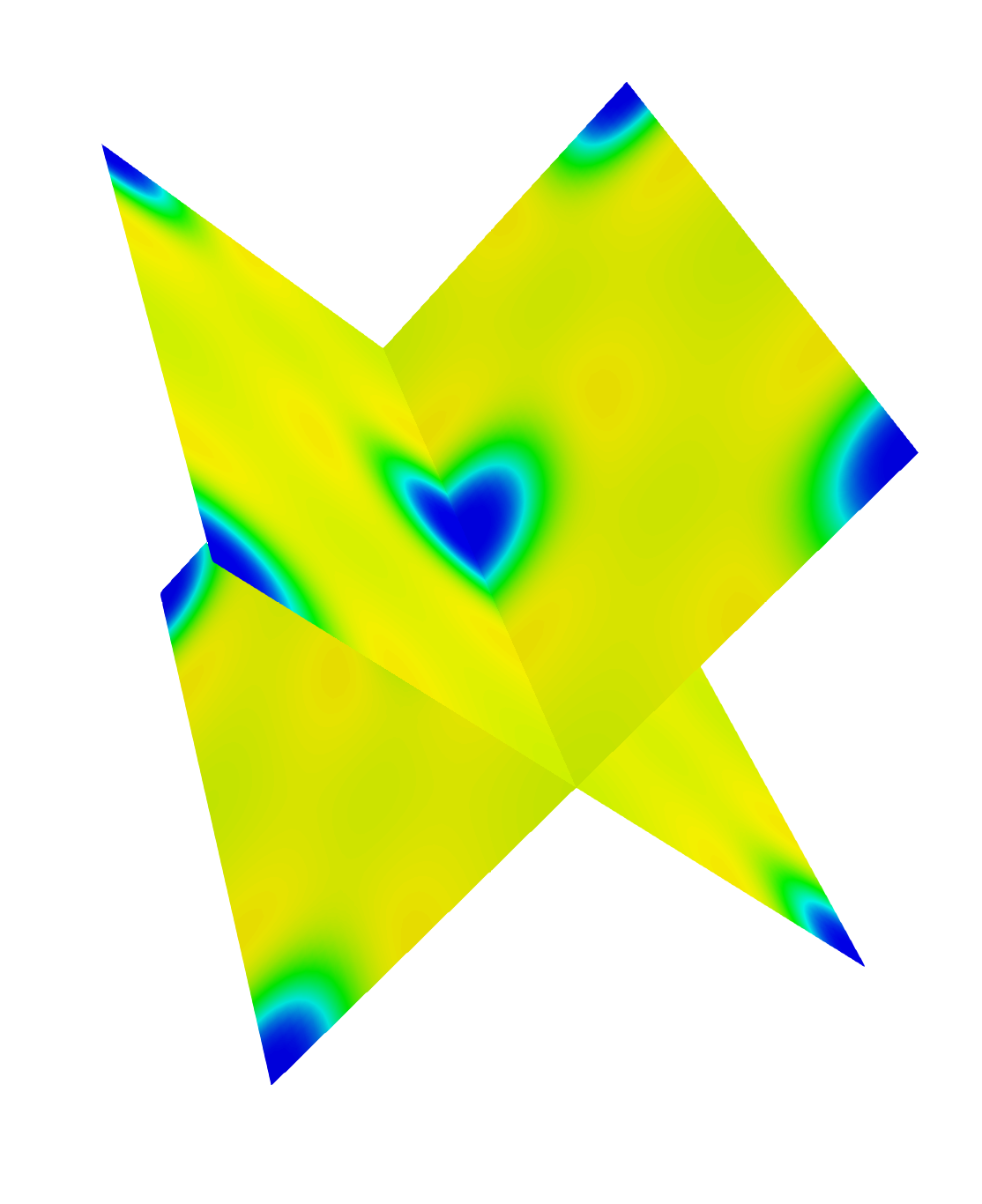} &
	\includegraphics[trim=4cm 4.8cm 2.5cm 3cm, clip,width=.22\textwidth]{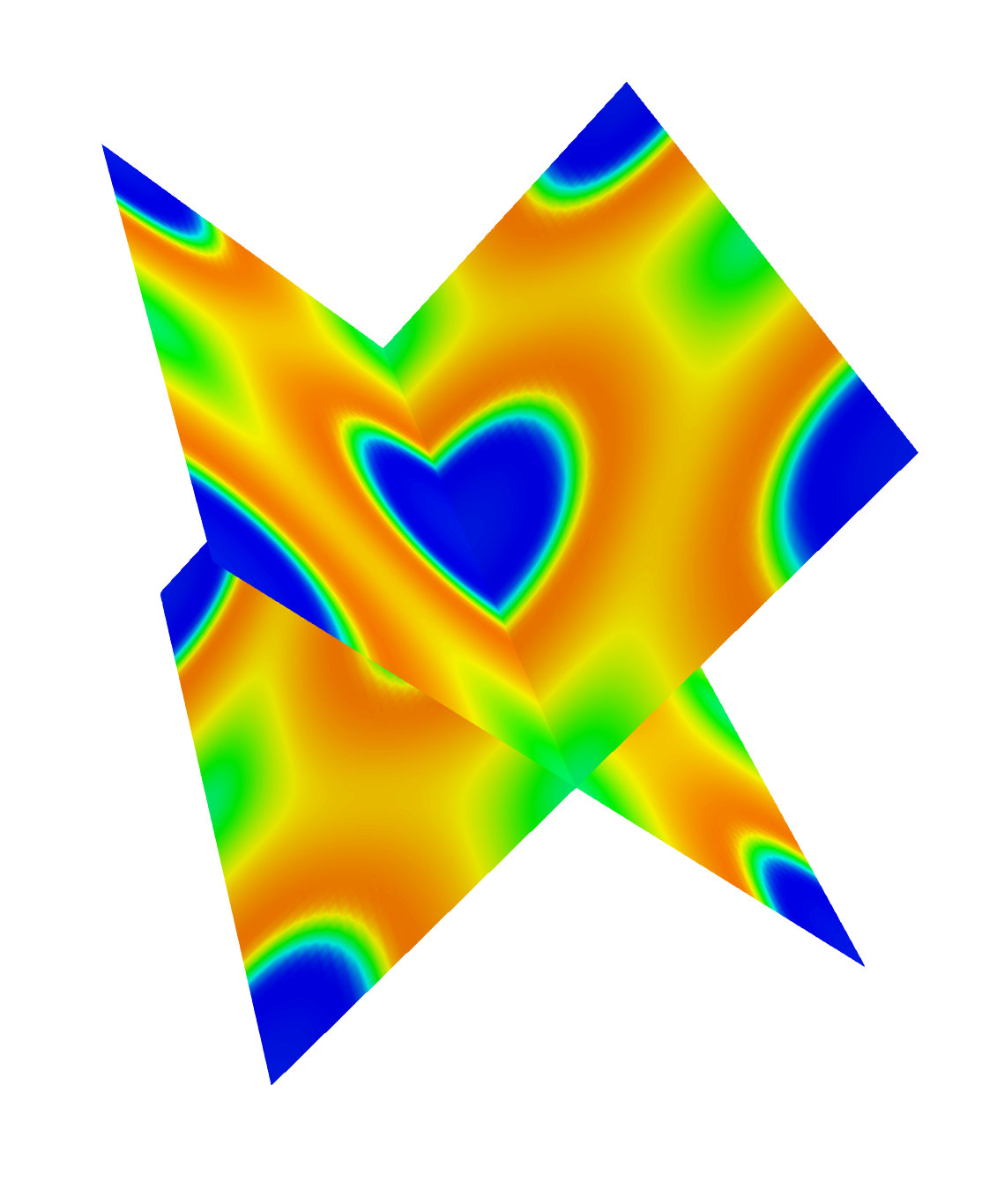} &
	\includegraphics[trim=4cm 4.8cm 2.5cm 3cm, clip,width=.22\textwidth]{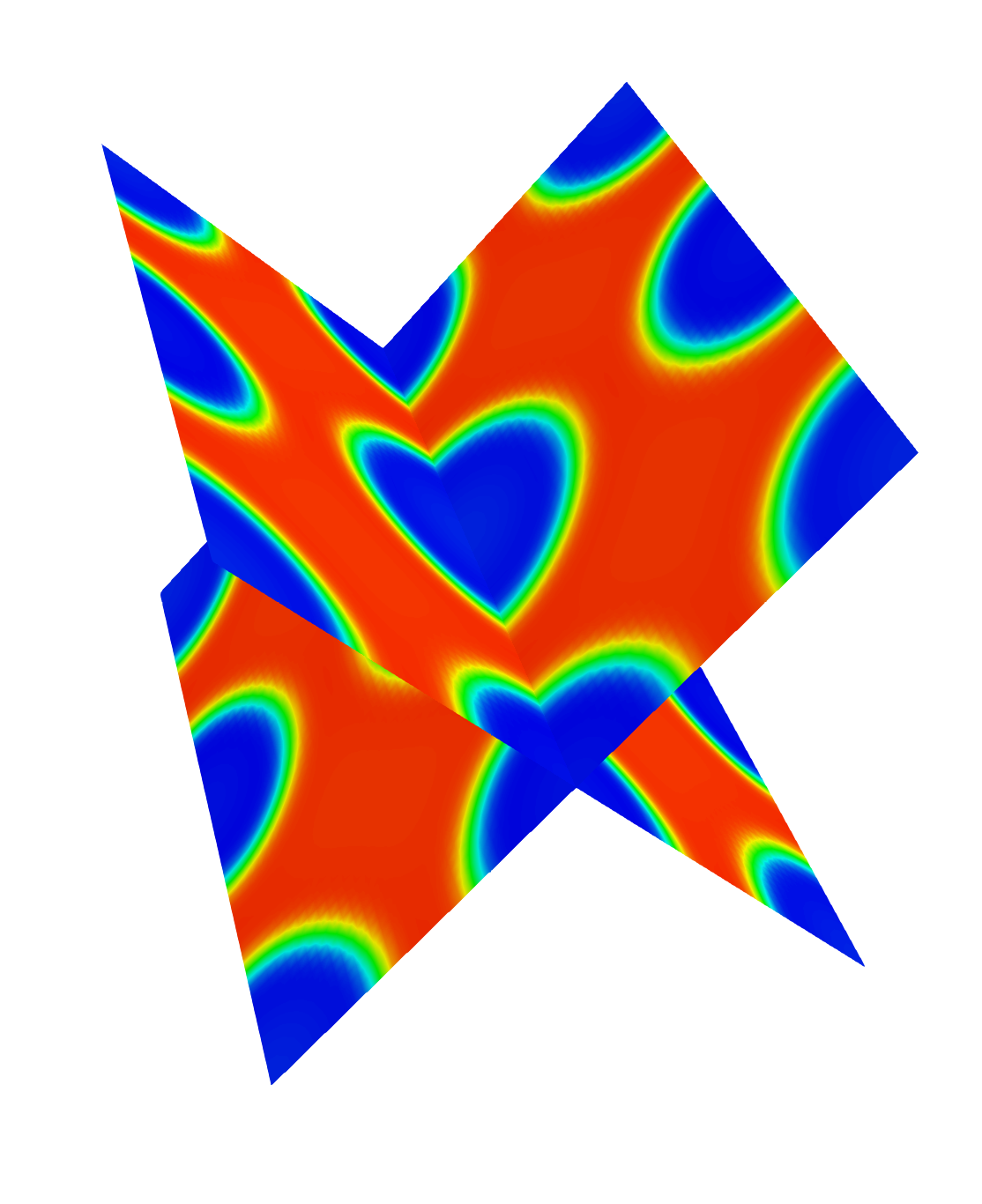}  \\
	\rotatebox{90}{{}\qquad~ \qquad$\alpha=0.75$}\, &	\includegraphics[trim=4cm 4.8cm 2.5cm 3cm, clip,width=.22\textwidth]{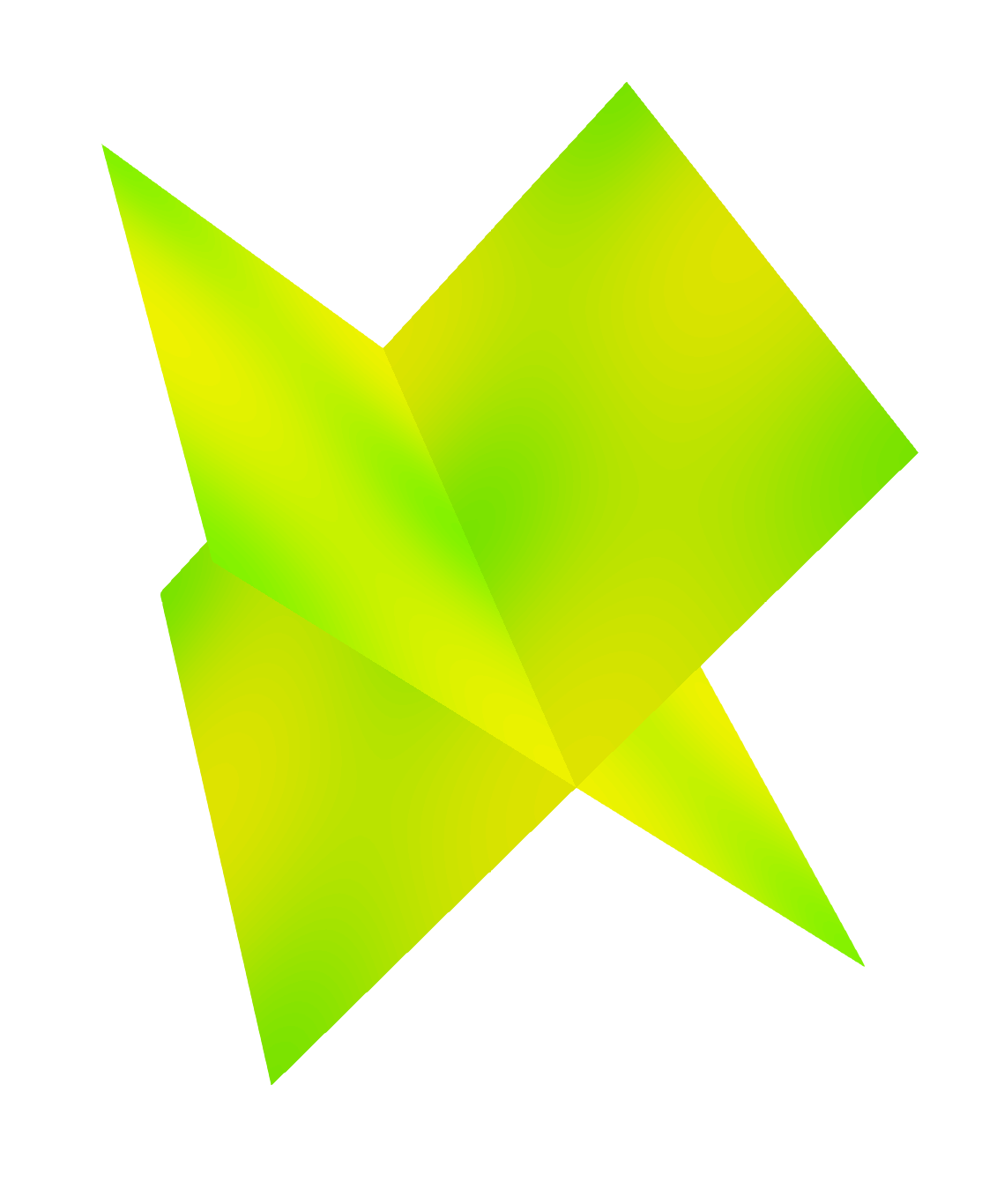} &
	\includegraphics[trim=4cm 4.8cm 2.5cm 3cm, clip,width=.22\textwidth]{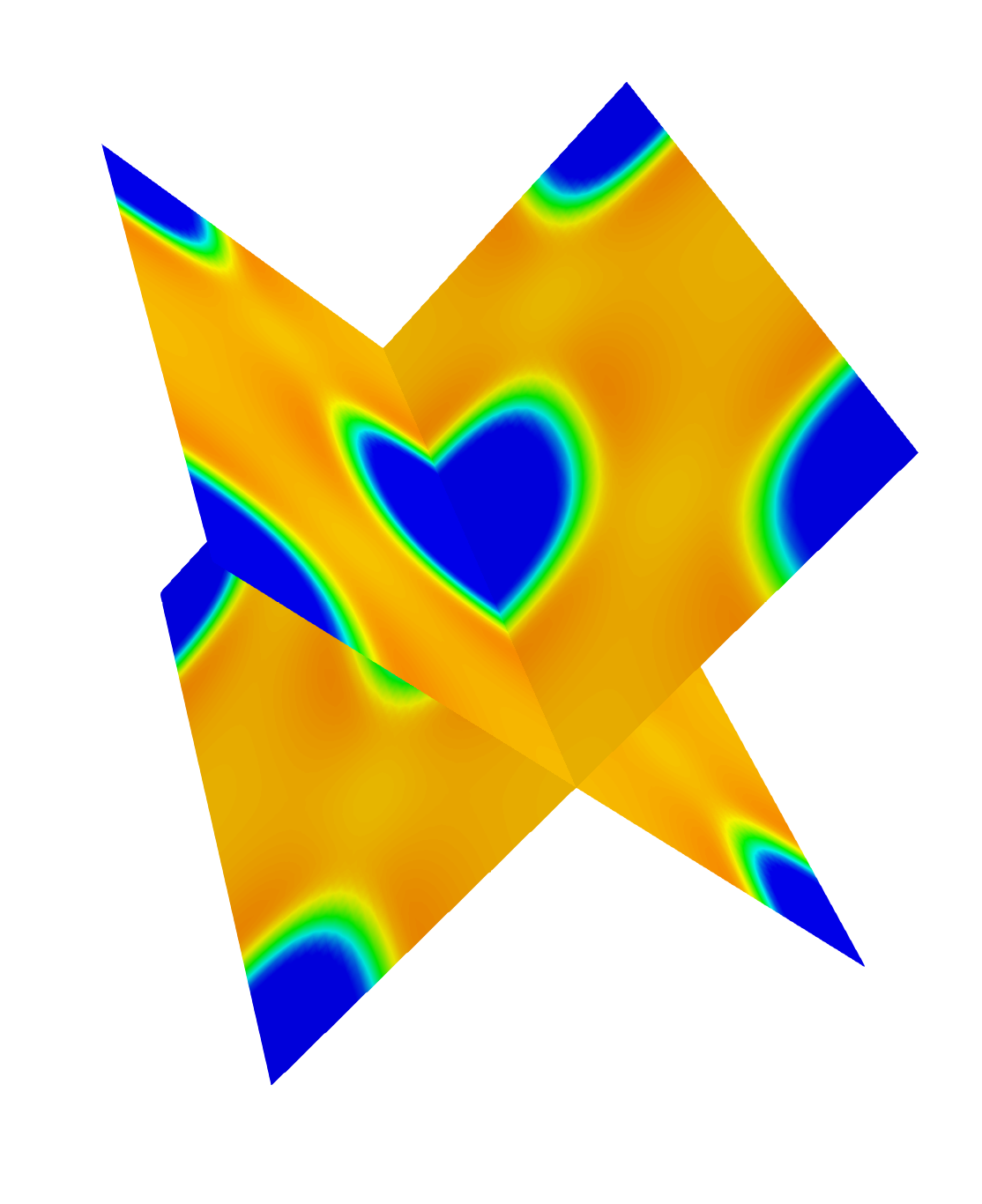} &
	\includegraphics[trim=4cm 4.8cm 2.5cm 3cm, clip,width=.22\textwidth]{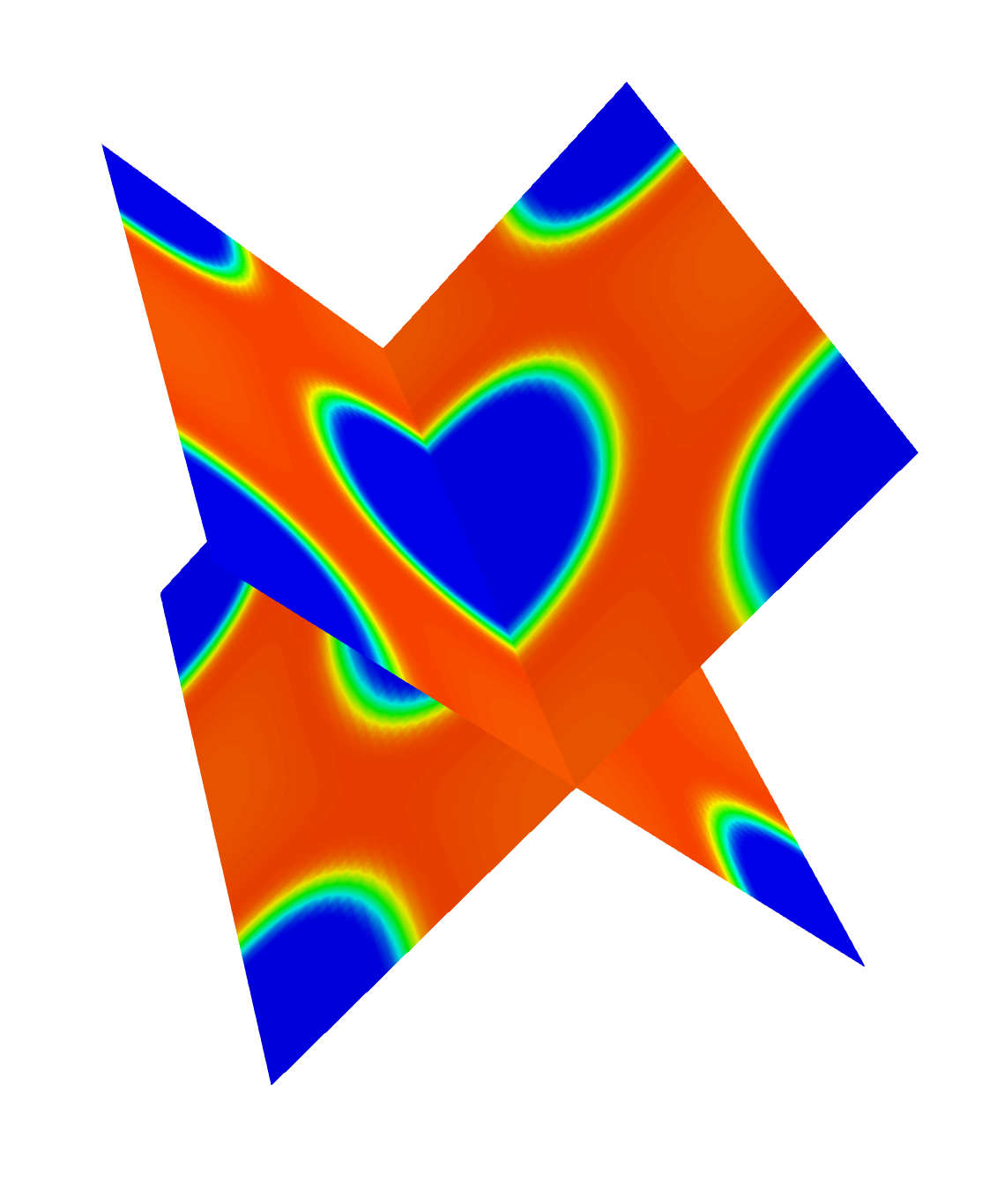} &
	\includegraphics[trim=4cm 4.8cm 2.5cm 3cm, clip,width=.22\textwidth]{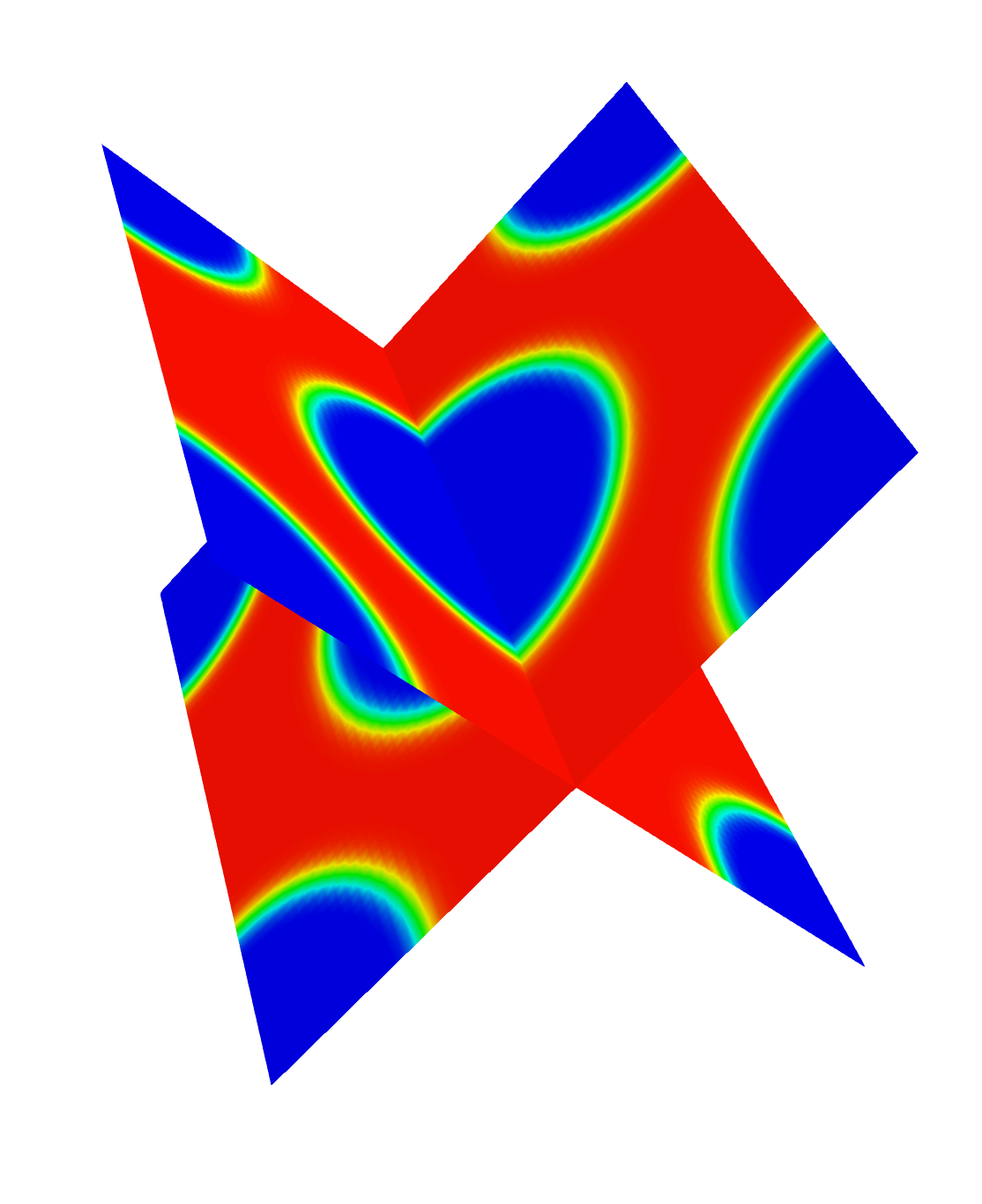}  
	\end{tabular} \vspace{-.2cm}
	
	\hspace{-0.02cm}\includegraphics[width=.55\textwidth]{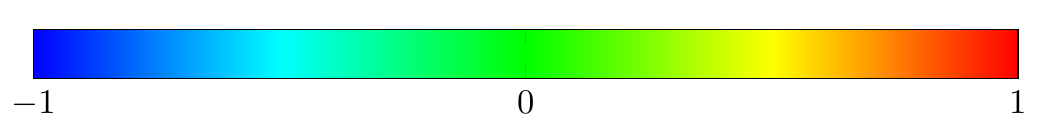}
	\caption{Evolution of block copolymers for $\alpha\in\{0.1,0.75\}$ (top, bottom) at the time spots $t \in \{0.015, 0.035, 0.05, 0.2\}$ (from left to right), again visualization on the intersecting planes with normals $(1,0,0)$ and $(0,1,0)$, respectively. \label{Fig:Copoly}}
\end{figure} 

\subsection{Application in subdiffusive tumor growth}

In this subsection, we investigate the time-fractional Cahn--Hilliard equation in an application to subdiffusive tumor growth. First, we motivate the model from mathematical modeling and afterwards, we treat the system numerically. In this regard, we do a sensitivity analysis on the model parameters including the fractional exponent $\alpha$.

\subsubsection{Modeling} It was shown in \cite{zhao2019power} that the free energy functional $\mathcal{E}$ and the interface roughness $W$ of the time-fractional Cahn--Hilliard equation follows a power law, whose power is proportional to the fractional order of the partial differential equation. In particular, 
	$$\mathcal{E}(\phi(t))\propto t^{\beta(\alpha)},\quad W(\phi(t))=\sqrt{\frac{1}{|\Omega|}\int_\Omega\left(\phi-\text{m} \right)^2\dd x}\propto t^{R(\alpha)}.$$
	It was reported in \cite{jiang2014anomalous} that the roughness of the peripheral border of tumor increased when subjected to haptotaxis or chemotaxis stimuli from the extracellular matrix or nutrients. This was shown by calculating border fractal dimension of clinical tumor using medical images which served as a measure for calculating the roughness of the interface. Further it was shown in \cite{bru2008fractal,bru2003universal} that the fluctuations of the interface between a tumor and its host follows a power law behavior. 

	This suggests that the time-fractional Cahn--Hilliard equation is suitable for describing the process of tumor growth and decline, as done similarly for the integer order case, e.g., see \cite{fritz2018unsteady,fritz2019local}. Let $\phi$ denote the tumor volume concentration, i.e., if a tumor cell is located at $x \in \Omega$, we have $\phi(x)=1$ and otherwise, $\phi(x)=-1$. In between a smooth interface marks the transition from zero to one. Moreover, $\sigma$ describes the nutrient-rich extracellular water, which provides the tumor cells with sufficient nutrients to grow. The Ginzburg--Landau energy with chemotaxis is given by
	\begin{equation*}
	\int_\Omega \Psi(\phi) + \frac{\eps^2}{2} |\nabla \phi|^2 - \chi \phi \sigma \, \dd x,
	\end{equation*}
	where $\chi$ is the parameter of chemotaxis, i.e., the adhesion of tumor cells and nutrients.
	
	We propose the following tumor growth model:
	\begin{equation}\label{Eq:Tumor}\begin{aligned}
	\p_t^\alpha \phi  &= \div(m(\phi) \nabla \mu) + \lambda \phi(1-\phi) \sigma - \delta \phi, \\
	\mu &= \Psi'(\phi) - \eps^2 \Delta \phi - \chi \sigma, \\	\p_t \sigma  &= D \Delta \sigma -D \chi \Delta \phi - \lambda \phi(1-\phi) \sigma + \delta \phi.
	\end{aligned}\end{equation}
	In this model, $\lambda$ is a proliferation factor of the tumor cells due to available nutrients, $\delta$ an degradation factor describing apoptosis, i.e., natural cell death, and $D$ the diffusion parameter of the nutrients. The existence of weak solutions for linearized source terms can be investigated similar to \cref{Sec:AnalysisCahn}; for the integer-order case we refer to \cite{garcke2017well}.

	\subsubsection{Numerical simulation and sensitivity analysis}
	In this section, we investigate the sensitivity of $\alpha$ on the tumor mass. 
	We apply the time and space discretizations as described in \cref{Sec:Numerics}. Let $\Omega=(0,1)$ be the one-dimensional space domain, which we equip with a uniform mesh with mesh size $h=5\cdot 10^{-3}$. Further, we consider the time domain $[0,2]$ with $\Delta t=10^{-3}$. 
	
	We select the mobility function $m(\phi)=M(1-\phi^2)^2$ and the Landau potential $\Psi(\phi)=C_\Psi(1-\phi^2)^2$.
	We assume an initial nutrient concentration of $\phi_\sigma=1$, and we place the initial tumor in the interval $(\frac25,\frac35)$, i.e., we set $\phi=1$ in the interval and $-1$ otherwise. We take a smooth interface to guarantee the $H^1(\Omega)$-regularity of the initial data for the existence result of \cref{Thm:DegMob}. E.g., one can choose the initial condition $$\phi_0(x)=-1+2\cdot \bbone_{\big(\tfrac25,\tfrac35\big)}\exp\left(1-\frac{1}{1-100|x-\tfrac12|^2} \right).$$

The relative effects of model parameters in determining key quantities of interest, such as the evolution of tumor mass over time, are very important in the development of predictive models of tumor growth. Accordingly, in this section we address the question of sensitivity of solutions to variations in the model parameters $$\theta = \left(\alpha, M, \lambda, \delta, C_\Psi, \eps, \chi, D\right) \in \R^8,$$  and we provide a sensitivity analysis using the
variance-based method, developed by \cite{sobol2001global}, and described in detail in the book \cite{saltelli2008global}. The variance-based method takes uncertainties from the
input factors into account, showing the dependency of the variance of the output on the uncertainties. 

As the quantity of interest in the sensitivity analysis, we choose the volume of the tumor
mass at different times $t\in \mathcal{T}$, i.e., the $\text{dim}(\mathcal{T})$-dimensional vector
$Q(\theta)=[\int_\Omega \phi(t,x) \dx]_{t \in \mathcal{T}},$
and we choose the following uniformly distributed priors,
$$\begin{alignedat}{7} \alpha &\sim \mathcal{U}(0.001,1), ~ &M &\sim \mathcal{U}(0.1,1), ~&\lambda &\sim \mathcal{U}(0.1,1), ~&\delta &\sim \mathcal{U}(0.001,0.01),\\ C_\Psi &\sim \mathcal{U}(0.025,2.5), ~&\eps &\sim \mathcal{U}(0.01,0.1), ~&\chi &\sim \mathcal{U}(0.01,0.5), ~&D &\sim \mathcal{U}(0.1,1).\end{alignedat}$$

In the variance-based method the symbol $S_i$ represents the sensitivity of the $i$-th parameter (also called: Sobol sensitivity index) and it is calculated by the formula, e.g., see \cite{saltelli2008global},
$$S_i=\frac{\mathbb{V}(\mathbb{E}(Q(\theta)|\theta_i))}{\mathbb{V}(Q(\theta))},$$
where $\mathbb{V}$ denotes the variation and $\mathbb{E}(Q(\theta)|\theta_i)$ is the expected value of the output $Q(\theta)$ when parameter $\theta_i$ is fixed. Mathematically, the $i$-th sensitivity index $S_i$ reflects the expected reduction in the variance of the model when the $i$-th parameter $\theta_i$ is fixed. We use the Monte Carlo method to approximate the sensitivity indices.
One generates two matrices $A,B\in \R^{N\times k}$, $N$ being the number of samples and $k$ being the number of parameters (here: $k=8$), where
each row of each matrix represents one set of values from the vector of parameters sampled from the priors. Further, one generates $k$ matrices $C_i$, where the $i$-th column comes from matrix $B$ and all other from matrix $A$.
The output for all the sample matrices are computed, i.e., $Q(A)$, $Q(B)$, $Q(C_i) \in \R^N$, where each line of the vectors represents the quantity of interest with the parameter of the respective row of the matrix. Lastly, one approximates the sensitivity of the $i$-th parameter via the formula, see \cite{saltelli2008global},
\begin{equation*}
S_i = \frac{Q(A) \cdot Q(C_i) - \frac{1}{N} \left(\sum_{n=1}^N Q(A)^{(n)} \right)^2}{ Q(A) \cdot Q(B) - \frac{1}{N}\left( \sum_{n=1}^N Q(A)^{(n)} \right)^2}.
\end{equation*}
These indices are always between $0$ and $1$. High values of $S_i$ indicate a sensitive parameter, and low values, for additive models, indicate a less-sensitive parameter. 

The result of the variance-based method applied to \cref{Eq:Tumor} with the given priors, $N=100$, and the mass as the QoI is given in \cref{Fig:Sensitivity}. We see that $\alpha$ and $\lambda$ are the dominant parameters in the influence to the tumor mass. Since we chose the mass as the QoI, we could have expected that the proliferation parameter $\lambda$ will be highly sensitive. The fractional parameter $\alpha$ might be more surprising. Therefore, we depict the tumor mass for different values of $\alpha$ in \cref{Fig:Sensitivity}. We see that small $\alpha$-values have an instantaneous effect and a subdiffusive behavior can be observed. In the case of integer-order $\alpha=1$, we notice an almost linear mass growth.

\begin{figure}[H] 
\centering
\begin{subfigure}[t]{0.494\textwidth}
    \centering
    \includegraphics[width=\textwidth]{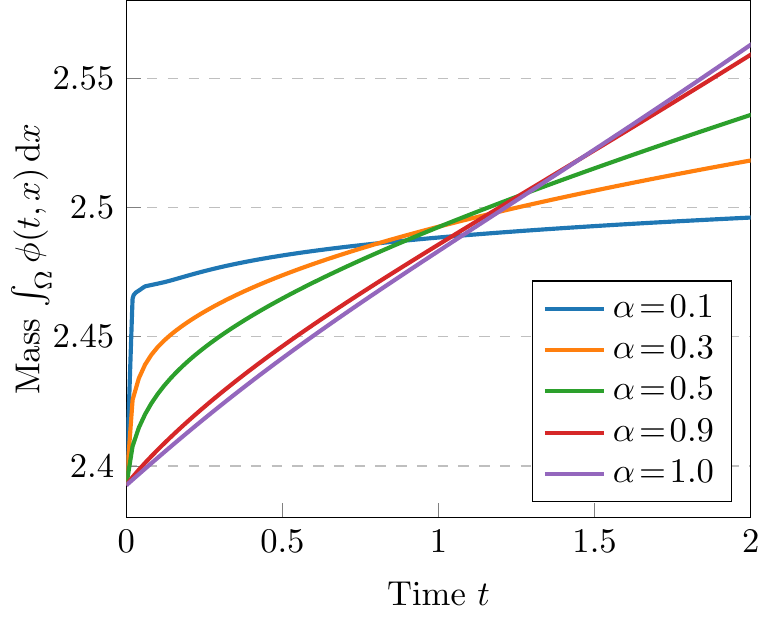}
\end{subfigure}
\begin{subfigure}[t]{0.475\textwidth}
    \centering
    \includegraphics[width=\textwidth]{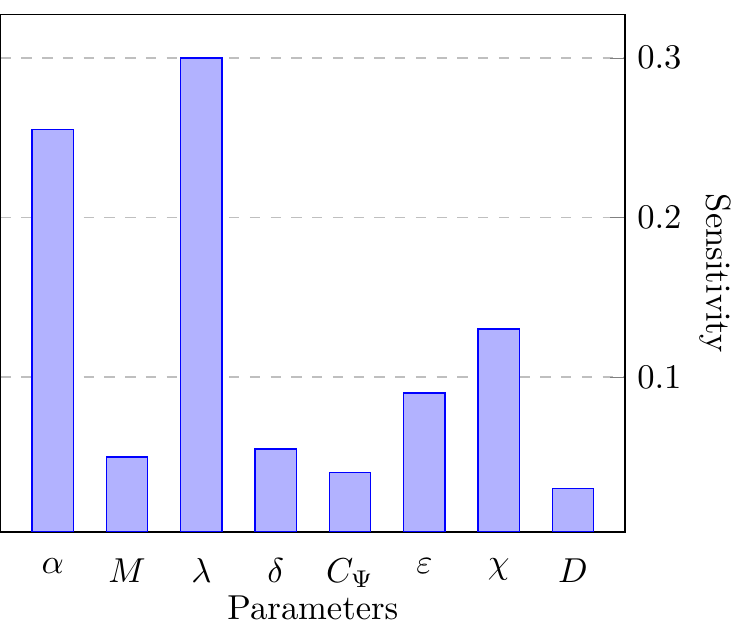}
\end{subfigure} 

\caption{Left: Evolution of the tumor mass for different values of $\alpha$. Right: Sensitivities $S_{i}$. \label{Fig:Sensitivity}}
\end{figure}

\section*{Acknowledgements}
The authors gratefully acknowledge the support from TUM
International Graduate School of Science and
Engineering (IGSSE). MLR acknowledges support from the Laura Bassi Postdoctoral Fellowship (Technical University of Munich). MF and BW were partially funded by DFG, WO-671 11-1.  

\bibliographystyle{siam}

\begin{thebibliography}{100}
	
	\bibitem{abels2013incompressible}
	{\sc H.~Abels, D.~Depner, and H.~Garcke}, {\em On an incompressible
		{N}avier--{S}tokes/{C}ahn--{H}illiard system with degenerate mobility},
	Annales de l'IHP Analyse Non Lin{\'e}aire, 30 (2013), pp.~1175--1190.
	
	\bibitem{agmon1959estimates}
	{\sc S.~Agmon, A.~Douglis, and L.~Nirenberg}, {\em Estimates near the boundary
		for solutions of elliptic partial differential equations satisfying general
		boundary conditions {I}}, Communications on pure and applied mathematics, 12
	(1959), pp.~623--727.
	
	\bibitem{akinyemi2020iterative}
	{\sc L.~Akinyemi, O.~S. Iyiola, and U.~Akpan}, {\em Iterative methods for
		solving fourth-and sixth-order time-fractional {C}ahn--{H}illard equation},
	Mathematical Methods in the Applied Sciences, 43 (2020), pp.~4050--4074.
	
	\bibitem{allen1972ground}
	{\sc S.~M. Allen and J.~W. Cahn}, {\em Ground state structures in ordered
		binary alloys with second neighbor interactions}, Acta Metallurgica, 20
	(1972), pp.~423--433.
	
	\bibitem{fenics}
	{\sc M.~Aln{\ae}s, J.~Blechta, J.~Hake, A.~Johansson, B.~Kehlet, A.~Logg,
		C.~Richardson, J.~Ring, M.~E. Rognes, and G.~N. Wells}, {\em {The FEniCS
			project version 1.5}}, Archive of Numerical Software, 3 (2015).
	
	\bibitem{bagley1983theoretical}
	{\sc R.~L. Bagley and P.~Torvik}, {\em A theoretical basis for the application
		of fractional calculus to viscoelasticity}, Journal of Rheology, 27 (1983),
	pp.~201--210.
	
	\bibitem{bagley1986fractional}
	{\sc R.~L. Bagley and P.~J. Torvik}, {\em On the fractional calculus model of
		viscoelastic behavior}, Journal of Rheology, 30 (1986), pp.~133--155.
	
	\bibitem{bai2007fractional}
	{\sc J.~Bai and X.-C. Feng}, {\em Fractional-order anisotropic diffusion for
		image denoising}, IEEE transactions on Image Processing, 16 (2007),
	pp.~2492--2502.
	
	\bibitem{bates2000block}
	{\sc F.~S. Bates and G.~Fredrickson}, {\em Block copolymers-designer soft
		materials}, Physics Today, 52 (2000).
	
	\bibitem{bates2013global}
	{\sc P.~W. Bates and J.~Jin}, {\em Global dynamics of boundary droplets},
	Discrete \& Continuous Dynamical Systems-A, 34 (2014).
	
	\bibitem{bertozzi2006inpainting}
	{\sc A.~L. Bertozzi, S.~Esedoglu, and A.~Gillette}, {\em Inpainting of binary
		images using the {C}ahn--{H}illiard equation}, IEEE Transactions on Image
	Processing, 16 (2006), pp.~285--291.
	
	\bibitem{binder1986kinetics}
	{\sc K.~Binder, H.~Frisch, and J.~J{\"a}ckle}, {\em Kinetics of phase
		separation in the presence of slowly relaxing structural variables}, The
	Journal of Chemical Physics, 85 (1986), pp.~1505--1512.
	
	\bibitem{brezis2010functional}
	{\sc H.~Brezis}, {\em Functional Analysis, Sobolev Spaces and Partial
		Differential Equations}, Springer, 2010.
	
	\bibitem{bru2003universal}
	{\sc A.~Br{\'u}, S.~Albertos, J.~L. Subiza, J.~L. Garc{\'\i}a-Asenjo, and
		I.~Br{\'u}}, {\em The universal dynamics of tumor growth}, Biophysical
	Journal, 85 (2003), pp.~2948--2961.
	
	\bibitem{bru2008fractal}
	{\sc A.~Br{\'u}, D.~Casero, S.~De~Franciscis, and M.~A. Herrero}, {\em Fractal
		analysis and tumour growth}, Mathematical and Computer Modelling, 47 (2008),
	pp.~546--559.
	
	\bibitem{burger2009cahn}
	{\sc M.~Burger, L.~He, and C.-B. Sch{\"o}nlieb}, {\em {C}ahn--{H}illiard
		inpainting and a generalization for grayvalue images}, SIAM Journal on
	Imaging Sciences, 2 (2009), pp.~1129--1167.
	
	\bibitem{cahn1958free}
	{\sc J.~W. Cahn and J.~E. Hilliard}, {\em Free energy of a nonuniform system:
		{I}. {I}nterfacial free energy}, The Journal of Chemical Physics, 28 (1958),
	pp.~258--267.
	
	\bibitem{caputo1967linear}
	{\sc M.~Caputo}, {\em Linear models of dissipation whose q is almost frequency
		independent -- {II}}, Geophysical Journal International, 13 (1967),
	pp.~529--539.
	
	\bibitem{caputo1999diffusion}
	\leavevmode\vrule height 2pt depth -1.6pt width 23pt, {\em Diffusion of fluids
		in porous media with memory}, Geothermics, 28 (1999), pp.~113--130.
	
	\bibitem{chen2001derivation}
	{\sc Z.~Chen, S.~L. Lyons, and G.~Qin}, {\em Derivation of the {F}orchheimer
		law via homogenization}, Transport in Porous Media, 44 (2001), pp.~325--335.
	
	\bibitem{cherfils2011cahn}
	{\sc L.~Cherfils, A.~Miranville, and S.~Zelik}, {\em The {C}ahn--{H}illiard
		equation with logarithmic potentials}, Milan Journal of Mathematics, 79
	(2011), pp.~561--596.
	
	\bibitem{choksi2009phase}
	{\sc R.~Choksi, M.~A. Peletier, and J.~Williams}, {\em On the phase diagram for
		microphase separation of diblock copolymers: {A}n approach via a nonlocal
		{C}ahn--{H}illiard functional}, SIAM Journal on Applied Mathematics, 69
	(2009), pp.~1712--1738.
	
	\bibitem{coleman1967equipresence}
	{\sc B.~D. Coleman and M.~E. Gurtin}, {\em Equipresence and constitutive
		equations for rigid heat conductors}, Zeitschrift f{\"u}r Angewandte
	Mathematik und Physik, 18 (1967), pp.~199--208.
	
	\bibitem{conti2010attractors}
	{\sc M.~Conti and M.~C. Zelati}, {\em Attractors for the {C}ahn--{H}illiard
		equation with memory in 2{D}}, Nonlinear Analysis: Theory, Methods \&
	Applications, 72 (2010), pp.~1668--1682.
	
	\bibitem{cuesta2011some}
	{\sc E.~Cuesta}, {\em Some advances on image processing by means of fractional
		calculus}, in Nonlinear Science and Complexity, J.~Machado et~al., eds.,
	Springer, 2011, pp.~265--271.
	
	\bibitem{dafermos1970asymptotic}
	{\sc C.~M. Dafermos}, {\em Asymptotic stability in viscoelasticity}, Archive
	for rational mechanics and analysis, 37 (1970), pp.~297--308.
	
	\bibitem{dai2016weak}
	{\sc S.~Dai and Q.~Du}, {\em Weak solutions for the {C}ahn--{H}illiard equation
		with degenerate mobility}, Archive for Rational Mechanics and Analysis, 219
	(2016), pp.~1161--1184.
	
	\bibitem{diestel1977vector}
	{\sc J.~Diestel and J.~Uhl}, {\em Vector Measures}, American Mathematical
	Society, 1977.
	
	\bibitem{diethelm2010analysis}
	{\sc K.~Diethelm}, {\em The Analysis of Fractional Differential Equations: An
		Application-Oriented Exposition using Differential Operators of Caputo Type},
	Springer, 2010.
	
	\bibitem{diethelm1999solution}
	{\sc K.~Diethelm and A.~D. Freed}, {\em On the solution of nonlinear
		fractional-order differential equations used in the modeling of
		viscoplasticity}, in Scientific Computing in Chemical Engineering II, F.~Keil
	et~al., eds., Springer, 1999, pp.~217--224.
	
	\bibitem{diethelm2020good}
	{\sc K.~Diethelm, R.~Garrappa, and M.~Stynes}, {\em Good (and not so good)
		practices in computational methods for fractional calculus}, Mathematics, 8
	(2020), p.~324.
	
	\bibitem{djilali2018galerkin}
	{\sc L.~Djilali and A.~Rougirel}, {\em Galerkin method for time fractional
		diffusion equations}, Journal of Elliptic and Parabolic Equations, 4 (2018),
	pp.~349--368.
	
	\bibitem{Dumitru12}
	{\sc B.~Dumitru, D.~Kai, and S.~Enrico}, {\em Fractional Calculus: Models and
		Numerical Methods}, World Scientific, 2012.
	
	\bibitem{eck2017mathematical}
	{\sc C.~Eck, H.~Garcke, and P.~Knabner}, {\em Mathematical Modeling}, Springer,
	2017.
	
	\bibitem{elliott1996cahn}
	{\sc C.~M. Elliott and H.~Garcke}, {\em On the {C}ahn--{H}illiard equation with
		degenerate mobility}, SIAM Journal on Mathematical Analysis, 27 (1996),
	pp.~404--423.
	
	\bibitem{elliott1993global}
	{\sc C.~M. Elliott and A.~Stuart}, {\em The global dynamics of discrete
		semilinear parabolic equations}, SIAM Journal on Numerical Analysis, 30
	(1993), pp.~1622--1663.
	
	\bibitem{evans2010partial}
	{\sc L.~C. Evans}, {\em Partial Differential Equations}, American Mathematical
	Society, 2010.
	
	\bibitem{fritz2020subdiffusive}
	{\sc M.~Fritz, C.~Kuttler, M.~L. Rajendran, L.~Scarabosio, and B.~Wohlmuth},
	{\em On a subdiffusive tumour growth model with fractional time derivative},
	preprint arXiv:2006.10670,  (2020).
	
	\bibitem{fritz2019local}
	{\sc M.~Fritz, E.~Lima, V.~Nikolic, J.~T. Oden, and B.~Wohlmuth}, {\em Local
		and nonlocal phase-field models of tumor growth and invasion due to {ECM}
		degradation}, Mathematical Models and Methods in Applied Sciences, 29 (2019),
	pp.~2433--2468.
	
	\bibitem{fritz2018unsteady}
	{\sc M.~Fritz, E.~Lima, J.~T. Oden, and B.~Wohlmuth}, {\em On the unsteady
		{D}arcy--{F}orchheimer--{B}rinkman equation in local and nonlocal tumor
		growth models}, Mathematical Models and Methods in Applied Sciences, 29
	(2019), pp.~1691--1731.
	
	\bibitem{galenko2005diffuse}
	{\sc P.~Galenko and D.~Jou}, {\em Diffuse-interface model for rapid phase
		transformations in nonequilibrium systems}, Physical Review E, 71 (2005),
	p.~046125.
	
	\bibitem{galenko2009kinetic}
	\leavevmode\vrule height 2pt depth -1.6pt width 23pt, {\em Kinetic contribution
		to the fast spinodal decomposition controlled by diffusion}, Physica A:
	Statistical Mechanics and its Applications, 388 (2009), pp.~3113--3123.
	
	\bibitem{galenko2007analysis}
	{\sc P.~Galenko and V.~Lebedev}, {\em Analysis of the dispersion relation in
		spinodal decomposition of a binary system}, Philosophical Magazine Letters,
	87 (2007), pp.~821--827.
	
	\bibitem{garcke2017well}
	{\sc H.~Garcke and K.~F. Lam}, {\em Well-posedness of a {C}ahn--{H}illiard
		system modelling tumour growth with chemotaxis and active transport},
	European Journal of Applied Mathematics, 28 (2017), pp.~284--316.
	
	\bibitem{garcke2018multiphase}
	{\sc H.~Garcke, K.~F. Lam, R.~N{\"u}rnberg, and E.~Sitka}, {\em A multiphase
		{C}ahn--{H}illiard--{D}arcy model for tumour growth with necrosis},
	Mathematical Models and Methods in Applied Sciences, 28 (2018), pp.~525--577.
	
	\bibitem{garcke2018optimal}
	{\sc H.~Garcke, K.~F. Lam, and E.~Rocca}, {\em Optimal control of treatment
		time in a diffuse interface model of tumor growth}, Applied Mathematics \&
	Optimization, 78 (2018), pp.~495--544.
	
	\bibitem{gatti2005memory}
	{\sc S.~Gatti, M.~Grasselli, A.~Miranville, and V.~Pata}, {\em Memory
		relaxation of first order evolution equations}, Nonlinearity, 18 (2005),
	p.~1859.
	
	\bibitem{gorenflo2002time}
	{\sc R.~Gorenflo, F.~Mainardi, D.~Moretti, and P.~Paradisi}, {\em Time
		fractional diffusion: {A} discrete random walk approach}, Nonlinear Dynamics,
	29 (2002), pp.~129--143.
	
	\bibitem{gripenberg1985volterra}
	{\sc G.~Gripenberg}, {\em Volterra integro-differential equations with
		accretive nonlinearity}, Journal of Differential Equations, 60 (1985),
	pp.~57--79.
	
	\bibitem{gripenberg1990volterra}
	{\sc G.~Gripenberg, S.~O. Londen, and O.~Staffans}, {\em Volterra Integral and
		Functional Equations}, Encyclopedia of Mathematics and its Applications,
	Cambridge University Press, 1990.
	
	\bibitem{guner2015variety}
	{\sc O.~G{\"u}ner, A.~Bekir, and A.~C. Cevikel}, {\em A variety of exact
		solutions for the time fractional {C}ahn--{A}llen equation}, The European
	Physical Journal Plus, 130 (2015), pp.~1--13.
	
	\bibitem{gurtin1996generalized}
	{\sc M.~E. Gurtin}, {\em Generalized {G}inzburg-{L}andau and {C}ahn-{H}illiard
		equations based on a microforce balance}, Physica D: Nonlinear Phenomena, 92
	(1996), pp.~178--192.
	
	\bibitem{hawkins2012numerical}
	{\sc A.~Hawkins-Daarud, K.~G. van~der Zee, and J.~T. Oden}, {\em Numerical
		simulation of a thermodynamically consistent four-species tumor growth
		model}, International Journal for Numerical Methods in Biomedical
	Engineering, 28 (2012), pp.~3--24.
	
	\bibitem{hilliard1970spinodal}
	{\sc J.~E. Hilliard}, {\em Spinodal decomposition}, Phase Transformation, 497
	(1970).
	
	\bibitem{inc2018time}
	{\sc M.~Inc, A.~Yusuf, A.~I. Aliyu, and D.~Baleanu}, {\em Time-fractional
		{C}ahn--{A}llen and time-fractional {K}lein--{G}ordon equations: {L}ie
		symmetry analysis, explicit solutions and convergence analysis}, Physica A:
	Statistical Mechanics and its Applications, 493 (2018), pp.~94--106.
	
	\bibitem{jakubowski2001nonlinear}
	{\sc V.~G. Jakubowski}, {\em Nonlinear elliptic-parabolic integro-differential
		equations with L1-data: existence, uniqueness, asymptotics}, PhD thesis,
	University of Essen, 2001.
	
	\bibitem{ji2020adaptive}
	{\sc B.~Ji, H.-l. Liao, Y.~Gong, and L.~Zhang}, {\em Adaptive linear
		second-order energy stable schemes for time-fractional {A}llen--{C}ahn
		equation with volume constraint}, Communications in Nonlinear Science and
	Numerical Simulation,  (2020), p.~105366.
	
	\bibitem{jiang2014anomalous}
	{\sc C.~Jiang, C.~Cui, L.~Li, and Y.~Shao}, {\em The anomalous diffusion of a
		tumor invading with different surrounding tissues}, PloS one, 9 (2014),
	p.~e109784.
	
	\bibitem{jost2014mathematical}
	{\sc J.~Jost}, {\em Mathematical Methods in Biology and Neurobiology},
	Springer, 2014.
	
	\bibitem{kemppainen2016decay}
	{\sc J.~Kemppainen, J.~Siljander, V.~Vergara, and R.~Zacher}, {\em Decay
		estimates for time-fractional and other non-local in time subdiffusion
		equations in $\mathbb{R}^d$}, Mathematische Annalen, 366 (2016),
	pp.~941--979.
	
	\bibitem{khristenko2021solving}
	{\sc U.~Khristenko and B.~Wohlmuth}, {\em Solving time-fractional differential
		equation via rational approximation}, arXiv:2102.05139,  (2021).
	
	\bibitem{kilbas2006theory}
	{\sc A.~A. Kilbas, H.~M. Srivastava, and J.~J. Trujillo}, {\em Theory and
		Applications of Fractional Differential Equations}, Elsevier, 2006.
	
	\bibitem{kruvzik2019mathematical}
	{\sc M.~Kru{\v{z}}{\'\i}k and T.~Roub{\'\i}{\v{c}}ek}, {\em Mathematical
		Methods in Continuum Mechanics of Solids}, Springer, 2019.
	
	\bibitem{lecoq2009evolution}
	{\sc N.~Lecoq, H.~Zapolsky, and P.~Galenko}, {\em Evolution of the structure
		factor in a hyperbolic model of spinodal decomposition}, The European
	Physical Journal Special Topics, 177 (2009), p.~165.
	
	\bibitem{lecoq2011numerical}
	{\sc N.~Lecoq, H.~Zapolsky, and P.~Galenko}, {\em Numerical approximation of
		the {C}ahn--{H}illiard equation with memory effects in the dynamics of phase
		separation}, Discrete and continuous dynamical systems A, 31 (2011),
	pp.~953--962.
	
	\bibitem{li2018generalized}
	{\sc L.~Li and J.-G. Liu}, {\em A generalized definition of {C}aputo
		derivatives and its application to fractional odes}, SIAM Journal on
	Mathematical Analysis, 50 (2018), pp.~2867--2900.
	
	\bibitem{li2018some}
	\leavevmode\vrule height 2pt depth -1.6pt width 23pt, {\em Some compactness
		criteria for weak solutions of time fractional {PDE}s}, SIAM Journal on
	Mathematical Analysis, 50 (2018), pp.~3963--3995.
	
	\bibitem{li2017space}
	{\sc Z.~Li, H.~Wang, and D.~Yang}, {\em A space--time fractional phase-field
		model with tunable sharpness and decay behavior and its efficient numerical
		simulation}, Journal of Computational Physics, 347 (2017), pp.~20--38.
	
	\bibitem{lions1969some}
	{\sc J.~L. Lions}, {\em Quelques Méthodes de Résolution des Problemes aux
		Limites Non Linéaires}, Dunod, 1969.
	
	\bibitem{liu2020fast}
	{\sc H.~Liu, A.~Cheng, and H.~Wang}, {\em A fast {G}alerkin finite element
		method for a space--time fractional {A}llen--{C}ahn equation}, Journal of
	Computational and Applied Mathematics, 368 (2020), p.~112482.
	
	\bibitem{liu2018time}
	{\sc H.~Liu, A.~Cheng, H.~Wang, and J.~Zhao}, {\em Time-fractional
		{A}llen--{C}ahn and {C}ahn--{H}illiard phase-field models and their numerical
		investigation}, Computers \& Mathematics with Applications, 76 (2018),
	pp.~1876--1892.
	
	\bibitem{Lubich86}
	{\sc C.~Lubich}, {\em Discretized fractional calculus}, SIAM Journal on
	Mathematical Analysis, 17 (1986), pp.~704--719.
	
	\bibitem{Lubich88}
	{\sc C.~Lubich}, {\em Convolution quadrature and discretized operational
		calculus}, Numerische Mathematik, 52 (1988), pp.~129--145.
	
	\bibitem{mainardi2010fractional}
	{\sc F.~Mainardi}, {\em Fractional Calculus and Waves in Linear
		Viscoelasticity: An Introduction to Mathematical Models}, World Scientific,
	2010.
	
	\bibitem{marks1981differintegral}
	{\sc R.~Marks and M.~Hall}, {\em Differintegral interpolation from a
		bandlimited signal's samples}, IEEE Transactions on Acoustics, Speech, and
	Signal Processing, 29 (1981), pp.~872--877.
	
	\bibitem{mustapha2014well}
	{\sc K.~Mustapha and D.~Sch{\"o}tzau}, {\em Well-posedness of $hp$-version
		discontinuous {G}alerkin methods for fractional diffusion wave equations},
	IMA Journal of Numerical Analysis, 34 (2014), pp.~1426--1446.
	
	\bibitem{novick2002phase}
	{\sc A.~Novick-Cohen}, {\em A phase field system with memory: {G}lobal
		existence}, The Journal of Integral Equations and Applications,  (2002),
	pp.~73--107.
	
	\bibitem{ohta1986equilibrium}
	{\sc T.~Ohta and K.~Kawasaki}, {\em Equilibrium morphology of block copolymer
		melts}, Macromolecules, 19 (1986), pp.~2621--2632.
	
	\bibitem{ouedjedi2019galerkin}
	{\sc Y.~Ouedjedi, A.~Rougirel, and K.~Benmeriem}, {\em Galerkin method for time
		fractional semilinear equations}, Preprint, HAL-02124150,  (2019).
	
	\bibitem{povstenko2015fractional}
	{\sc Y.~Povstenko}, {\em Fractional Thermoelasticity}, Springer, 2015.
	
	\bibitem{povstenko2017two}
	{\sc Y.~Povstenko and T.~Kyrylych}, {\em Two approaches to obtaining the
		space-time fractional advection-diffusion equation}, Entropy, 19 (2017),
	p.~297.
	
	\bibitem{prakasha2019two}
	{\sc D.~Prakasha, P.~Veeresha, and H.~M. Baskonus}, {\em Two novel
		computational techniques for fractional gardner and {C}ahn--{H}illiard
		equations}, Computational and Mathematical Methods, 1 (2019), p.~e1021.
	
	\bibitem{pruss2013evolutionary}
	{\sc J.~Pr{\"u}ss}, {\em Evolutionary Integral Equations and Applications},
	Birkh{\"a}user, 2013.
	
	\bibitem{pruss2010well}
	{\sc J.~Pr{\"u}ss, V.~Vergara, and R.~Zacher}, {\em Well-posedness and
		long-time behaviour for the non-isothermal {C}ahn--{H}illiard equation with
		memory}, Discrete \& Continuous Dynamical Systems-A, 26 (2010), p.~625.
	
	\bibitem{robinson2001infinite}
	{\sc J.~C. Robinson}, {\em Infinite-Dimensional Dynamical Systems: An
		Introduction to Dissipative Parabolic {PDE}s and the Theory of Global
		Attractors}, Cambridge University Press, 2001.
	
	\bibitem{rotstein2001phase}
	{\sc H.~G. Rotstein, S.~Brandon, A.~Novick-Cohen, and A.~Nepomnyashchy}, {\em
		Phase field equations with memory: {T}he hyperbolic case}, SIAM Journal on
	Applied Mathematics,  (2001), pp.~264--282.
	
	\bibitem{roubicek}
	{\sc T.~Roub{\'\i}{\v{c}}ek}, {\em Nonlinear {P}artial {D}ifferential
		{E}quations with {A}pplications}, Springer, 2013.
	
	\bibitem{rubinstein1992nonlocal}
	{\sc J.~Rubinstein and P.~Sternberg}, {\em Nonlocal reaction-diffusion
		equations and nucleation}, IMA Journal of Applied Mathematics, 48 (1992),
	pp.~249--264.
	
	\bibitem{runst2011sobolev}
	{\sc T.~Runst and W.~Sickel}, {\em Sobolev Spaces of Fractional Order,
		Nemytskij Operators, and Nonlinear Partial Differential Equations}, De
	Gruyter, 2011.
	
	\bibitem{saltelli2008global}
	{\sc A.~Saltelli, M.~Ratto, T.~Andres, F.~Campolongo, J.~Cariboni, D.~Gatelli,
		M.~Saisana, and S.~Tarantola}, {\em Global Sensitivity Analysis: {T}he
		Primer}, John Wiley \& Sons, 2008.
	
	\bibitem{scarpa2020stochastic}
	{\sc L.~Scarpa}, {\em The stochastic {C}ahn--{H}illiard equation with
		degenerate mobility and logarithmic potential}, arXiv preprint
	arXiv:1909.12106,  (2019).
	
	\bibitem{simon1986compact}
	{\sc J.~Simon}, {\em Compact sets in the space ${L}^p({0},{T};{B})$}, Annali di
	Matematica Pura ed Applicata, 146 (1986), pp.~65--96.
	
	\bibitem{sobol2001global}
	{\sc I.~M. Sobol}, {\em Global sensitivity indices for nonlinear mathematical
		models and their {M}onte--{C}arlo estimates}, Mathematics and Computers in
	Simulation, 55 (2001), pp.~271--280.
	
	\bibitem{tang2019energy}
	{\sc T.~Tang, H.~Yu, and T.~Zhou}, {\em On energy dissipation theory and
		numerical stability for time-fractional phase-field equations}, SIAM Journal
	on Scientific Computing, 41 (2019), pp.~A3757--A3778.
	
	\bibitem{tarasov2016chain}
	{\sc V.~E. Tarasov}, {\em On chain rule for fractional derivatives},
	Communications in Nonlinear Science and Numerical Simulation, 30 (2016),
	pp.~1--4.
	
	\bibitem{taylor1994linking}
	{\sc J.~E. Taylor and J.~W. Cahn}, {\em Linking anisotropic sharp and diffuse
		surface motion laws via gradient flows}, Journal of Statistical Physics, 77
	(1994), pp.~183--197.
	
	\bibitem{temam2012infinite}
	{\sc R.~Temam}, {\em Infinite-Dimensional Dynamical Systems in Mechanics and
		Physics}, Springer, 2012.
	
	\bibitem{torvik1984appearance}
	{\sc P.~J. Torvik and R.~L. Bagley}, {\em On the appearance of the fractional
		derivative in the behavior of real materials}, Journal of Applied Mechanics,
	51 (1984), pp.~294--298.
	
	\bibitem{vergara2007conserved}
	{\sc V.~Vergara}, {\em A conserved phase field system with memory and relaxed
		chemical potential}, Journal of Mathematical Analysis and Applications, 328
	(2007), pp.~789--812.
	
	\bibitem{vergara2007maximal}
	\leavevmode\vrule height 2pt depth -1.6pt width 23pt, {\em Maximal regularity
		and global well-posedness for a phase field system with memory}, The Journal
	of Integral Equations and Applications,  (2007), pp.~93--115.
	
	\bibitem{vergara2008lyapunov}
	{\sc V.~Vergara and R.~Zacher}, {\em Lyapunov functions and convergence to
		steady state for differential equations of fractional order}, Mathematische
	Zeitschrift, 259 (2008), pp.~287--309.
	
	\bibitem{wittbold2020bounded}
	{\sc P.~Wittbold, P.~Wolejko, and R.~Zacher}, {\em Bounded weak solutions of
		time-fractional porous medium type and more general nonlinear and degenerate
		evolutionary integro-differential equations}, arXiv preprint
	arXiv:2008.10919,  (2020).
	
	\bibitem{ye2007generalized}
	{\sc H.~Ye, J.~Gao, and Y.~Ding}, {\em A generalized {G}ronwall inequality and
		its application to a fractional differential equation}, Journal of
	Mathematical Analysis and Applications, 328 (2007), pp.~1075--1081.
	
	\bibitem{zacher2008boundedness}
	{\sc R.~Zacher}, {\em Boundedness of weak solutions to evolutionary partial
		integro-differential equations with discontinuous coefficients}, Journal of
	mathematical analysis and applications, 348 (2008), pp.~137--149.
	
	\bibitem{zacher2009weak}
	\leavevmode\vrule height 2pt depth -1.6pt width 23pt, {\em Weak solutions of
		abstract evolutionary integro-differential equations in {H}ilbert spaces},
	Funkcialaj Ekvacioj, 52 (2009), pp.~1--18.
	
	\bibitem{zacher2012global}
	\leavevmode\vrule height 2pt depth -1.6pt width 23pt, {\em Global strong
		solvability of a quasilinear subdiffusion problem}, Journal of Evolution
	Equations, 12 (2012), pp.~813--831.
	
	\bibitem{zhang2020non}
	{\sc J.~Zhang, J.~Zhao, and J.~Wang}, {\em A non-uniform time-stepping convex
		splitting scheme for the time-fractional {C}ahn--{H}illiard equation},
	Computers \& Mathematics with Applications, 80 (2020), pp.~837--850.
	
	\bibitem{zhao2019power}
	{\sc J.~Zhao, L.~Chen, and H.~Wang}, {\em On power law scaling dynamics for
		time-fractional phase field models during coarsening}, Communications in
	Nonlinear Science and Numerical Simulation, 70 (2019), pp.~257--270.
	
\end{thebibliography}

\end{document}